\documentclass[12pt]{amsart}
\usepackage[utf8]{inputenc}
\usepackage{color}
\usepackage{amssymb}
\usepackage{mathabx}
\newtheorem{theorem}{Theorem}[section]
\newtheorem{lemma}[theorem]{Lemma}
\newtheorem{proposition}[theorem]{Proposition}
\newtheorem{corollary}{Corollary}[section]
\theoremstyle{definition}
\newtheorem{definition}{Definition}[section]
\theoremstyle{remark}
\newtheorem{remark}{Remark}[section]

\numberwithin{equation}{section}

\newcommand{\R}{\mathbb R}
\newcommand{\Z}{\mathbb Z}

\begin{document}
\title[Topology of QC spaces]{On the topology of closed manifolds 
with quasi-constant sectional curvature}
\author{Louis Funar}
 \date{\today}
\address{Institut Fourier,
Laboratoire de Mathematiques UMR 5582,   
Universit\'e Grenoble Alpes,
CS 40700, 
38058 Grenoble, France}
\email{louis.funar@univ-grenoble-alpes.fr}

{\abstract
We prove that closed manifolds admitting a  generic metric (see Definition \ref{generic})
whose sectional curvature is locally quasi-constant 
are graphs of space forms. In the more general setting of QC spaces where 
sets of isotropic points are arbitrary, under suitable 
positivity assumption  and for torsion-free fundamental groups 
they are still diffeomorphic to connected sums of 
spherical space forms and spherical bundles over the circle. \\

\noindent 
AMS Math. Subj.Classification:  53C21, 53C23, 53C25, 57R42.
}
\vspace{1cm}

\maketitle

\vspace{1,5cm}

\section{Introduction}
The classification of (locally) conformally flat closed manifolds
seems out of reach in full generality since M. Kapovich proved (see \cite{Ka}) 
that an arbitrary finitely presented group  $G$ is a subgroup 
of a free amalgamated product of the form 
$G * H$  which is a  fundamental group of a conformally
flat closed manifold of dimension (at least) 4. 

On the other hand, simple infinite groups (e.g. Thompson's groups) do not occur 
among fundamental groups of conformally flat closed manifolds.  Indeed the holonomy 
representation would be injective, so the fundamental group has to be linear, but finitely generated 
linear groups are residually finite, according to classical results of Malcev, and hence they cannot be infinite and simple.    
Kamishima (\cite{Kam}), improving previous work by  
Goldman (\cite{Go})  for virtually nilpotent groups, showed that a closed conformally flat 
$n$-manifold with virtually solvable
fundamental group is conformally equivalent to a 
$n$-spherical space form,  an Euclidean space-form or it is 
finitely covered by some Hopf manifold  
$S^1\times S^{n-1}$. Further Goldman and Kamishima in 
\cite{Go-Kam} classified up to a finite covering the conformally flat closed manifolds whose universal covering
admits a complete conformal Killing 
vector field, by adding to the previous list the closed 
hyperbolic manifolds,  their products with a circle 
and conformally homogeneous quotients  of the form 
$(S^n-\lim \Gamma)/\Gamma$, where 
$\lim \Gamma\subset S^{n-2}$ is the limit set of the conformal 
holonomy group $\Gamma$. 

The aim of this paper is to describe the 
topology of a family of conformally flat 
manifolds admitting a  local vector field with respect to which 
the curvature is quasi-constant. This hypothesis is similar in spirit  
to the one used by Goldman and Kamishima (\cite{Go-Kam}), but in our case 
the vector field is not a conformal Killing field, in general,
and the methods of the proof are rather different coming from foliations and 
Cheeger-Gromov's theory.

The first family of spaces under consideration here are manifolds admitting a Riemannian 
metric of quasi-constant sectional curvature of dimension $n\geq 3$ (see the definition below).
These spaces were first studied from a geometrical point of view 
by Boju and Popescu in \cite{BP} and coincide with the 
$k$-special conformally flat spaces considered earlier in a different context 
by B.-Y.Chen and K.Yano (\cite{H,CY}). These were further analyzed by Ganchev and Mihova in \cite{GM}. 
It was noticed in \cite{BF,H} and later 
in \cite{GM} that these are  locally conformally flat manifolds for $n\geq 4$.

\begin{definition}
Let  $\xi$ be a smooth line field on the Riemannian manifold $M$.
We denote by $K(\sigma)$ the sectional curvature 
of the tangent 2-plane $\sigma$ at $p\in M$ and by 
$\measuredangle(\sigma,\xi)$ the angle between $\sigma$ and the line $\xi|_p$.
  
The metric of $M$ is {\em globally $1$-QC} with 
respect to the line field $\xi$ if there exists some function  
$\theta: M\to  (0,\frac{\pi}{2})$ such that for any 
point $p\in M$, and any 2-planes $\sigma_1,\sigma_2$ tangent at $p$ to $M$ verifying   
$ \measuredangle(\sigma_1,\xi)=\measuredangle(\sigma_2,\xi)=\theta(p)$ 
we have:
\[ K(\sigma_1)=K(\sigma_2).\]  
In this case $\xi$ is called a {\em distinguished} line field of the Riemannian manifold $M$. 
Further we say that (the metric of) $M$ is {\em globally} $1$-QC if it is {\em globally $1$-QC} 
with respect to some smooth line field $\xi$.  
\end{definition}

The original definition in \cite{BP,GM} (see also the 
comments of the last section in \cite{BP})  used a vector field instead of a line field and is 
too restrictive for our purposes. 

It is proved in \cite{BP} that the sectional curvature of a globally $1$-QC manifold is constant  along any cone of 2-planes making a constant angle in  $[0,\frac{\pi}{2}]$ with $\xi$, 
in every point. 

Recall that a {\em line} field is a section of the projectivized tangent bundle. 
A line field $\xi$ on $M$ determines a real line bundle, still denoted $\xi$ 
over $M$.  Its first Stiefel-Whitney class $w_1(\xi)\in H^1(M,\Z/2\Z)$ vanishes 
if and only if $\xi$ has a section, namely a  non-zero vector field $\widehat{\xi}$ 
defining the same line field. When this is the case, the line field will be called {\em orientable} and such a vector field 
$\widehat{\xi}$ will be called a {\em lift} of the line field. 
When $\xi$ is non-orientable, its pull-back to the 2-fold cover $\widehat{M}$ of $M$ associated to 
the class $w_1(\xi)$ is orientable and defines a vector field $\widehat{\xi}$ on $M$. 
We call $\widehat{\xi}$ the distinguished vector field associated to the globally $1$-QC manifold 
with distinguished line field $\xi$. 

Throughout this paper  we will use the term {\em cylinder} in its topological acception, 
namely a cylinder is a manifold diffeomorphic to the product of a compact manifold with an interval (which might be open or closed).
More generally, a {\em twisted cylinder} is an interval bundle over a compact manifold, which might possibly be a non-trivial bundle.  

The topology of globally $1$-QC manifolds is rather 
simple, as we will see that they must be  2-fold covered by cylinders.
Further the geometry of globally $1$-QC manifolds  with orientable line field was completely described 
under an additional condition (namely, that the distinguished field 
be geodesic) in (\cite{GM}, section 6): 
these are precisely the warped product Riemannian manifolds 
between some closed space form  and an interval which are called   
sub-projective spaces. Notice that the vector field $\xi$ is (locally) conformally Killing if and only if 
it is Killing and iff  $\xi$ is geodesic.

In order to enlarge the number of possible topologies arising in the definition above we slightly weaken it as follows:

\begin{definition}
The (metric of the) Riemannian manifold $M$ is  called 
\begin{enumerate}
\item {\em locally  $1$-QC},  
if any point of $M$ has an open neighborhood endowed with some line field $\xi$ such that 
the induced metric is globally $1$-QC with respect to $\xi$.  
\item {\em $1$-QC} if any {\em non-isotropic} point for the sectional curvature 
function has an open neighborhood  which is locally $1$-QC.    
\end{enumerate}
\end{definition} 
Note that line fields do not necessarily define a global line field on the manifold $M$.   
Observe that a locally $1$-QC manifold is a $1$-QC manifold.

Recall that a point of a Riemannian manifold  is called {\em isotropic} 
if all its tangent 2-planes have the same sectional curvature. 
Denote by $C$  the set of  isotropic points of $M$.  
In the case of a globally $1$-QC manifold the sectional curvatures $N$ and $H$  
of 2-planes containing  and respectively orthogonal to 
the  distinguished line field $\xi$ will be called {\em vertical} and {\em horizontal} 
curvature functions, respectively. 

\begin{remark}
An open set of dimension at least $3$  consisting of isotropic points has constant sectional 
curvature, by F. Schur's theorem (\cite{Sch}). 
\end{remark}

It will be shown in the next section that both $H$ and $N$ extend smoothly from the globally $1$-QC 
neighborhoods to all of $M$, when the manifold $M$ is $1$-QC.  Moreover, the  set of non-isotropic points 
carries an unique distinguished line field extending the distinguished line fields associated to various 
locally $1$-QC neighborhoods.  

For a $1$-QC manifold $M$ the set  $C$ of its isotropic points 
then consists of those points $p\in M$  satisfying the equation: 
\[ H(p)=N(p).\]

\begin{definition}\label{generic}
The metric on a $1$-QC manifold is {\em generic} if the horizontal curvature function $H$ 
is not eventually constant on $M-C$,  namely  for every 
connected open subset $U\subset M-C$ such that the restriction of $H$ to  $U$ is constant
we have $\overline{U}\cap C=\emptyset$.  
\end{definition}

Our main issue in this paper is to give some insights on the topology of $C$ 
and  $\overline{M-C}$. Results of foliation theory give the description of $M-C$  as twisted cylinders. 
Further, we will prove  that $\overline{M-C}$ 
and $\overline{{\rm int}(C)}$ can be approximated by manifolds with boundary  when 
$M$ is locally $1$-QC.

\begin{definition}
A {\em space form} is a compact Riemannian manifold with constant sectional curvature. 
 
A {\em graph of space forms} is a closed manifold $M$ which can be obtained 
from a finite set of compact space forms with totally umbilical 
boundary components by gluing isometrically 
interval bundles (twisted cylinders) along boundary components. 
The pieces will be accordingly called {\em vertex} manifolds and {\em edge} 
manifolds or {\em tubes}. 
\end{definition}

Note that the induced metric on the graph of space forms is not necessarily smooth and its 
diffeomorphism type depends on additional choices, like the gluing maps. 
In particular, fiber bundles over the circle whose fibers are closed space forms 
are such graphs of space forms.

The structure of conformally flat closed manifolds whose universal covering 
have a conformal vector field was described by Goldman and Kamishima (see \cite{Go-Kam}).
This includes the case of globally $1$-QC manifolds with geodesic distinguished field $\xi$. 
Motivated by this, our first result is a  general description of the 
topology of the locally $1$-QC spaces, under a mild genericity condition, as follows:

\begin{theorem}\label{arbitrary}
Let $M^n$ be a closed locally  $1$-QC $n$-manifold, $n\geq 3$, whose metric is  generic and   
assumed conformally flat, when $n=3$.  Assume that the rank of $H_1(M-C;\mathbb Q)$ is finite. 
Then $M$ is diffeomorphic to a graph of space forms. 
\end{theorem}

\begin{remark}
The same result holds under the assumption that $H_1(M-C;\Z/2\Z)$ is finite. 
\end{remark}

The {\em curvature leaves function} $\lambda$ is defined at non-isotropic points 
(see also the section \ref{leafcurv}) by: 
\[
\lambda=H + \frac{\parallel {\rm grad} H\parallel^2}{4(H-N)^2}
\]

Our second result aims to extend Theorem \ref{arbitrary} to the case of $1$-QC manifolds under additional hypotheses.

\begin{theorem}\label{lambdapositive}
Let $M^n$,  $n\geq 3$, be a closed  $n$-manifold with infinite torsion-free 
$\pi_1(M)$, admitting a $1$-QC metric with orientable distinguished line field, which is conformally flat when $n=3$.
Moreover, suppose  $\lambda >0$ on $M-C$ and that every tangent $2$-plane at an isotropic point has 
positive sectional curvature.
Then $M$ is diffeomorphic to a connected sum of  $S^{n-1}$-bundles over $S^1$ (possibly not orientable).
\end{theorem}
By convention a connected  sum of an empty set of $n$-manifolds is the $n$-sphere. 

The main issue is the description of degeneracies occurring  when compactifying $M-C$. One key ingredient  is the characterization of compact $1$-QC manifolds as those Riemannian manifolds whose universal coverings have codimension one isometric immersions into hyperbolic spaces (see 
Propositions \ref{immerse} and \ref{characterize}). The topology of closed manifolds  admitting such 
codimension one immersions is known, as this problem already appeared in conformal geometry. Kulkarni  has proved in \cite{Kul}  that closed orientable 
conformally flat manifolds  which admit conformal embeddings as hypersurfaces in 
$\mathbb R^{n+1}$ are conformally equivalent to  either the round $n$-sphere or 
some Hopf manifold $S^1\times S^{n-1}$, if  both 
the immersion and the metric are {\em analytic}. Moreover, in \cite{DDM,Pi} the authors proved that closed conformally flat manifolds  which admit conformal embeddings as hypersurfaces in some 
hyperbolic of Euclidean space are diffeomorphic and respectively 
conformally equivalent to an $n$-sphere with several (possibly unoriented) 
$1$-handles, i.e. connected sums of $S^{n-1}$-bundles the circle. We use the methods from \cite{DDM,Pi} suitably extended 
to our non-compact situation, in order to analyze $\overline{M-C}$.  We will show that a suitable neighborhood 
of $\overline{{\rm int}(C)}$ in $M$ can be obtained from a closed space form by removing finitely many 
disjoint disks. To this purpose we prove that one can cap off its boundary spheres by using spherical caps with controlled geometry so that the manifold obtained this way is diffeomorphic to a space form, according to a theorem of Nikolaev (\cite{Nik}).

\begin{remark}
A classical result of Kuiper states that a closed simply connected conformally flat manifold is conformally equivalent 
to the round sphere (see \cite{Ku1}). Kuiper worked under $\mathcal C^3$ differentiability assumptions, but 
this result is known to hold under the weaker $\mathcal C^1$-assumptions. 
If  $M$ is a closed  $1$-QC manifold with finite $\pi_1(M)$, then its universal covering $\widetilde{M}$  
is therefore conformally equivalent to the round sphere. 
\end{remark}

\begin{remark}
If $H, N >0$, then   
classical results of Bochner, Lichnerowicz and Myers show 
that the manifold is a homology sphere with finite 
fundamental group. In particular manifolds with a positive number of 1-handles in theorem \ref{lambdapositive} 
cannot have metrics with positive $N$. 
\end{remark}
\begin{remark}
If $(n-2)H + 2N \geq 0$ then $M^n$ is conformally equivalent to a 
Kleinian quotient (see \cite{SY}, Thm. 4.5). Moreover, if the inequality above is strict, then  
$\pi_j(M)=0$, for  $2\leq j\leq n/2$ (see \cite{SY}, Thm. 4.6).
\end{remark}

\begin{remark}
Rational Pontryagin  forms, and hence classes,  
vanish on  conformally flat manifolds, see \cite{CS,Kul2}.
Conversely, it seems unknown whether conformally flat 
manifolds admit conformally flat metrics for which the 
curvature tensor in suitable basis has all its components with 
3 and 4 distinct indices vanishing.
\end{remark}

Throughout this paper (local) $1$-QC manifold or space 
will mean a manifold which admits a Riemannian 
(local) $1$-QC metric.

\tableofcontents

The structure of this paper is as follows. In section 2 we collect a number of general results concerning 
1-QC manifolds for further use. We show that the horizontal and vertical curvature functions extend smoothly 
to the entire manifold first in the locally 1-QC case (Lemma  \ref{local-isotrop}), then using the higher Weitzenb\"ock  curvatures in the general 
1-QC case (Lemma  \ref{continuous}). In dimensions $n\geq 4$ 1-QC manifolds are  conformally flat, while this has to be added to our assumptions when $n=3$ (see Lemma \ref{conformal}). 
The level hypersurfaces of the horizontal curvature define a partial foliation with complete leaves which are 
totally umbilical submanifolds of intrinsic constant curvature $\lambda$. We show that for locally $1$-QC manifolds 
$\lambda$ extends to the closure of the open set of non-isotropic points (Lemma  \ref{extension-lambda}). 
Eventually, we show that the 1-QC manifolds of dimension $n\geq 4$  are precisely those Riemannian 
manifolds whose universal coverings have codimension one isometric immersions into the hyperbolic space, which are 
equivariant with respect to some holonomy homomorphism (Propositions  \ref{immerse}-\ref{characterize}).

Section 3 is devoted to the proof of Theorem \ref{arbitrary}. 
We start by describing the topology of globally 1-QC manifolds with non-empty set of isotropic points: 
they are twisted cylinders over compact space space forms (Proposition  \ref{dicho}), since the curvature leaves 
are compact (Proposition \ref{compactleaf}). This is a consequence of deep results of Reckziegel (\cite{Re2,Re}) concerning the completeness 
of curvature leaves and Haefliger's classical theory of foliations with compact leaves (\cite{Hae}). 
In section \ref{Uniformgeom} we collect a few results to be used later about the geometry of the  curvature leaves, 
by showing that their second fundamental form,  injectivity radius and (local) normal injectivity are uniformly bounded.  
The existence of a well-defined distinguished vector field $\widehat{\xi}$ in a neighborhood of an isotropic point shows that 
curvature leaves have a controlled behavior when they approach the isotropic point, by limiting to some branched 
hypersurface (Proposition  \ref{integrability}).  
We  further need to analyze the degeneracies occurring in the closure of these components by using the 
branched hypersurfaces which arise as limit leaves (Proposition \ref{bdy}). We show that we can choose finitely many limit leaves 
such that the  open isotropic components can be compactified to manifolds with boundary (Propositions 
\ref{cylinderneighb} and \ref{boundary}). By gluing isometrically some 
cylinders endowed with a warped product metric  along the boundary we construct space forms with umbilical boundary. 

In section 4 we are concerned with the proof of Theorem \ref{lambdapositive}.  The distinguished line field 
cannot be extended to isotropic points as in the previous chapter but the curvature leaves are now constrained to be 
diffeomorphic to spheres, as they have positive curvature. 
Every connected component of a saturated neighborhood of the set of isotropic points is a manifold whose boundary 
components are spheres (Proposition  \ref{compact}).
Group theory arguments involving Gruschko's theorem prove that there are only finitely many 
spheres which do not bound a ball (Proposition  \ref{essential}).  The intrinsic curvature of a spherical leaf provides control on the 
size of a spherical cap in the hyperbolic space which bounds that leaf (Lemmas  \ref{caps} and  \ref{cap}).   We then prove that 
we can realize the capping off  the boundary spheres  in the Riemannian setting and the closed manifold obtained 
this way is diffeomorphic to a space form, according to a theorem of Nikolaev (\cite{Nik}) (see Proposition  \ref{isotrop}). 

In chapter 5 we complete the description of  isometric immersions of $1$-QC conformally flat manifolds with  the 
case of dimension $n=3$ (Proposition  \ref{immerse3}).  Further, we note that a $1$-QC metric admits an extension to a 
hyperbolic metric on the product $M\times [0,\infty)$ (Proposition \ref{hyperbolicmetric}).

\vspace{0.2cm}

{\bf Acknowledgements}. The author is indebted to Gerard Besson, 
Zindine Djadli, Bill Goldman,  Luca Rizzi, Vlad Sergiescu, Ser Peow Tan and Ghani Zeghib for 
useful discussions and to the referees for 
helpful suggestions and pointing out some incomplete arguments.  
The author was partially supported by the ANR 2011 BS 01 020 01 ModGroup and 
he thanks the Erwin Schr\"odinger Institute for hospitality and 
support.

\section{Preliminaries}

\subsection{Isotropic points of  local $1$-QC manifolds}\label{section-isotropic}

\begin{lemma}\label{isotropy}
If $M$ is locally $1$-QC with respect to distinct distinguished line fields at $p$ then 
$p$ is an isotropic point. 
\end{lemma}
\begin{proof}
Let $X$ and $Y$ be unit vectors representing the distinguished lines fields at $p$. 
According to \cite{BP} the curvature of a 2-plane $\sigma$ is given by 
\[ K(\sigma)=H_X \sin^2\measuredangle (X,\sigma) + N_X \cos^2\measuredangle (X,\sigma)=
H_Y \sin^2\measuredangle (Y,\sigma) + N_Y \cos^2\measuredangle (Y,\sigma),\]
where $H_X, N_X$ and $H_Y, N_Y$ are the vertical and horizontal curvatures associated to the respective lines fields. 

Take $\sigma$ to be the span of the orthogonal unit vectors $X$ and $Z$. If $g$ denotes the metric on $M$, then:
\[K(\sigma)=H_Y + (g(Y,X)^2+ g(Y, Z)^2)(N_Y-H_Y)=H_X.\]
If $H_Y=N_Y$ then $p$ is isotropic. Otherwise $g(Y, Z)^2$ should not depend on the choice of  
$Z$ orthogonal to $X$ and hence $g(Y, Z)=0$, so that the lines defined by $Y$ and $X$ coincide, contradicting our hypothesis.  
\end{proof}

\begin{lemma}\label{local-isotrop}
If $M$ is a compact locally $1$-QC manifold then the functions $H$ and $N$ are well-defined smooth functions on $M$. 
\end{lemma}
\begin{proof}  
For any $p\in M$ there is a line field $\xi_U$ defined on an open neighborhood $U$ of $p$ 
such that the metric is globally $1$-QC on $U$ relative to $\xi_U$. This provides   
the smooth curvature functions $H_U$ and $N_U$ on $U$. 
We extract a finite covering  $\{U_i\}$ of $M$ by such open sets. 

If $U_i\cap U_j\neq \emptyset$ we set 
\[C_{U_iU_j}=\{p\in U_i\cap U_j; \xi_{U_i}|_p \neq  \xi_{U_j}|_p\}\]
Then, by Lemma \ref{isotropy} every point $p\in C_{U_iU_j}$ 
is isotropic for the sectional curvature. This implies that 
\[ H_{U_i}(p)=N_{U_i}(p)=H_{U_j}(p)=N_{U_j}(p).\]

Further, if $p\in U_i\cap U_j\setminus C_{U_iU_j}$, then $\xi_{U_i}|_p = \xi_{U_j}|_p$ and thus 
$H_{U_i}(p)=H_{U_j}(p)$, $N_{U_i}(p)=N_{U_j}(p)$. 

Therefore we have a finite collection of smooth functions $H_{U_i}$ and $N_{U_i}$ whose restrictions 
to intersections  $U_i\cap U_j$ coincide. Thus they define smooth functions on $M$. 
\end{proof}

Recall that $C$ denotes the closed subset of $M$ of isotropic points for the sectional curvature.

\begin{lemma}\label{local-line}
If $M$ is a compact locally $1$-QC manifold then the distinguished line field $\xi$ 
is well-defined on  an open set $\Omega \supset \overline{M-C}$. 
\end{lemma}
\begin{proof}
Consider a finite open cover of $M$ by globally 1-QC subsets $\{U_i\}$, as above. 
We take a refined open cover $\{V_j\}$ such that  every $\overline{V_j}$ is contained into some $U_i$.  
If $p\in (V_i\cap V_j)\setminus C$, then $\xi_{V_i}|_p=\xi_{V_j}|_p$, by  Lemma \ref{isotropy}. 
This closed condition is therefore valid for any point $p\in \overline{(V_i\cap V_j)\setminus C}$. 
It follows that the line fields $\xi_{V_i}$ agree at any point $p\in \overline{M-C}$. They define therefore 
a distinguished line field on $\overline{M-C}$. 

When the line field is orientable, it can be extended to an open set $\Omega$ containing  $\overline{M-C}$, by a 
classical partition of unity argument. Otherwise, we construct an extension on the 2-fold cover of an open neighborhood 
of  $\overline{M-C}$ associated to the first Stiefel-Whitney class of $\xi$ and then project it down. 
 \end{proof}

\subsection{Weitzenb\"ock curvatures} 
The $m$-th Weitzenb\"ock curvatures were explicitly computed 
for manifolds with vanishing Weyl tensor by Guan and Wang (\cite{GW}), as follows. 

Let $\Delta$ be the Hodge-de Rham Laplacian 
$\Delta=dd^*+d^*d$ and $\nabla^*\nabla$ be the Bochner  Laplacian of the Riemannian manifold $M$. 
Recall that $\nabla:\Omega^m(M)\to T^*(M)\otimes \Omega^m(M)$ is the covariant derivative
induced by the Levi-Civita connection on  the space of $m$-forms  $\Omega^m(M)$ and $\nabla^*$ its adjoint, namely
\[ \nabla^*\omega(X_1,\ldots,X_m)=-\sum_{i=1}^n\nabla_{e_i}\omega(e_i,X_1,\ldots,X_m) ,\] 
for any vector fields $X_j$ and an orthonormal local basis $\{e_i\}$. 

If $\omega$ is a $m$-form, then the Weitzenb\"ock formula states that: 
\[ \Delta \omega = \nabla^*\nabla + \mathcal R_m \omega.\]
Here $\mathcal R_m$ is the $m$-th Weitzenb\"ock curvature operator, namely a degree zero pseudo-differential operator 
determined by the Riemann tensor:
\[  \mathcal R_m\omega= \sum_{i,j=1}^n (R(e_i,X_j)\omega)(X_1,\ldots, X_{j-1},e_i,X_{j+1}\ldots,X_m)\]
The Riemannian connection induces a metric connection in the bundle $\Omega^m(M)$ and its curvature 
$R$ used above is a section of $\Omega^2(M)\otimes \Omega^m(M)$. 

Let now consider the real function defined on $M$ by the formula:  
\[ G_m=\left( (n-p)\sum_{i=1}^p\lambda_i +p\sum_{i=p+1}^n\lambda_i\right),\]
where $\lambda_1\leq \lambda_2\leq \cdots \leq \lambda_n$ denote 
the eigenvalues of the Schouten operator $S$ defined as
\[ S(X,Y)= \frac{1}{n-2}\left(Ric(X,Y)-\frac{1}{2(n-1)}R g(X,Y)\right)\] 
in terms of the Ricci tensor $Ric$ and the scalar curvature $R$. 
Assume  that the Weyl tensor of $M$ vanishes so that 
the Riemann tensor satisfies  the following identity (see \cite{La}):
\[ R(X,Y,Z,T)=g(Y,Z)S(X,T)-g(X,Z)S(Y,T)+g(X,T)S(Y,Z)-g(Y,T)S(X,Z)\]
It follows by direct calculation (see \cite{GW}, Appendix)  that $\mathcal R_m \omega = G_m \omega$.  
In particular we obtain the following:
\begin{lemma}\label{weitz}
Let $M$ be a Riemannian $n$-manifold, $n\geq 3$, whose Weyl tensor vanishes. 
Then the function $G_m:M\to \R$ is  smooth. 
\end{lemma}

\subsection{Isotropic points of  $1$-QC manifolds}\label{section-extension}
In this section $M$ is a $1$-QC manifold and  $C$ denotes the closed set of its isotropic points. 
The first observation is: 
\begin{lemma}\label{global}
If $M$ is a compact $1$-QC then $M-C$ is globally $1$-QC. In particular $H$ and $N$ are smooth 
real functions defined on $M-C$. 
\end{lemma}
\begin{proof}
According to the definition $M-C$ is locally $1$-QC, hence it admits a covering by open sets $U_i\subset M-C$
which are globally $1$-QC with respect to line fields $\xi_{U_i}$.   
However, points $p\in U_i\cap U_j$ where the  corresponding line fields do not agree should be 
isotropic, following  Lemma \ref{isotropy} and the proof of Lemma \ref{local-isotrop}, which  
contradicts $p\in M-C$. Thus the distinguished 
line field is well-defined and  $M-C$ is a globally $1$-QC manifold.
\end{proof}

\begin{lemma}\label{weyl}
The Weyl tensor of $M$ vanishes.
\end{lemma}
\begin{proof}
This is known for points of $M-C$ (see \cite{BF,GM}) and is classical for isotropic points. 
\end{proof}

Let $\widehat{\xi}$ be a locally defined unit vector field 
which determines the line field $\xi$. 
Moreover, we denote by the same letter $\widehat{\xi}$ an arbitrary extension 
(not necessarily continuous) to all of $M$. 

\begin{lemma}\label{continuous}
If $M$ is a compact $1$-QC manifold 
then the functions $H$ and $N$ have a canonical smooth
extension to $M$. 
\end{lemma}
\begin{proof}
Let $R$ denote the scalar curvature and $Ric$ the Ricci curvature. 
We set then: 
\[H(p)=N(p)=\frac{1}{n(n-1)}R(p), \; {\rm if } \; p\in C.\]

The formulas for the Ricci tensor from \cite{BF,BP} imply that 
\[ Ric(X,Y)=((n-2)H+N)g(X,Y)+(n-2)(N-H)g(X,\widehat{\xi})g(Y,\widehat{\xi})\]
at all points $p\in M-C$. Note that 
$g(X,\widehat{\xi})g(Y,\widehat{\xi})$ is well-defined on $M-C$ 
and independent on the choice of the vector field $\widehat{\xi}$. 
Nevertheless this formula  makes sense 
and holds at points $p\in C$, as well.

The Ricci operator has eigenvalues 
$(n-2)H+N$ and $(n-1)N$ with multiplicities $n-1$ and $1$, respectively at points of $M-C$ and 
$(n-2)H+N=(n-1)N$ with multiplicity $n$ at points of $C$.  
  
The Schouten tensor $S$ is from above:  
\[ S(X,Y)=\frac{H}{2}g(X,Y)+(N-H)g(X,\widehat{\xi})g(Y,\widehat{\xi}).\]
The eigenvalues of the Schouten operator are therefore 
$\frac{H}{2}$ and $N-\frac{H}{2}$, with multiplicities  $n-1$ and $1$, respectively at points of $M-C$ and 
$\frac{H}{2}=N-\frac{H}{2}$ with multiplicity $n$ at points of $C$.  

According to Lemmas \ref{weitz} and \ref{weyl} the function $G_m$ is smooth on $M$. 
An explicit calculation shows that: 
\[ G_1= N+(n-2)\min(H,N),  {\rm for }\;  n\geq 3,\] 
\[G_2=2N+ (n-2)H+(n-4)\min(H,N)=\frac{R}{n-1}+(n-4)\min(H,N),  {\rm for }\;  n\geq 3,\]
\[G_3=3N+ 2(n-3)H+(n-6)\min(H,N),  {\rm for }\;  n\geq 4.\]
If $n\neq 4$, since $R, G_1$ and $G_2$ are smooth,  we derive that  
$\min(H,N)$ and then $N$ and $H$ are smooth functions on $M$. 
When $n=4$, $G_1+G_3=2H+4N$ and $\frac{R}{3}=2N+2H$  are smooth, so 
$H,N$ and $\min(H,N)$ are smooth. 
\end{proof}
\begin{remark}
The  above extensions of $H$ and $N$ are the unique continuous extensions to $\overline{M-C}$. 
\end{remark}

\subsection{Conformal flatness}\label{section-conformal}
Let $M$ be a compact 1-QC manifold. We define the {\em horizontal distribution} $\mathcal D$ to be the 
orthogonal complement to the distinguished line field $\xi$.  By Lemmas \ref{global} and \ref{local-line}, the distribution 
$\mathcal D$  is well-defined on $M-C$. 
One important  ingredient  used in this paper is the 
following easy observation (see also \cite{BF,GM}):
\begin{lemma}\label{conformal}
\begin{enumerate}
\item 
$1$-QC  manifolds  of dimension $n\geq 4$ are conformally flat. 
\item $1$-QC manifolds of dimension $n=3$ are conformally flat if and only if 
the horizontal distribution $\mathcal D$ is completely integrable on  $M-C$.
\end{enumerate}  
\end{lemma}
\begin{proof}
This is already known when the dimension $n\geq 4$  
because the Weyl tensor vanishes (see \cite{BF}, p.379 and \cite{GM}). 
 
By the Weyl-Schouten theorem (see \cite{La}, section C, Thm. 9) 
the 3-manifold $M$ is conformally flat iff 
the so-called Cotton-Schouten tensor vanishes, namely if:
\[ (\nabla_XS)(Y,Z)-(\nabla_YS)(X,Z)=0.\]
Given $p\in M-C$, let $U$ be a disk neighborhood of $p$ within $M-C$. 
Define the 1-form $\eta$ on $U$ by $\eta(X)=g(\widehat{\xi|_U},X)$,
where $\widehat{\xi|_U}$ is a vector field generating $\xi|_U$. 
Then the horizontal 
distribution  $\mathcal D$ is the kernel of $\eta$ on $U$. 
The vanishing of the Cotton-Schouten tensor on $U$ is equivalent to: 
\[ dH(X)g(Y,Z)-dH(Y)g(X,Z) + 2(N-H) d\eta(X,Y)\eta(Z) + \]
\[+2(N-H)\left(\eta(Y)(\nabla_X\eta)Z-\eta(X)(\nabla_Y\eta)Z\right)+\]
\[ +2\left(d(N-H)(X)\eta(Y)-d(N-H)(Y)\eta(X)\right)\eta(Z)=0\]
Using all possibilities for  choosing $X,Y,Z$  among 
the elements of a basis of the 
horizontal distribution $\mathcal D$ or $\widehat{\xi|_U}$ 
it follows that this is equivalent to the system of the following three 
equations:
\[dH=\widehat{\xi|_U}(H) \eta;\]
\[ \nabla_{X}\widehat{\xi|_U}= \frac{\widehat{\xi|_U}(H)}{2(N-H)} X, \; {\rm if } \; X\in \mathcal D;\]
\[ d\eta(X,Y)=0, \; {\rm if } \;  X,Y\in \mathcal D.\]
The first two equations are always verified by using the method of 
\cite{BP,GM} in all dimensions $n\geq 3$.  The third one is equivalent 
to the fact that the horizontal distribution $\mathcal D$ 
is completely integrable on $U$.  
Thus  $\mathcal D$ is integrable  if $M$ is conformally flat. 

Conversely, if $\mathcal D$ is integrable, the Cotton-Schouten tensor vanishes 
on $M-C$, and hence on $\overline{M-C}$. On the other hand 
$M-{\overline{M-C}}$ coincides with the set ${\rm int}(C)$ of interior points of $C$, namely 
those points of $C$ having an open neighborhood in $M$ which is contained in $C$. 
By Schur's theorem  $H=N$ are constant in any connected component of ${\rm int}(C)$ and 
in particular the Cotton-Schouten tensor vanishes. This proves that $M$ is conformally flat.  
\end{proof}

\begin{remark}
One cannot dispose of the integrability condition  for $\mathcal D$ when $n=3$.
In fact, any closed 3-dimensional nilmanifold 
is covered by the Heisenberg group, endowed with a 
left invariant Riemannian metric. Then  
the tangent space at each point is identified to the 
nilpotent Lie algebra generated by $X,Y,\xi$ with relations 
$[X,Y]=\xi$, $[X,\xi]=[Y,\xi]=0$. The 
horizontal distribution corresponds to the subspace 
generated by $X,Y$ and is well-known to be non-integrable.
However, the left invariant Riemannian metric on the Heisenberg 
group is globally $1$-QC, as we can compute 
$H=-\frac{1}{4}$ and $N=\frac{3}{4}$. This furnishes the 
typical example of a closed $1$-QC 3-manifold which is not conformally flat. 
In fact closed nilmanifolds are not conformally flat 
(see \cite{Go}) unless they are finitely 
covered by a sphere, a torus or $S^1\times S^2$.  
\end{remark}

\subsection{The leaf curvature function $\lambda$}\label{leafcurv}
Let $M$ be a 1-QC manifold, which is assumed to be conformally flat, when $n=3$. 
Then  the distribution $\mathcal D$ is integrable and hence defines a foliation of $M-C$. The leaves 
of the foliation will be called {\em curvature leaves}. 
According to \cite{BP,GM} every curvature leaf is a totally umbilical submanifold 
of $M$ of intrinsic constant curvature equal to: 
\[\lambda=H + \frac{\widehat{\xi}(H)^2}{4(N-H)^2}.\]
Observe that the right hand side above is independent on the choice of a local distinguished unit vector field 
$\widehat{\xi}$ which lifts $\xi$.

Before to proceed, recall that a real function $f$ defined on a compact subset $K$ of a manifold $M$ is 
called {\em smooth} if for every point $p\in K$ there exists an open neighborhood $U\subset M$ containing $p$ 
and a real smooth function $f_U$ defined on $U$ such that $f|_{K\cap U}=f_U|_K$.

\begin{lemma}\label{extension-lambda}
Let $M$ be a  locally $1$-QC manifold and $\xi$ denote a  distinguished line field extension to $\Omega\supset \overline{M-C}$.  
Then  $\lambda$ admits a well-defined smooth extension to $\overline{M-C}$.  
\end{lemma}
\begin{proof}
For every point $p$ of $M$ we choose an open neighborhood $V\subset M$ which is contractible so that 
the distinguished line field $\xi_V$  is defined and can be lifted to a vector field $\widehat{\xi}_V$. 
Let  $X_V$ be a unit vector field  on $V$ orthogonal to $\widehat{\xi}_V$.  
According to (\cite{BP}, (3.11)) and (\cite{GM} (3.5)) we have:
\[\left.\frac{\widehat{\xi}_V(H)}{2(N-H)}\right|_p= g(\nabla_{X_V}\widehat{\xi}_V, X_V)|_p, \; {\rm for } \; p\in V\cap (M-C).\]
The function $(g(\nabla_{X_V}\widehat{\xi}_V, X_V))^2$ is smooth and well-defined on $V$, namely independent on the choice of the lift $\widehat{\xi}_V$. We set therefore 
\[\lambda_V(p)=\left\{\begin{array}{ll}
\lambda(p), & {\rm if }\; p\in V\cap (M-C);\\ 
H+(g(\nabla_{X_V}\widehat{\xi}_V, X_V))^2, & {\rm if }\; p\in  V\cap C,  
 \end{array}\right.\]
where $X_V$ is an arbitrary smooth unit vector field on $V$ orthogonal to $\widehat{\xi}_V$. 
Then $\lambda_V$ is a smooth function on $V$.

Consider a finite open covering $\{U_i\}$ of $M$ by open sets as above. Let $\Omega\supset \overline{M-C}$ be an  
open set on which the distinguished vector field $\xi$ is well-defined, as provided by Lemma \ref{local-line}. 
If $U_i,U_j$ are two such open sets and $p\in U_i\cap U_j\cap \overline{M-C}$, then Lemma \ref{local-line} implies that 
$\widehat{\xi}_{U_i}|_p=\pm\widehat{\xi}_{U_j}|_p$, so that 
\[ (g(\nabla_{X_{U_i}}\widehat{\xi}_{U_i}, X_{U_i}))^2|_p= (g(\nabla_{X_{U_j}}\widehat{\xi}_{U_j}, X_{U_j}))^2|_p, \; {\rm if } \; p\in U_i\cap U_j \cap \overline{M-C}.\]
Therefore 
the collection $\lambda_{U_i}$ defines a smooth extension $\lambda:\overline{M-C}\to \R$. 
\end{proof}
\begin{remark}
Lemma \ref{extension-lambda} shows that $\lambda(p)$ remains bounded 
as $p$ approaches $C$. 
\end{remark}

\begin{remark}
Recall also (see \cite{BP,GM}) that 
\[ |\widehat{\xi}(H)|=\parallel {\rm grad }\; H\parallel \]
and the function $\lambda$  is constant on the 
level hypersurfaces of $H$. 
\end{remark}

\subsection{Isometric immersions of $1$-QC manifolds}\label{section-isometric}
One important tool which  will be used throughout this paper 
is the  codimension one immersability of simply connected $1$-QC manifolds. 
Let $\mathbb H_{\kappa}^m$ denote the $m$-dimensional 
hyperbolic space of constant sectional curvature 
$-\kappa$, when $\kappa>0$ and the Euclidean space when $\kappa=0$, 
respectively. For a $1$-QC manifold  $M^n$ we denote 

\[ \kappa= \left\{\begin{array}{ll}
1-\inf_{p\in M} H(p), & {\rm when } \;  \inf_{p\in M}H(p) \leq 0;  \\
0, &{\rm when } \;  \inf_{p\in M}H(p) > 0.
\end{array}\right.\]
The main result of this section is: 

\begin{proposition}\label{immerse}
Assume that $M^n$, $n\geq 3$, admits a $1$-QC metric which is conformally flat when $n=3$.  
Then there exists an equivariant 
isometric immersion, namely an isometric immersion 
$f:\widetilde{M}\to \mathbb H_{\kappa}^{n+1}$, and 
a group homomorphism $\rho:\pi_1(M)\to Iso( \mathbb H_{\kappa}^{n+1})$ such that  
\[ f(\gamma\cdot x)=\rho(\gamma)\cdot f(x), x\in \widetilde{M}, \gamma \in \pi_1(M)\]
where the action on the right side is the action by isometries of 
$Iso( \mathbb H_{\kappa}^{n+1})$ on $\mathbb H_{\kappa}^{n+1}$ while the left side action is by 
deck transformations on the universal covering.
\end{proposition}
\begin{proof}
The universal covering  $\widetilde{M}$ is a 
$1$-QC manifold with the induced metric by the covering $\pi:\widetilde{M}\to M$. 

Every $p\in M-C$ has an open contractible neighborhood such that the line field $\xi$ 
can be lifted to a vector field $\widehat{\xi}_U$.  We derive a smooth tensor  
$h'_U(X,Y)=g(\widehat{\xi}_U,X)g(\widehat{\xi}_U,Y)$ on $U$. If $p\in U\cap V$ then 
$\widehat{\xi}_U|_p=\pm \widehat{\xi}_V|_p$ and thus 
$h'_U(X,Y)|_p=h'_V(X,Y)|_p$, for any $p\in U\cap V$. Thus the collection $h'_U$, where $\{U\}$ 
is a covering  of $M-C$ defines a smooth tensor field $h'$ on $M-C$.

Define now the  following smooth  symmetric (1,1) tensor field on $M$ pointwise, by means of 
the formula:  
\[h(X,Y)|_p=\left\{\begin{array}{ll}
 \sqrt{H+\kappa}\cdot g(X,Y)|_p+ \frac{N-H}{\sqrt{H+\kappa}}\cdot h'(X,Y)|_p, &{\rm if } \;p\in M-C; \\
 \sqrt{H+\kappa} \cdot g(X,Y)|_p, &{\rm if } \;p\in C, 
\\
 \end{array}\right.
\]
and also denote by $h$ the pull-back to $\widetilde{M}$.

One verifies immediately the Gauss and Codazzi-Mainardi equations associated to the tensor $h$, namely: 
\[ R(X,Y,Z,W)=-\kappa\left(g(Y,Z)g(X,W)-g(X,Z)g(Y,W)\right)+h(Y,Z)h(X,W)- h(X,Z)h(Y,W)\]
\[ \nabla_Xh(Y,Z)=\nabla_Yh(X,Z)\]
This follows from the second Bianchi identity of the curvature tensor, when 
$n\geq 4$ and from the vanishing of the Cotton-Schouten tensor 
in dimension 3. The second equality is treated as in
the proof of lemma \ref{conformal} by noticing that only the last
condition  discussed there has to be verified, 
the other ones being automatic. 

Therefore, by the fundamental theorem for submanifolds of constant curvature 
space, (see \cite{S}, chap.7, section C, Thm. 21)   $\widetilde{M}$
 admits an isometric immersion $f:\widetilde{M}\to \mathbb H_{\kappa}^{n+1}$, whose 
second fundamental form is $h$. 

The second fundamental form $h$ on $\widetilde{M}$ is invariant 
by the action of the deck group, by construction. This means that there exist 
open neighborhoods of $x$ and $\gamma\cdot x$ which are isometric. 
Let $\rho(\gamma)_x$ be the global isometry sending one neighborhood onto the 
other. Now $\rho(\gamma)_x$ is locally constant, as a function of $x$, 
and hence it is independent on $x$. Furthermore $\rho(\gamma)$ is 
uniquely defined by the requirement to send $n+1$ points $x_i$ 
in general position in $\mathbb H_{\kappa}^{n+1}$, which belong to 
a small open neighborhood of $x$  into their respective 
images  $\gamma\cdot x_i$. Therefore, if the image of the immersion is 
not  contained in a hyperbolic hyperplane then $\rho(\gamma)$ is unique 
and this also implies that $\rho$ is a group homomorphism. 
In the remaining case $M$ is hyperbolic and the result follows again.   
\end{proof}

\begin{remark}
The immersability result holds also when $M^n$ is not 
necessarily compact, provided $\inf_{p\in M} H(p)$ exists. 
\end{remark}

Contrasting with the 2-dimensional case, but in accordance with Schur's 
theorem the space of  such isometric immersion is discrete. The first 
and simplest case is:
\begin{proposition}\label{unique-immerse}
Suppose that $n\geq 4$. If $\kappa$ is such that $H+\kappa >0$ on $M$, then 
there exists a unique isometric immersion 
$f:\widetilde{M}\to \mathbb H_{\kappa}^{n+1}$, up  
to an isometry (possibly reversing the orientation) of $\mathbb H_{\kappa}^{n+1}$.
\end{proposition}
\begin{proof}
If $h$ is the second fundamental form of such an immersion, then 
linear algebra computations show that the Gauss equation already 
determine $h$ up to a sign. A hyperbolic isometry reversing orientation will 
change the sign of $h$. Now an isometric 
immersion is uniquely determined by its second fundamental form. 
\end{proof}
\begin{remark}
The same argument shows that there does not exist any isometric immersion in 
$\mathbb H_{\kappa'}^{n+1}$ around a point where  $H+\kappa' \leq 0$.
\end{remark}

Conversely we have: 
\begin{proposition}\label{characterize}
If  the conformally flat manifold $M^n$, $n\geq 4$, admits  an equivariant  isometric immersion 
$f:\widetilde{M}\to \mathbb H_{\kappa}^{n+1}$, and 
a group homomorphism $\rho:\pi_1(M)\to Iso( \mathbb H_{\kappa}^{n+1})$ such that 
\[ f(\gamma\cdot x)=\rho(\gamma)\cdot f(x), x\in \widetilde{M}, \gamma \in \pi_1(M)\]
then $M$ is $1$-QC. 
\end{proposition}
\begin{proof}
According to Cartan 
(see \cite{Ca} and \cite{La}, Thm. 4) at each point $p\in \widetilde{M}$ there exists 
a principal curvature $\mu$  (i.e. an eigenvalue of the second fundamental form) of  the immersed 
manifold of multiplicity at least $n-1$. 
Let $\mathcal N$ be the  open set of non-umbilical points, 
namely those for which $\mu$ has multiplicity exactly $n-1$.  The distribution $\mathcal D$ is the eigenspace for the 
shape operator corresponding to the multiple eigenvalue. 
From Gauss' equations we obtain that $\mathcal N$ is globally $1$-QC, with distinguished line field given by the 
normal line field to $\mathcal D$. The complement of $\mathcal N$ is totally umbilical in  
$\mathbb H_{\kappa}^{n+1}$ and hence it has constant curvature. Since the immersion is equivariant 
these properties descend to $M$. 
\end{proof}
\begin{remark}
Note that the statement of Proposition \ref{characterize} is not true when $n=3$ 
as Cartan's theorem does not extend  to $n=3$ (see \cite{La}). 
\end{remark}

\section{The topology of locally  $1$-QC manifolds}
\subsection{Globally $1$-QC manifolds with orientable line field}
The purpose of this section is to prove that 
the globally $1$-QC manifolds  have a rather simple topological structure.
Throughout this section $M$ is a conformally flat $1$-QC manifold of dimension $n\geq 3$.  
Let $E$ be a connected component of $M-C$.  
By Lemma \ref{global} the open set $M-C$ is globally $1$-QC.  
In order to have a uniform approach we denote by $\widehat{E}$ the 
2-fold cover determined by the first Stiefel-Whitney class of $\xi|_{E}$, 
when $\xi|_E$ is non-orientable and $E$ itself otherwise.

Moreover, $M-C$ is endowed with a smooth foliation $\mathcal D$, which is orthogonal to the line field $\xi$. 
Following \cite{Re} the leaves of the foliation $\mathcal D$ will be called {\em curvature leaves}, 
as the principal curvatures are  constant along them.

\begin{proposition}\label{dicho}
If $\widehat{E}$ is closed then it is diffeomorphic to a 
fibration over $S^1$. If $\widehat{E}$ is non-compact then 
$\widehat{E}$ is diffeomorphic to $\widehat{P}\times \mathbb R$, where 
$\widehat{P}$ is a closed space form. 
\end{proposition}

\begin{corollary}\label{sign}
If $E$ is a connected component of $M-C$ then 
$\lambda(p)$ is either positive, or negative, or else 
vanishes for all $p\in E$. 
\end{corollary}
\begin{proof}
With respect to the pull-back metric to the cover the slices 
$\widehat{P}\times \{a\}$ of $\widehat{E}$ have constant curvature $\lambda$. 
But a closed manifold $\widehat{P}$ cannot support constant 
curvature metrics  with different signs. In fact $\pi_1(\widehat{P})$ is either 
hyperbolic and not virtually cyclic, virtually abelian or finite, according to whether $\lambda$ is negative, null or positive. 
\end{proof}

The main step is to establish the following: 

\begin{proposition}\label{compactleaf}
Leaves of the curvature foliation on $\widehat{E}$ are compact. 
\end{proposition}

\begin{proof}[Proof of Proposition \ref{dicho} 
assuming Proposition \ref{compactleaf}]
The strong form of the global Reeb stability theorem states that 
a codimension one transversely orientable 
foliation whose leaves are all compact 
is locally a fibration and the holonomy groups are finite (see 
\cite{Mil,EMS}).  
\end{proof}

{\em Outline of the proof of Proposition \ref{compactleaf}}: We use a Theorem of Reckziegel to prove that the foliation has complete leaves 
with respect to the induced metric in section \ref{sect:Reckziegel}. Since $H$ is constant along the curvature leaves we derive that those 
leaves on which the value of $H$  is non-critical should be compact in section \ref{sect:noncritical}. This can be improved 
to those leaves which can be approached by  non-critical leaves by using a Theorem of Haefliger (section \ref{sect:Haefliger}) which states that 
the union of closed leaves is closed. It remains then to analyze the leaves for which $H$ is constant on an open neighborhood.  
Such a maximal saturated open set  on which $H$ is constant admits a  metric completion whose general structure was 
described by Dippolito (see section \ref{Dippolito}). We will first show that the completion is transversely compact and further that 
the connected components of its boundary are compact, as they can be approached by non-critical and hence compact leaves, in section \ref{sect:constant}. We derive then that the completion is a foliated product and hence all leaves are compact.

\subsection{Completeness of curvature leaves}\label{sect:Reckziegel}
Let $f:N\to P$ be an isometric immersion of the Riemannian manifold $N$ into a Riemannian manifold $P$ with constant curvature and dimension strictly larger than that of $N$.
Set $\nu(f)$ for the normal bundle of $f$, $\nu^*(f)$ for its dual and $II:TN\times TN \to \nu(f)$ for the second fundamental tensor of $f$. 
The shape operator $A_{\eta}:TN\to TN$ is defined by the identity 
$g(A_{\eta}X,Y)=g_P(II(X,Y),\eta)$, for all sections $X,Y$ of $TN$ and $\eta$ of $\nu(f)$, where $g_P$ and $g$ are the metric tensors on $TP$ and $\nu(f)$, respectively. 
 
A {\em principal curvature function} is a continuous section $\mu$ of $\nu^*(f)$ for which the eigenspace  $\mathcal E(\mu(p))$ has dimension 
$\dim  \mathcal E(\mu(p))\geq 1$ for each $p\in N$, where we put 
\[  \mathcal E(\mu(p))= \{ v \in T_pN; A_{\eta} v= \mu(\eta) v, \; {\rm for \; any} \; \eta\in  \nu_p(f)\}\]
Let $U$ be an open subset of $N$ for which $\dim \mathcal  E(\mu(p))$ is constant, say equal to $k$, for every $p\in U$.

A key ingredient of the proof of Proposition \ref{compactleaf} is the following result of  Reckziegel (\cite{Re}, Thm.1.1, Rem.1,2): 

\begin{theorem}\label{Reckziegel}
\begin{enumerate}
\item The principal curvature function $\mu$ is smooth on $U$. 
\item The vector spaces $\mathcal E(\mu(p))$, for $p\in U$ form a subbundle $\mathcal E(\mu)$ of $TN|_U$ which is integrable. 
\item If $L$ is a leaf of the foliation defined by the subbundle $\mathcal E(\mu)$ then $L$ is a totally umbilical submanifold of $N$ of dimension $k$ and 
$f|_L: L\to P$ is a totally umbilical immersion. 
\item If $k\geq 2$ then 
for any geodesic $c:[0,\ell) \to L$ with $\ell < \infty$ for which $\lim_{s \to \ell} c(x)=q$ exists in $N$, then $\dim \mathcal E(\mu(q))=k$. 
\item If $k\geq 2$, $N$ is complete and $U$ is the subset of those points $p\in N$ for which $\dim \mathcal E(\mu(p))$ is minimal 
($U$ is open since this $\dim \mathcal  E(\mu(p))$ is upper semicontinuous), then 
every leaf $L$ of the foliation of $U$ defined by $\mathcal E(\mu)$ is a complete metric space. 
\end{enumerate}
\end{theorem}

Consider now a conformally flat 1-QC manifold $M$ of dimension $n\geq 3$. 
Proposition \ref{immerse} shows that there exists an equivariant isometric immersion 
$f: \widetilde{M}\to \mathbb H_{\kappa}^{n+1}$. Set  $\pi:\widetilde{M}\to M$ for the universal covering projection. 
Observe now that one principal curvature $\mu$ of the immersion $f$ is of multiplicity $n-1$  
precisely along the open set $U=\pi^{-1}(M-C)$. Moreover, the pull-back $\pi^{*}\mathcal D$ of the distribution $\mathcal D$ 
by $\pi$ coincides with the subbundle $\mathcal E(\mu)$ on $U$.  

Reckziegel's Theorem \ref{Reckziegel} in the case 
$N=\widetilde{M}$ and $P=\mathbb H_{\kappa}^{n+1}$ therefore has the following corollary:

\begin{lemma}\label{Reck}
The leaves of the foliation determined by $\pi^{*}\mathcal D$ on 
the subset $U=\pi^{-1}(M-C)$ of $\widetilde{M}$  are complete. 
\end{lemma}

\subsection{Haefliger's theorem about the set of compact leaves}\label{sect:Haefliger}
One more ingredient needed in the sequel is the following theorem due to Haefliger (see \cite{Hae}, Thm.3.2, p.386):

\begin{theorem}\label{Haefliger}
Let $V$ be a connected manifold with countable basis such that the rank of $H_1(V;\mathbb Q)$ is finite. Consider a codimension one 
foliation of $V$. Then the union of closed leaves of the foliation is a closed subset of $V$. 
Moreover, if $V$ is  also an open subset of a compact manifold, then the union of compact leaves of the foliation is a closed subset of $V$.  
\end{theorem}
\begin{remark}
The proof given by Haefliger works as well under the assumption that  $H_1(V;\Z/2\Z)$ is finite. 
\end{remark}

\subsection{Open saturated subsets after Dippolito}\label{Dippolito}
The structure of open saturated subsets $U$ of a foliated manifold $W$ was described by Dippolito (\cite{Dip}) 
in the case when $W$ is compact and  it was extended to relatively compact saturated  subsets $U$ of open manifolds in
\cite{CaC}. 
Specifically, let $\wideparen{U}$  
denote the completion of $U$ with respect to the induced Riemannian metric from $W$. Note that, 
in general, $\wideparen{U}$ is non-compact.  
Then the inclusion $i:U\to W$ induces an immersion $\wideparen{i}:\wideparen{U}\to W$, by 
extending $i$ to limit points of trajectories of a transverse foliation. 
  
Dippolito's nucleus Theorem still holds in the open case as explained in (\cite{CaC}, section 4):
\begin{proposition}
There exists a 
compact connected manifold with piecewise smooth boundary $K\subset \wideparen{U}$ called a {\em nucleus}  
and finitely many $V_i\subset \wideparen{U}$ such that:  
\begin{enumerate}
\item $\wideparen{U}=K\cup V_1\cup\cdots V_r$
\item $V_i$ is diffeomorphic to $B_i\times [0,1]$, where $B_i\times\{0\}$ and 
$B_i\times\{1\}$ are complete connected submanifolds of  the part of $\partial \wideparen{U}$ which is tangent to $\mathcal D$.  
\item for each $z\in B_i$, $\{z\}\times [0,1]$ is a trajectory of $\widehat{\xi}$. 
\item if $K\neq\emptyset$, then $V_i\cap K$ is non-empty and connected.
\item if $i\neq j$ then $V_i\cap V_j=\emptyset$. 
\item if $L$ is a component of $\partial\wideparen U$ which is tangent to $\mathcal D$ and $K$ is non-empty,  
then $K\cap L$ is non-empty and connected. 
\end{enumerate}
\end{proposition}
The nucleus is not unique in general.  If we can take $K=\emptyset$, then $\wideparen{U}=L\times [0,1]$ is a 
{\em foliated product}, foliated by the leaves of the foliation diffeomorphic to $L$ and also transversally by closed intervals.  

We will need in the sequel the following result stated by Dippolito (\cite{Dip}, Prop.2)  in the compact case and by Cantwell and Conlon \cite{CaC}, Lemma 4.2) in the open case: 
\begin{lemma}\label{finiteborder}
Let $U\subset W$ be  an open connected saturated set. 
Then  $\partial \wideparen{U}$ has only finitely many connected components tangent to $\mathcal D$. 
\end{lemma}
Note that the number of such boundary components is bounded by twice the number of
so-called biregular charts needed to cover $\overline{U}$, which is finite since $\overline{U}$ is compact (see \cite{Dip}).

\subsection{Proof of Proposition \ref{compactleaf} for everywhere non-constant $H$}\label{sect:noncritical}
The function $H:M\to \R$ is smooth and proper, because $M$ is compact. 
Then, its critical values are nowhere dense and the preimage 
of any non-critical value is a codimension one compact 
submanifold of $M$. A curvature leaf $L$ 
will be called {\em non-critical} if $H(L)$ is a non-critical value. 
We want to stress that a curvature leaf $L$ is contained in $M-C$.

\begin{lemma}\label{noncritical}
A non-critical curvature leave $L$ is a connected component of 
$H^{-1}(H(L))$, and in particular it is compact. 
\end{lemma}
\begin{proof} 
Recall that $H$ is constant on curvature leaves. This follows from (\cite{BP}, (3.5), and \cite{GM}, (3.4-3.5)(i.)), which state that 
\[ X(H)=0, \; {\rm for}   \; X\in \mathcal D\]

If $\pi:\widetilde{M}\to M$ is the universal covering projection, then 
each component of the preimage $\pi^{-1}(L)$ of $L$ 
in $\widetilde{M}$ is complete, following Reckziegel's Lemma \ref{Reck}.
This implies that the leaf $L$ itself should be 
geodesically complete, as otherwise 
we could adjoin  additional limit points to $\pi^{-1}(L)$  contradicting 
its completeness.  Thus $L$ is complete with respect to its 
intrinsic metric. 

As a word of warning the projection $\pi$ is not a closed map, if 
the covering has infinite degree. In particular, $L$ is not necessarily 
a complete subspace of $M$ and in particular 
it might not be a closed subset of $M$.  

If $H(L)$ is a non-critical value of $H$ then  $H^{-1}(H(L))$ is a smooth hypersurface of $M-C$; since $H$ is proper, 
 $H^{-1}(H(L))$ is closed.  Then $L$ is a codimension zero 
submanifold of  the closed manifold $H^{-1}(H(L))$. The level hypersurface 
$H^{-1}(H(L))$ has an induced Riemannian metric from $M$ whose restriction 
to its open subset $L$ coincides with the intrinsic metric of $L$. 
Therefore $L$ is a complete subspace of  the compact space 
$H^{-1}(H(L))$  and hence it has also to be closed.  
This implies that  $L$ must be a connected component 
of the level hypersurface  $H^{-1}(H(L))$ and hence compact. 
\end{proof}

Note that, as $\mathcal D$ is an integrable distribution on $M-C$, 
all curvature leaves, in particular critical leaves, are nonsingular.

Suppose now that $H$ is {\em everywhere non-constant}, namely that 
there is no open subset $U\subset M-C$ for which the restriction of $H$ to $U$ is constant. 
We want to apply Haefliger's result to the foliation by curvature leaves of the open manifold $M-C$. 
Assume that the set of non-compact leaves is non-empty. 
By Lemma \ref{Haefliger} the union $U$  of non-compact leaves is an open subset of $M-C$.  
Let $p\in U$, so that $L_p$ is a non-compact leaf and $H(L_p)=c$ is a critical value of $H$. 
Consider  a germ of an integral curve $\gamma\subset U$ of the line field $\xi$. Then the restriction of $H$ to $\gamma$ is non-constant, as 
otherwise  the union of leaves intersecting $\gamma$ would contain an open set on which $H$ is constant, contradicting the fact that  $H$ is everywhere non-constant. 
It follows that $H(\gamma)$ contains a small interval and hence $c$ can be approached by a sequence of regular values $c_i\in H(\gamma)$. 
Take points $p_i\in \gamma$ with $H(p_i)=c_i$. Therefore the leaves $L_{p_i}$ are non-critical leaves and hence compact leaves by above. 
This contradicts the fact that $U$ is the union of non-compact leaves. 
This ends the proof of Proposition \ref{compactleaf} for the case of everywhere non-constant $H$.

\subsection{Proof of Proposition \ref{compactleaf} for generic but not everywhere non-constant $H$}\label{sect:constant}
By the previous section  \ref{sect:noncritical} it remains to show that leaves passing through an open set on which $H$ is constant are also 
compact. Let $p\in M$ be a point such that $H$ is constant on an open neighborhood $U$ of $p$ in $M-C$.
Since the union of leaves intersecting an open set is an open set we can suppose that $U$ is saturated, namely union of leaves.
Moreover, as $H$ is smooth, $H$ is constant on the closure $\overline{U}$. 
Curvature leaves $L_p$ arising in this way will be called {\em super-critical}. 
  
\begin{lemma}\label{totallygeodesic}
A super-critical curvature leaf $L_p$ is totally geodesic. 
\end{lemma} 
\begin{proof}
Since $H$ is constant in a neighborhood of $p$ we have $dH|_p=0$.  Since $p\in M-C$, we can write  
$dH=\widehat{\xi}(H) \eta$, for some locally defined 
vector field $\widehat{\xi}$ lifting the line field $\xi$. Thus 
$\nabla_X\widehat{\xi}=\alpha X=\frac{\widehat{\xi}(H)}{N-H} X$ vanishes at $p$, for any 
horizontal vector field $X$. In particular $\alpha(p)=0$  and 
$\lambda(p)=H(p)$, according to the proof of Lemma \ref{extension-lambda}.   
Since $\lambda$ is constant on the level 
hypersurfaces of $H$ it follows that $\lambda(q)=H(q)=c_0$, 
for any $q\in L_p$ and thus the curvature leaf $L_p$ is totally geodesic. 
\end{proof}

Let $E$ be the connected component of $M-C$ containing $p$,  
$\widehat{E}$ the cover of $E$ determined by  $w_1(\xi|_E)$ 
and $\widehat{\xi}$ be a vector field lift of the line field. 
The cover $\widehat{E}$ is endowed with a pull-back metric and a 
completely integrable distribution still denoted 
$g$ and $\mathcal D$, respectively. We keep the notation $U$ for the lift of $U$ and work henceforth on $\widehat{E}$ in this section. 
Note that $\widehat{E}$ is contained in $\widehat{\Omega}$ and has compact closure there. 

We want to construct a {\em compact} submanifold $N\subset \widehat{E}$ 
containing $U$ which is foliated by curvature leaves.   
Possibly enlarging $U$ we can assume that $U$ is the 
{\em maximal} open connected saturated subset of $E$ containing $p$ 
such that $H|_{U}=H(p)$. Note that $U$ is the connected component 
of ${\rm int}(H^{-1}(H(p))\cap \widehat{E})$ containing $p$. Since $H$ is not eventually constant, the closure 
$\overline{U}$ in $M$ is contained in $\widehat{E}$, so that $U$ is relatively compact in $\widehat{E}$.

For every $p\in U$ we consider the {\em flow} $\varphi_p(t)$ determined by $\widehat{\xi}$ with initial condition 
$\varphi(0)=p$. Then $\varphi_p(t)$ is not necessarily defined on all of $\R$, since 
$\xi$ was only defined on an open neighborhood $\Omega$ of $\overline{M-C}$. 
Let $I_p\ni 0$ be the maximal interval consisting of  $t\in \R$ for which $\varphi_p(t)\in \widehat{E}$. 
Set further $J_p\subset I_p$ for the maximal interval on which $H\circ \varphi_p$ is constant, namely such that $0\in J_p$ and 
all $t\in J_p$ satisfy $H(\varphi_p(t))=H(p)$. Then $J_p$ is non-empty, because it contains 
the connected component of $0$ in $\varphi_p^{-1}(U)$.

\begin{lemma}\label{maxinterval}
The maximal interval $J_p$ on which $H\circ \varphi_p$ is constant is a finite proper sub-interval of $I_p$. 
\end{lemma}
\begin{proof}
If $I_p$ has a finite supremum (or infimum), then $\varphi_p(t)$ 
reaches $C$ in finite time. In this case $\sup J_p <\sup I_p$. Otherwise $H$ would be constant on the set 
$U'$, which is the union of all leaves intersecting $\varphi_p(I_p)$ 
(actually a small open neighborhood of it), while $\overline{U'}\cap C\neq \emptyset$.  
This would contradict the genericity of $H$. 

If $I$ is unbounded, say it contains $[0,\infty)$, we have two cases: 
either the forward limit set $\omega(\varphi_p)$ intersects $C$, or else $\omega(\varphi_p)\subset \widehat{E}$. 
Recall that the {\em forward limit} (or {\em  $\omega$-limit}) set of the flow 
$\varphi_p$ is the set of points obtained as limits of sequences 
$\varphi_p(s_i)$, where $s_i\to \infty$. In particular,  $\omega(\varphi_p)$ is a compact set 
invariant by the flow. 

In the first case, if $\sup J_p=\infty$, then again 
$H$ is constant on the union $U'$ of leaves intersecting  $\varphi(I_p)$, 
while $\overline{U'}\cap C\neq \emptyset$, 
contradicting the genericity of $H$. 

Consider now the second case,  when $\omega(\varphi_p)\subset \widehat{E}$ and assume that $J_p\supset [0,\infty)$. 
Recall that $H|_{\widehat{E}}$ is non-constant and thus there exist non-critical values $c\neq  H(p)$.
Let then $L_q$ be a non-critical leaf through a point $q\in E$, with $H(L_q)=c$. 
By Lemma \ref{noncritical} $L_q$ is a compact leaf. 

Further $\omega(\varphi_p)\cap L_q=\emptyset$, since $H(\omega(\varphi_p))=H(p)$ and 
$H(L_q)=c\neq H(p)$. Since both sets are compact, the distance between $\omega(\varphi_p)$ and $L_q$ is non-zero. 
Let then  $\beta:[0,1]\to \widehat{E}$, with  $\beta(0)\in \omega(\varphi_p)$,  $\beta(1)\in L_q$
be a minimal geodesic in $M$ realizing the distance 
between them.  
Since $\omega(\varphi_p)$ is invariant by the one-parameter flow generated by $\widehat{\xi}$ 
the integral orbit $\mathcal O$ of $\beta(0)$ under this one-parameter flow is 
contained in   $\omega(\varphi_p)$. 
Standard variational arguments show that $\beta$ must be orthogonal to both 
the  orbit $\mathcal O$ and the leaf $L_q$, namely that 
\[g(\dot{\beta}(0),\widehat{\xi}|_{\beta(0)})=0, \; g(\dot{\beta}(1), \mathcal D|_{\beta(1)})=0
\]
The first equation means that $\dot{\beta(0)}$ is tangent to the leaf 
$L_{\beta(0)}$ passing through $\beta(0)\in \widehat{E}$. 

Write then $\beta(0)=\lim \varphi_p(s_i)$, where $s_i\to \infty$. 
All leaves $L_{\varphi_p(s_i)}$ are super-critical (by our assumption on $J$) and hence 
$\alpha(x)=0$, for every $x\in L_{\varphi_p(s_i)}$. Since $\alpha$ is smooth, we derive that 
$\alpha(x)=0$ for all $x\in L_{\beta(0)}$, which means that $L_{\beta(0)}$  is a totally geodesic leaf.
 
Eventually, a minimal geodesic in $M$ whose direction is tangent to a totally geodesic 
leaf must be contained in that leaf. On the other hand the second equation 
above shows that $\beta$ is orthogonal to the leaf $L_q$ at $\beta(1)$,  which is 
a contradiction.
\end{proof}

Note that Lemma \ref{maxinterval} already shows that $\wideparen{U}$ is foliated transversely by closed intervals. 

\begin{lemma}\label{border}
Every component of $\delta(\wideparen{U})=\wideparen{i}(\partial \wideparen{U})$ is a compact leaf. 
\end{lemma}
\begin{proof}
The image $\wideparen{i}(\wideparen{U})$ of the completion $\wideparen{U}$ is the subset of $\widehat{\Omega}$ obtained 
by completing every transverse trajectory of $\widehat{\xi}$, namely replacing every open interval  $\varphi_p(J_p)$ by the corresponding closed 
interval  (or circle) $\varphi_p(\overline{J_p})$. 

In general, one can have part of the border tangent to the foliation $\mathcal D$ and part of it 
tangent to some trajectory of $\widehat{\xi}$ (see \cite{CaC}, section 4). However, the second case cannot occur here 
because $\widehat{E}$ is relatively compact within $\widehat{\Omega}$, more precisely it is 
contained in the  pull-back of $\overline{E}$ to the double cover $\widehat{\Omega}$. 
Thus every point of  $\delta(\wideparen{U})$ is of the form $\varphi_p(\sup J_p)$, or $\varphi_p(\inf J_p)$, for some $p\in U$ and 
 $\delta(\wideparen{U})$ is also the union of the corresponding curvature leaves passing through these points. 

As $\overline{J_p}$ is a closed proper interval within $I_p$, the restriction $H|_{\varphi_p([\sup J_p, \sup J_p+\varepsilon]))}$
is not constant and hence there exists a sequence of 
non-critical values of the form $H(\varphi_p(\sup J_p+t_i))$ approaching  $H(p)$, 
for some decreasing $t_i\to 0$. 
Since non-critical leaves are compact, Haefliger's Theorem   \ref{Haefliger} implies  
that the leaves $L_{\varphi_p(\sup J_p)}$ and $L_{\varphi_p(\inf J_p)}$ are compact. 
\end{proof}

There are only finitely many components of  $\partial\wideparen{U}$ and hence 
$\delta(\wideparen{U})$, by Lemma \ref{finiteborder}. It follows that $\wideparen{U}$ is a compact manifold. 
Now, every pair of points $p,q\in \partial \wideparen{U}$ having the same image by $\wideparen{i}$,  give raise 
to two leaves which are tangent at $\wideparen{i}$ so that they  should coincide. 
Therefore $N=\wideparen{i}(\wideparen{U})$  is a 
codimension zero compact submanifold of $\widehat{E}$. 
Specifically, $N$ is the closure of the connected component of ${\rm int}(H^{-1}(H(p))\cap \widehat{E})$ containing $p$.

Proposition \ref{compactleaf} will then follow from the following:

\begin{proposition}\label{cylinder} 
The submanifold   $N$  is  diffeomorphic to a cylinder foliated as a product.
\end{proposition}
\begin{proof}
Lemmas \ref{maxinterval} and \ref{border} and Dippolito's description from section \ref{Dippolito}, 
show that $U$ has an empty nucleus, and hence $\wideparen{U}$  is a foliated product by closed intervals. 

Here is a direct proof. 
Each component of $\partial N$ has a sign determined by 
the fact that $\widehat{\xi}$ points either inward or outward with respect to $N$. 
This is well-defined because $\widehat{\xi}$ is a smooth vector field on $\widehat{E}$ 
orthogonal to $\partial N$, so that if $\widehat{\xi}$ points inward at some point of 
$\partial N$ then it points inward at all points of its connected boundary component.

Let $\partial^+ N$ be the union of those components of 
 $\partial N$ for which 
$\widehat{\xi}$ is inward pointing and  $\partial^- N$ be its complementary. 
Observe that $N$ cannot be a closed manifold without boundary since it is 
contained in $\widehat{E}$. By the symmetry of the situation we can assume that 
$\partial^+ N$ is non-empty. 
We claim now that:
\begin{lemma}
Every integral trajectory of $\widehat{\xi}$ starting at a point 
$x\in \partial^+ N$ should cross  $\partial^- N$.
\end{lemma} 
\begin{proof}
Assume the contrary. 
As the trajectory  cannot cross again $\partial^+ N$, 
necessarily in outward direction, because this  would contradict the choice of 
$\partial^+ N$, it will remain forever in $N$.  
Thus its  forward limit set   is a 
compact invariant subset of $N$. The argument used in the proof of Lemma \ref{maxinterval} shows that this is impossible. 
\end{proof}

Let then denote by $l(x)$ the length of the integral trajectory $\varphi_x$
of $\widehat{\xi}$  between $x$ and its endpoint on $\partial^- N$.
We consider then $\phi: \partial^+ N\times [0,1]\to N$ 
defined by: 
\[ \phi(x,t)= \varphi_x\left(\frac{t}{l(x)}\right)\]
Then $\phi$ is a diffeomorphism on its image. 
The only obstruction for  the surjectivity of $\phi$ 
is the existence of closed trajectories of $\widehat{\xi}$. 
Again, the argument used in the proof of Lemma \ref{maxinterval}  shows that 
there are no closed orbits of $\widehat{\xi}$ contained in ${\rm int}(N)$. 
This proves Proposition \ref{cylinder}. 
\end{proof}

\subsection{Another proof of Proposition \ref{compactleaf}}
We first recall that two points of a foliated manifold are {\em Novikov equivalent} if  
they belong to the same leaf or to a closed transversal. 
An equivalence class of points is called a {\em Novikov component} of the foliation. 
The following classical 
result is due to Novikov:
\begin{lemma}[\cite{Nov}]\label{novikov}
Assume that we have a foliation of a compact manifold. 
Then every Novikov component is either a compact leaf which does not admit 
any closed transversal or else an open codimension zero 
submanifold whose closure has finitely many boundary components 
consisting of compact leaves. 
\end{lemma}
An alternative way to complete the proof of Proposition \ref{compactleaf}, 
which does not use Proposition \ref{cylinder} 
is as follows. If  $H$ is generic and 
not everywhere non-constant, we proved above that we have an induced 
foliation of the compact submanifold  $N\subset \widehat{E}$. 
If some  super critical leaf $L_p$ were not compact, 
then  $L_p$ would be contained within some open Novikov component 
$V\subset N$.  Recall that a Novikov component $V$ is a saturated open submanifold   
such that every leaf contained in $V$ is non-compact  and approaches some boundary leaf of 
its closure $\overline{V}$.  Moreover, from the proof of Lemma \ref{totallygeodesic} 
all leaves in $V$, including the boundary ones in $\overline{V}$ should be totally geodesic.
We are in position to apply now the following result of Carri\`ere and Ghys (\cite{Car-Ghys}) on 
totally geodesic foliations:   
  
\begin{lemma}\label{cylinder2} 
The foliated manifold $\overline{V}$  is  diffeomorphic to a 
cylinder foliated as a product.
\end{lemma}
\begin{proof}
The codimension one foliation of $\overline{V}$ is transversely orientable and 
totally geodesic. According to (\cite{Car-Ghys}, Prop. II.2) such a foliation with a 
compact leaf is a fibration over an interval. 
\end{proof}

But this leads us to a contradiction, as we supposed $V$ to be a Novikov component. 
This ends the proof of Proposition \ref{compactleaf} in the case when $H$ is generic and 
not everywhere non-constant.

\begin{remark}
The arguments in the proof of Lemma \ref{totallygeodesic} 
also show that critical leaves $L_p$ for which 
$\nabla H|_p=0$ are totally geodesic. 
\end{remark}

\begin{remark}
The present proof does not extend to the case when $H$ is eventually constant. There exist 
for instance totally geodesic foliations all whose leaves are non-compact and have constant 
curvature, as the ones described in \cite{Car-Ghys} in dimension 3.       
\end{remark}

\subsection{Uniform geometry of curvature leaves}\label{Uniformgeom}
In this section we restrict ourselves to locally $1$-QC manifolds. 
With notation from above we set: 
\[ \alpha_V(p)=g(\nabla_{X_V}\widehat{\xi}_V, X_V), \; {\rm for} \; p\in V\cap (M-C)\]
Then curvature leaves are totally umbilical with principal curvatures equal to $\alpha_V$. 
Note that if $\xi|_{L_p}$ is orientable then the curvature leaf $L_p$ is $2$-sided and there exists a coherent choice of the lifts $\widehat{\xi}_V$ such that the value of $\alpha_V(x)$ is independent on $V$, for any $x\in L_p$. 
However, if $\xi|_{L_p}$ is non-orientable only the absolute value  $|\alpha_V(x)|$ is well-defined globally, 
as the curvature leaf is $1$-sided. Moreover, from (\cite{BP}, (3.5)) and (\cite{GM}, (3.4-3.5)(i)) 
$\alpha_V(x)$ and  $|\alpha_V(x)|$ are constant along a curvature leaf 
$L_p$ if  $\xi|_{L_p}$ is orientable and non-orientable, respectively. 

Recall that the {\em focal radius} of $L$ at  a point $p\in L$ is the smallest $r$ for which $r\widehat{\xi}|_p$
is a critical point of the exponential map and the {\em focal radius} $f_{L\subset M}$ is the smallest focal radius at $p$ among all $p\in L$.

\begin{lemma}\label{injectivity}
Let $M$ be a compact locally $1$-QC manifold. Then 
there exists a positive constant  $\gamma=\gamma(A)>0$ such that 
given a manifold  $M$ with curvature bounded by $A$ any codimension 1 submanifold $L\subset M$ satisfying  $\parallel II_{L} \parallel \leq A$ has both injectivity radius  $i_L$ and focal radius  $f_{L\subset M}$ (i.e. the distance to the closest focal point of $L$) 
bounded from below by $\gamma$.
\end{lemma}
\begin{proof}
The first claim is the content of (\cite{FG1}, Lemma 3). The second claim is well-known (see \cite{FG2}, Lemma 2.3). 
An explicit lower bound for the focal radius of $L$ was provided by Warner in (\cite{War}, Cor. 4.2), namely the focal radius of the totally umbilical 
hypersurface of principal curvature $\min \alpha$ within   the space 
of constant curvature $\max(H,N)$. 
\end{proof}

For every $p\in M-C$ we consider the metric disk $D_{\gamma}(p)\subset L_p$ of radius 
$\gamma$ within the leaf $L_p$ endowed with the induced metric. 
Let $i_{p, \gamma}$ denote the {\em normal injectivity radius} of $D_{\gamma}(p)$ inside $M$, i.e. the largest $r$ such that the restriction of the 
exponential map to the radius $r$ disk subbundle of the normal bundle of $L_p$ (restricted to  $D_{\gamma}(p)$) provides a diffeomorphism onto its image. 

We say that an arc $\zeta$ joining $p,q\in L$ is an {\em orthogeodesic} if it is a minimal geodesic in $M$ between $p$ and $q$ which is orthogonal 
to $L$ at its endpoints.  We denote by $o_{D_{\gamma}(p)\subset M}$ the smallest orthogeodesic length between points of $D_{\gamma}(p)$. 
It follows from (\cite{Herm}, Thm. 4.2) that the normal injectivity radius $i_{p,\gamma}$ is given by 
\[ i_{p,\gamma}= \min(f_{D_{\gamma}(p)\subset M}, \frac12 o_{D_{\gamma}(p)\subset M}) \]

Our next  goal is to show that $ i_{p,\gamma}$ is bounded from below. We will follow closely the reasoning of Corlette from 
\cite{Cor}. 

Let $d_L$ denote the Riemannian distance on the leaf $L$ and $d_M$ the Riemannian metric on $M$. 
We denote by $D^M_{r}(p)\subset M$ the radius $r$ metric disk on $M$ with center at $p$. 
It will be convenient for us to change the metric into a flat metric in a small neighborhood of a given point. The exponential map $\exp_p: T_pM \to M$ 
restricts to a diffeomorphism from  the tangent metric ball $\{v \in T_pM; \parallel v \parallel \leq \gamma\}$ onto $D^M_{\gamma}(p)$. 
Let $g^{\flat}$ denote the pull-back by $\exp_p^{-1}$ of the flat Riemannian metric on $T_pM$ and  $d^{\flat}_M$ the associated distance. 
According to (\cite{Cor}, Lemma 2.1, and the lines after its proof p.157) by choosing $\gamma$ sufficiently small, we can arrange that the distorsion 
between  $g^{\flat}$ and the initial metric $g$ be uniformly bounded: 

\begin{lemma}[\cite{Cor}, Lemma 2.1]\label{flatdeform}
For any $\lambda >1$, $\delta>0$ there exists some $R(\lambda,\delta)$ (smaller than the injectivity radius)  with the property that:
\begin{enumerate}
\item 
for any $p\in M$ the identity map  between $(D^M_{\gamma}(p), d_M)$  and  $(D^M_{\gamma}(p), d^{\flat}_M)$ is 
$\lambda$-Lipschitz, for any $\gamma \leq R(\lambda, \delta)$; 
\item the difference 
between the two Levi-Civita connections on these two metric balls is bounded by 
$\delta$ (in the $\mathcal C^0$ topology). 
\end{enumerate}
In particular, if $J\subset D^M_{\gamma}(p)$ is a submanifold with $\parallel II_J\parallel < A$ in the initial  metric $g$, then 
 $\parallel II^{\flat}_J\parallel < C(A,\lambda, \delta)$ in the flat metric $g^{\flat}$, where $ C(A,\lambda, \delta)$ is a continuous function 
 depending only on $A, \lambda, \delta$ such that $\lim_{\lambda \to 1, \delta\to 0}C(A,\lambda,\delta)=A$. 
\end{lemma}

The second ingredient needed is the following version of a theorem of Schwarz, as stated by Corlette in every dimension:

\begin{theorem}[\cite{Cor}, Thm. 2.3]\label{Schwarz}
Let $c$ be a curve in the Euclidean space joining the points $q$ and $q'$  and whose second fundamental form is 
bounded in norm by $\frac{1}{r}$, where $r$ is at least half the distance between $q$ and $q'$. 
Consider the circle of curvature $\frac{1}{r}$ which passes through $q$ and $q'$, which is separated 
by the pair of points into two arcs, say of lengths $\ell_1 \leq \ell_2$. 
Then, either the length of $c$ is smaller than $\ell_1$ or else it is larger than $\ell_2$. 
\end{theorem}

We are ready to prove now: 
\begin{lemma}\label{localnormal}
Let $M$ be a compact locally $1$-QC manifold.  Then there exists some $\gamma$ depending only on $A$ and the sectional curvature bounds 
such that the normal injectivity radius  $i_{p,\gamma}$ of curvature leaves in $E$ is  
uniformly bounded from below.   
\end{lemma}
\begin{proof}
Choose $\lambda, \delta$ such that $C(A,\lambda,\delta) < 2A$ and
$\lambda <1.1$.  Further change $\gamma$ into $\min(\gamma, \frac{1}{2A}, \frac{R(\lambda, \delta)}{3})$.

Consider an orthogeodesic $\beta$ of length $2t$ joining the points $q,q'\in D_{\gamma}(p)$. 
It is enough to consider that $\beta$ is the shortest one, namely the midpoint $u$ of $\beta$ 
has $d_M(u,x) \geq t$, for every $x\in D_{\gamma}(p)$. 
Let $c$ be a geodesic on $L$ joining $q$ and $q'$, which is therefore contained in $D_{\gamma}(p)
\subset D^M_{\gamma}(p)$. Observe that the second fundamental form of the geodesic $c$ is bounded in norm by $A$.  
Now, $D^M_{\gamma}(p)\subset D^M_{R(\lambda,\delta)}(u)$. 

Consider next the flat metric $g^{\flat}$ on $D^M_{R(\lambda,\delta)}(u)$. Then the metric sphere 
$\partial D^M_{t}(u)$ becomes an ordinary sphere of radius $t$ in the Euclidean space. 
The second fundamental form of $c$ in the flat metric $g^{\flat}$ is bounded by $2A$, by Lemma \ref{flatdeform}. 
The shortest arc of circle  joining $q$ to $q'$ of curvature equal to $2A$ has length 
$\ell_1=\frac{\arcsin(2At)}{A}$.

Suppose that $100t < \gamma$. Then $2At< \frac{1}{100}$ and so $\ell_1 < 2.2 t$. 
On the other hand the curve $c$ lies outside the sphere of radius $t$ and hence its length is 
at least the length of a great arc of circle joining $q$ and $q'$ on the sphere, namely $\pi  t$. 

From  the Schwarz Theorem \ref{Schwarz} it follows that the length of $c$ must be at least 
the length $\ell_2=\frac{\pi}{A}-\ell_1$ of the longest arc of circle  joining $q$ to $q'$ of curvature equal to $2A$. 
Observe now that: 
\[ \ell_2> \frac{\pi}{A}-2.2 t > \left(\pi -\frac{1}{40}\right)\frac{1}{A} > \frac{1.1}{A} > 2.2 \gamma\]
Therefore the curve $c$ in the flat metric  should have length greater than $2.2 \gamma$, contradicting the fact that 
$d_L(q,q') \leq 2\gamma$ and the Lipschitz constant is $\lambda < 1.1$. 

We derive that  $t \geq \frac{1}{100} \gamma$ and hence 
\[ i_{p,\gamma}= \min(f_{D_{\gamma}(p)\subset M}, \frac12 o_{D_{\gamma}(p)\subset M}) \geq \frac{1}{100} \gamma\]

\end{proof}

\begin{lemma}\label{Weinstein}
Let $M$ be a compact locally $1$-QC manifold. Then the curvature leaves  have uniformly bounded metric distorsion at small scales: 
there exists some constant $B>0$ such that for any leaf $L$ and any two points $p,q\in L$ with 
 $d_L(p,q) < \gamma$, we have $d_M(p,q)\geq B d_L(p,q)$.  
\end{lemma}
\begin{proof}
This follows directly from Lemma \ref{localnormal} and Weinstein's estimate of the distorsion for submanifolds 
with normal  injectivity radius  bounded from below (see \cite{Wein}, Cor. 3.5). 
\end{proof}

\begin{lemma}\label{bdvolume}
Assume that $M$ is locally $1$-QC and the metric is generic. 
\begin{enumerate}
\item Let $E$ be a connected component of $M-C$ with the property that $\lambda|_E$ is not identically zero. 
Then both the volume and the diameter of a curvature leaf $L\subset E$ are uniformly bounded 
from above and from below away from zero. 
\item  Let $E$ be a connected component of $M-C$ with the property that $\lambda|_E\equiv 0$. Then 
the volume of a curvature leaf   $L\subset E$ is uniformly bounded from below away from zero.  
\end{enumerate}
\end{lemma}
\begin{proof}
By Proposition \ref{dicho} $E$ is diffeomorphic to an interval bundle over some closed space form $P$. 
If $\lambda >0$ or $\lambda <0$ and dimension $n\geq 4$, then constant curvature metrics on $P$ are unique up to 
isometry. In particular, the volume of $L$ is a constant times $\lambda^{-\frac{n}{2}}$ and the diameter 
is also a constant times $\lambda^{-\frac{1}{2}}$. As $\lambda$ is bounded, the first claim follows. 
By Corollary \ref{sign} the only other possibility is that $\lambda\equiv 0$ on $E$.  
Moreover, if this is the case, then by Lemmas \ref{injectivity} the injectivity radius is bounded from below by $\gamma$ and hence 
the volume of $L$ is at least the volume of the  Euclidean ball of radius $\gamma$. 
\end{proof}

The next step is to construct uniform cylindrical neighborhoods around point of $M$. 
To this purpose, for every $p\in M-C$ we now consider the relative metric neighborhood $N_{\gamma}(D_{\gamma}(p))$, which is the image of the set $\{(q, \rho\xi); q\in D_{\gamma}(p), |\rho| \leq  \gamma\}\subset TM$ by 
the exponential map.  Since the exponential map is a diffeomorphism on its image, the later is a 
relative regular neighborhood with respect to its boundary. 
Let now $q\in M$ at distance $\frac{\gamma}{20}$ from $D_{\gamma}(p)$ the point which belongs to the 
orthogonal geodesic $\theta_p$ to $L_p$ issued from $p$ in the direction given by a lift $\widehat{\xi}$. 
  
\begin{lemma}\label{round}
There exists some constants $Q, R$ such that for $\gamma < R$ 
the metric disk $D^M(q,Q \gamma)$ is contained within $N_{\gamma}(D_{\gamma}(p))$. 
\end{lemma}
\begin{proof} 
We change the metric in $D^M(p, 4\gamma)$, with $\gamma$ small enough, to the flat metric such that 
the second fundamental form is bounded by $2A$, using Lemma \ref{localnormal}. 
We can also assume that the line field $\xi$ is defined on  $D^M(p, 4\gamma)$ and that in the flat 
metric the angle between $\xi|_p$ and $\xi|_y$ is smaller than $\frac{\pi}{10}$, for any $y\in D^M(p, 4\gamma)$. 
Let $p'$ be a point of the metric sphere $\partial D_{\gamma}(p)\subset L_p$ and $\theta$ be the 
orthogonal geodesic arc of length $\gamma$ issued from $p'$. 
By Lemma \ref{localnormal}  $N_{\gamma}(D_{\gamma}(p))$ is diffeomorphic to a cylinder 
$D_{\gamma}(p)\times [0,\gamma]$ whose lateral surface corresponding to 
$\partial D_{\gamma}(p)\times [0,\gamma]$ is the union of all arcs $\theta_{p'}$. Since the second fundamental form of $D_{\gamma}(p)$ is bounded, 
we can suppose from Lemma \ref{Weinstein}, by passing to a small $\gamma$ if needed, that 
$d_M(p,p')\geq B d(p,p')$. 
In the flat metric the geodesics arcs $\theta_p$ and $\theta_{p'}$ are line segments and their angle 
is smaller than $\frac{\pi}{10}$. Then elementary geometry arguments show that 
$d_M(q,\theta_{p'}) > \frac{9}{10}d_M(p,\theta_{p'})$, 
$d_M(p,\theta_{p'})\geq  \cos\left(\frac{\pi}{10}\right)d_M(p,p')$ 
and so $d_{M}(q, \theta_{p'}) >  \frac{B}{2} d(p,p')$. This implies that the distance between $q$ and the 
lateral surface of $\partial N_{\gamma}(D_{\gamma}(p))$ is at least $\frac{B}{2}\gamma$. On the other hand 
a similar argument shows that the distance between $q$ and the upper cap of the cylinder 
corresponding to $D_{\gamma}(p)\times \{\gamma\}$ is also at least $\frac{B}{2}\gamma$,  which implies our claim
for $Q=B/2$.  
\end{proof}

\subsection{Limit leaves for locally $1$-QC manifolds}

\begin{definition}
A  smooth {\em branched} hypersurface in $M$ is a codimension 1 immersion 
$f: X\to M$  with the following properties:
\begin{enumerate}
\item the source $X$ is the union of closed manifolds, namely compact manifolds with empty boundary;
\item $f$ has no triple points;
\item if $Y$ is any (topological) connected component of  the manifold $X$, then the 
restriction $f|_{Y}: Y\to M$ is an embedding;
\item double points are tangencies between images of distinct connected components of $X$. 
Specifically, if 
$(x_1,x_2)\in X\times X$ and $f(x_1)=f(x_2)$, then  there exist two distinct connected components $X_1, X_2\subset X$ 
of $X$ such that $x_1\in X_1$, $x_2\in X_2$ and  $f(X_1)$ and $f(X_2)$ are tangent at $f(x_1)=f(x_2)$.    
\end{enumerate}  
 The connected components of $X$ are called the {\em branches} at the source, their images in $M$ are the 
 branches of $f(X)$  and the set of double points in $M$ is the {\em branch} set of $f$.  

We extend this definition to  smooth {\em branched hypersurfaces with self-tangencies} in $M$ 
by allowing the restrictions $f|_{Y}: Y\to M$ to be immersions with self-tangencies instead of  
embeddings.  
\end{definition}
 
By abuse of language we will also refer to the image $f(X)$ as a branched hypersurface. With this convention 
the tangent space to a branched hypersurface is well-defined also at branch points. 

Recall that any smooth hypersurface could be defined locally as the graph of some smooth function. 
This permits to define limits of hypersurfaces as follows: 

\begin{definition}
A sequence of compact hypersurfaces $H_n\subset M$ {\em converges} to a smooth hypersurface $H_{\infty}$ if there is an open finite cover $\{U_i\}$ of $M$ by cylinders 
$U_i={\rm int}(D_i)\times (0,1)$ such that  for every $i$ and large enough $n$  the hypersurfaces $\overline{U_i}\cap H_n$ 
are  either empty or the graphs of some 
smooth functions $\psi_{i,n}:D_i\to [0,1]$ with the property that $\psi_{i,n}$ converges in the $C^k$-topology to $\psi_{i,\infty}$ 
for all $k\geq 2$.   

Further, the sequence of submanifolds $H_n$ {\em limits to} the finite union  $\cup_{i=1}^m H_{\infty,i}$ of submanifolds 
 $H_{\infty, i}$ if  $\Z_+=\cup_{i=1}^mJ_i$ such that each 
 sub-sequence $(H_{n})_{n\in J_i}$ converges to $H_{\infty,i}$. 
\end{definition}

We know that the distribution $\mathcal D$ is integrable on $M-C$. Our aim is to show that a weaker version of integrability 
holds on the closure $\overline{M-C}$. Our main result in this section is:

\begin{proposition}\label{integrability}
Let $M$ be a  locally $1$-QC manifold, which is assumed to be conformally flat when $n=3$ and let $\xi$ denote a  distinguished line field extension to  an open set containing $\overline{M-C}$.  
\begin{enumerate}
\item Then for every $q\in \overline{M-C}\cap C$  there exists a limit leaf $L_q$ for $\mathcal D$. 
\item If $\xi$ is orientable then the limit leaf $L_q$ passing through $q\in \overline{M-C}\cap C$ 
is a branched hypersurface.  Moreover, each branch of the limit leaf $L_q$  is a totally umbilical compact submanifold of constant 
intrinsic sectional curvature $\lambda$ tangent to  (the extension to $\overline{M-C}$ of) the horizontal distribution 
$\mathcal D$. 
\item When $\xi$ is non-orientable, the limit leaf  $L_q$ is a branched hypersurface 
as above,  possibly with self-tangencies. 
\item Every limit leaf $L_q$ is contained within $ \overline{M-C}\cap C$. 
\end{enumerate}
\end{proposition}

{\em Outline of the proof of Proposition \ref{integrability}}: We have to analyze the behaviour of leaves $L_{p_i}$, when $p_i$ approach 
a point $q\in C\cap \overline{M-C}$. We first consider the problem locally. Using the bounded geometry of leaves from section \ref{Uniformgeom} 
we shall consider specific neighborhoods which are products of a disk in a leaf and an interval. The size of these neighborhoods is bounded from below and chosen small enough such that leaves intersections are graphs of functions. The Arzela-Ascoli argument provides the existence  
of a lower and an upper limit for the corresponding sequence of functions, whose associated graphs form the two local branches passing through 
$q$. An uniform lower bound for the normal injectivity radius of leaves permits to show that each leaf $L_{p_i}$ intersects only once (or twice in the non-orientable case) a specific neighborhood. This permits to pass from a local limit leaf to a global 
limit leaf which should be compact since it can be approached everywhere by some sequence of compact leaves  $L_{p_i}$.

\begin{proof}
We use notation from the proof of Lemma \ref{extension-lambda}. 
Since smooth, $\alpha_V$ is bounded along with all its derivatives on any open subset $V'$ with closure $\overline{V'}$ contained in $V$. 
  
Further there exist two finite open ball coverings  $\{V_i\}$ and $\{V'_i\}$ of $\overline{M-C}$ such that 
$\overline{V'_i}\subset V_i$, for every $i$. Then $\alpha_{V'_i}$ are all bounded, so that there exists some 
constant $A$ such that for every $i$: 
\[ |\alpha_{V'_i}| \leq A.\]
We can take $A$ large enough such that $\max(|H(x)|, |N(x)|) \leq A$, for all $x\in M$. 
Since every curvature leaf $L\subset M-C$ is totally umbilical of principal curvatures $|\alpha|$, the previous inequality shows that 
the second fundamental form $II_{L}$ has  uniformly bounded norm: 
\[ \parallel II_{L} \parallel \leq A.\]

By hypothesis there exists a local conformal change of the metric to a flat metric, in particular we can assume that 
there is a flat metric on each $V_j$. Changing the original metric on $V_j'$ to the flat one amounts 
to changing the constant $A$ above.    Now, we can therefore compare tangent vectors at different points of $V_j$. 
As $\xi|_{V'_j}$ is a smooth vector field defined on $V'_j$, there exists $\delta(\varepsilon)>0$ such that 
whenever $x,y\in V'_j$ are at distance at most $\delta$ apart then  the angle between  
$\xi|_x$ and $\xi|_y$ is at most $\varepsilon$.

Our key ingredient is the following lemma which seems to be well-known to the specialists: 
\begin{lemma}\label{convergence}  
Suppose that $C\neq \emptyset$. 
Let $q\in \overline{M-C}\cap C$ and  $p_i\in M-C$ such that $\lim_{i\to \infty}p_i=q$. Then there exists an open 
neighborhood $U_q$ of $q$ in $M$ such that: 
\begin{enumerate}
\item The size of $ \overline{U_q}$ is uniformly bounded from below, namely it contains an interval bundle over a metric disk having both radius and 
height larger than some constant independent on $q$. 
\item The intersection of the curvature leaf $L_{p_i}$ with $U_q$ is connected if $\xi$ is orientable and has at most 
two connected components, if the restriction of $\xi$ is to the connected component of $M-C$ containing $p_i$ is non-orientable. 
\item Let $(L_{p_i}\cap \overline{U_q})^{\circ}$ denote the connected component of $L_{p_i}\cap \overline{U_q}$ containing $p_i$. 
Then the sequence $(L_{p_i}\cap \overline{U_q})^{\circ}$ limits to a branched hypersurface 
$(L_q\cap \overline{U_q})^{\circ}$, called a piece of the local limit leaf at $q$. Moreover,  $(L_q\cap \overline{U_q})^{\circ}$ is the union of two 
codimension 1  submanifolds of $M$ passing through $q$ which are  tangent to $\mathcal D$ which are called (local) branches. 
\item For any sequence of points $p_i\in M-C$ with $\lim p_i=q$,  
the sequence $(L_{p_i}\cap \overline{U_q})^{\circ}$ limits to $(L_q\cap \overline{U_q})^{\circ}$ or a branch of it.   
\item Each piece  $(L_q\cap \overline{U_q})^{\circ}$ of the local limit leaf at $q$ is contained in $(\overline{M-C})\cap C$. 
\item The sequence of curvature leaves $L_{p_i}\cap \overline{U_q}$ limits to a local limit leaf $L_q\cap \overline{U_q}$. 
Moreover, the local limit leaf  $L_q\cap \overline{U_q}$ is the union of at most 
two pieces of local limit leafs. 
\end{enumerate}
\end{lemma}
\begin{proof}[Proof of Lemma \ref{convergence}] 
According to Lemmas \ref{injectivity} and \ref{localnormal} there is $\gamma$ such that any leaf $L$ 
has both injectivity radius  $i_L$ and focal radius $f_{L\subset M}$ bounded from below by $\gamma$ while the normal injectivity radius 
$i_{p, \gamma}$ is bounded from below.  Now, as  $\parallel II_{L} \parallel \leq A$ and each leaf is of constant curvature $\alpha$, uniformly bounded, it follows that each the embedding  of each leaf $L\to M$ has uniformly bounded local metric distorsion. 
Change now $\gamma$ into  $\min(\gamma, i_{p,\gamma})$.  
It follows that $i_{p,\gamma}\geq\gamma$.

Consider now the image $N_{\gamma}(D_{\gamma}(p))$ of 
$\{(q, \rho\xi); q\in D_{\gamma}(p), |\rho| < \gamma\}\subset TM$ by the exponential map.  
By above the exponential map is a diffeomorphism on its image, which is a 
relative regular neighborhood with respect to its boundary. 
Moreover, as $D_{\gamma}(p)$ is diffeomorphic to a $(n-1)$-dimensional disk, 
$\xi|_{D_{\gamma}(p)}$ is a trivial line bundle and hence  $N_{\gamma}(D_{\gamma}(p))$ is diffeomorphic to 
$D_{\gamma}(p)\times [-\gamma, \gamma]$.

We can assume that  $q\in \overline{M-C}\cap C\cap V'_j$, $p_i\in (M-C)\cap V'_j$ and  
 $d(p_i,q) < \gamma/20$, for all $i$, so that $q\in {\rm int}(N_{\gamma}(D_{\gamma}(p_1))$.  
Consider now the cylindrical neighborhood  $\overline{U_q}= N_{\gamma}(D_{\gamma}(p_1))$ of $p_1$ which is diffeomorphic to  
$D_{\gamma}(p_1)\times [-\gamma, \gamma]$. 
Note that there is a projection on the first factor $\pi: \overline{U_q}\to D_{\gamma}(p_1)\subset L_{p_1}$. Specifically, 
$\pi(z)\in D_{\gamma}(p_1)$ is the closest point to $z$. By our choice of $\gamma$, the map $\pi$ is well-defined (see Lemma \ref{localnormal}). 
Moreover, we can suppose that $q, p_i\in {\rm int}(\overline{U_q})$, for all $i$ and that $\overline{U_q}\subset V'_j$. 

We will suppose henceforth that $\xi$ is defined on a $\gamma$-metric neighborhood $\Omega$ of $\overline{M-C}$, and in 
particular on every  $\overline{U_q}$, for $q\in \overline{M-C}\cap C$.

By taking a smaller $\gamma$ if needed, we can assume that  
$\overline{U_q}$ is included in a metric ball of radius $\delta(\pi/10)$, namely that 
$\xi|_x$ has only a small  angle variation for all $x\in \overline{U_q}$ and moreover 
$d(p_i,p_1)<\gamma/10$. 

Now, $U_q$ satisfies the requirements of the first claim in the Lemma, by construction.

From Proposition \ref{dicho} every connected component $E$ of $M-C$ is an interval bundle over a compact space form $P$, as $\widehat{E}$ is a cylinder. 
Therefore the complement of the lift of every curvature leaf in $\widehat{E}$ has two connected components.  
Then (\cite{Hae}, Lemme 2, p.385)  shows that under these conditions any transversal arc will cut a leaf of $\widehat{E}$ in at most one point. 
On the other hand, note that a geodesic arc issued from a point of $D_{\gamma}(p_1)$ and orthogonal to $L_{p_1}$ should be 
transversal to all leaves that it intersects in $\overline{U_q}$, because of the small  angle variation assumption. 
If $\xi$ is orientable, then the intersection of curvature leaves $L_{p_i}$ with $\overline{U_q}$ to should be connected, 
as otherwise orthogonal geodesics will intersect twice the same leaf. If $\xi$ is non-orientable, the same argument works for the pull-back 
of  $\overline{U_q}$ to $\widehat{E}$ should be connected and hence the intersection of curvature leaves $L_{p_i}$ with $\overline{U_q}$ 
have at most two connected components. This proves the second item of our Lemma.

We further claim that  the connected  component of $(L_{p_i}\cap \overline{U_q})^{\circ}$ is 
the graph of a smooth function $y_i:D_{\gamma}(p_1)\to [-\gamma,\gamma]$. 
This is clear for $(L_{p_1}\cap \overline{U_q})^{\circ}$ and hence for all leaves $(L_{p_i}\cap \overline{U_q})^{\circ}$ 
which are close enough to it. We have to prove that this holds for 
all points $p_i$ with $d(p_i,p_1) <\gamma/10$. 

A compact hypersurface $L\cap\overline{U_q}$ is the graph of some smooth function 
$D_{\gamma}(p_1)\to [-\gamma,\gamma]$ if  and only if the 
projection $\pi:\overline{U_q}\to D_{\gamma}(p_1)$ restricts to a diffeomorphism  
$\pi|_{L}:L \cap \overline{U_q}\to D_{\gamma}(p_1)$.
The key point is to show that $\pi|_{L}$ is a bijection, as the smoothness of its inverse will follow 
from the smoothness of $L\cap\overline{U_q}$.  

If $\pi|_{L_{p_i}}$ were not surjective, then there are two possibilities: 
\begin{itemize}
\item[-] 
either $(L_{p_i}\cap \overline{U_q})^{\circ}$ would intersect the top or the base of the cylinder $\overline{U_q}$, namely the
image of $D_{\gamma}(p_1)\times\{-\gamma,\gamma\}$;
\item[-] or else a boundary point of the image $\pi|_{L_{p_i}}((L_{p_i}\cap \overline{U_q})^{\circ})$  would be a 
critical value of $\pi|_{L_{p_i}}$ lying in ${\rm int}(D_{\gamma}(p_1))$. 
\end{itemize}

In the first case we can find a curve  contained in $(L_{p_i}\cap \overline{U_q})^{\circ}$ joining $p_i$ to some point $z\in D_{\gamma}(p_1)\times\{-\gamma,\gamma\}$. By the mean value theorem there is some point $r$ on this curve where the angle between $\xi|_{p_1}$ and $\xi|_r$ is at least $\arcsin\left((\sqrt{199}-1)/20\right)> \pi/10$, contradicting our hypothesis on the small variation of the angle in the (flat) chart $U_q$. 

In the second case $(L_{p_i}\cap \overline{U_q})^{\circ}$ only intersects non-trivially the lateral surface 
$\partial \overline{U_q}-D_{\gamma}(p_1)\times\{-\gamma,\gamma\}$ of the  boundary 
$\partial \overline{U_q}$.   
Now,  if $x$ is a critical point of $\pi|_{L_{p_i}}$ then $\xi|_x$ should be parallel to  the tangent space 
$T_{\pi(x)}D_{\gamma}(p_1)$ and hence orthogonal to $\xi|_{\pi(x)}$, contradicting our assumptions on 
the small variation of the angle. 
We conclude that $\pi|_{L_{p_i}}$ is surjective. 

Next,  if $\pi|_{L_{p_i}}$ were not injective when restricted to a connected component of  $(L_{p_i}\cap \overline{U_q})^{\circ}$,   
then there would exist a smooth path in $(L_{p_i}\cap \overline{U})^{\circ}$ joining two points with the same image under $\pi|_{L_{p_i}}$. The restriction of $\pi|_{L_{p_i}}$ to this path must then have a critical point and, as above, this would provide a 
point $x$ of $(L_{p_i}\cap \overline{U_q})^{\circ}$ for which $\xi|_x$ is orthogonal to $\xi|_{\pi(x)}$. This would contradict 
our assumption on the small variation of the angle. Our claim follows. 
 
Observe that $y_i:D_{\gamma}(p_1)\to \R$ are smooth  bounded functions, as $ |y_i|\leq \gamma$. 
Their derivatives are also bounded as $\xi|_{x}$ is close 
to $\xi|_{q}$, for any $x\in \overline{U_q}$, by the small angle variation assumption. 
Specifically,  assume that we are in a flat chart containing $\overline{U_q}$ obtained by a conformal change of metric. 
We pick up coordinates $\{x_1,x_2, \ldots,x_{n-1},y\}$, the first $n-1$ coordinates corresponding to  $D_{\gamma}(p_1)$. 
If the point $x\in \overline{U_q}$ belongs to the image of $y_i$, then we have: 
\[ \sin^2\measuredangle(\xi|_{x},\xi|_{q})=\frac{\sum_{j=1}^{n-1}\left| \frac{\partial y_i}{\partial x_j}\right|^2}
{1+\sum_{j=1}^{n-1}\left| \frac{\partial y_i}{\partial x_j}\right|^2} <\sin^2\frac{\pi}{10} < 0.1\]
from which we derive the following upper bound for the first derivatives of $y_i$:  
\[ \sum_{j=1}^{n-1}\left| \frac{\partial y_i}{\partial x_j}\right|^2 < \frac{\sin^2\frac{\pi}{10}}{1+ \sin^2\frac{\pi}{10}}<0.1\]
In the same flat chart the second fundamental form of the hypersurface given by the image of $y_i$ has the matrix:  
\[ II_{jk}= \frac{1}{1+\sum_{j=1}^{n-1}\left| \frac{\partial y_i}{\partial x_j}\right|^2}\cdot \frac{\partial^2 y_i}{\partial x_j \partial x_k}\]
Since the norm of the second fundamental form of $L_{p_i}\cap \overline{U_q}$ is uniformly bounded by $A$ we derive the following bound for the second derivatives of $y_i$: 
\[ \sum_{j,k=1}^{n-1}\left| \frac{\partial^2 y_i}{\partial x_j\partial x_k}\right|^2 < 1.1\cdot A\]
Alternatively, recall that $II$ is diagonal in the 1-QC metric, because leaves are totally umbilical  and the principal curvatures 
equal $\alpha$. The entries of $II$ in the flat chart can be expressed  in terms of the entries of $II$ in the 1-QC metric (namely $0$ and $\alpha$) and the real function specifying the conformal change in the metric.  Since all derivatives of $\alpha$ are bounded in $\overline{U_q}$  higher derivatives of $y_i$ are also bounded.

By the Arzela-Ascoli theorem the sequence $y_i$  has a sub-sequence which converges to a smooth function $y_{\infty}: D_{\gamma}(p_1)\to [-\gamma,\gamma]$ in the $C^{k}$ topology, for every $k\geq 2$. 

Now, the leaves $L_{p_i}$ are pairwise disjoint and hence 
there is a total order among the  functions $y_i$. Therefore, there are only two possible limits of sub-sequences of $y_i$, namely 
$y_{\infty}^+=\lim\sup y_i$ and $y_{\infty}^-=\lim\inf y_i$. 
We say that the sequence $(L_{p_i}\cap \overline{U_q})^{\circ}$ is increasing or decreasing if the sequence of functions $y_i$ is increasing or decreasing, respectively. 

Denote by $(L_q^{\pm}\cap \overline{U_q})^{\circ}$ the graph of the limit function $y^{\pm}_{\infty}$ and call them the {\em local branches} of the local limit leaf. We say that the sequence $L_{p_i}^{\circ}\cap \overline{U_q}$ is {\em below} $q$ if 
the sequence $y_i$ is increasing and {\em above} $q$ if the sequence $y_i$ is decreasing. 
Each local  branch is a smooth  totally umbilical hypersurface of principal curvature $\alpha(q)$. 
Since $\xi$ is everywhere orthogonal at $L_{p_i}$, by passing to the limit we find that each $(L_q^{\pm}\cap U_q)^{\circ}$ is a codimension 1 submanifold which is tangent to the distribution $\mathcal D$ at each point $x\in (L_q^{\pm}\cap \overline{U_q})^{\circ}$.
  
Set $(L_q\cap \overline{U_q})^{\circ}=(L_q^{+}\cap \overline{U_q})^{\circ}\cup (L_q^{-}\cap \overline{U_q})^{\circ}$ for the union of the two local branches.  
The local branch set is the set of tangency points between 
the two local branches $(L_q^{\pm}\cap \overline{U_q})^{\circ}$. The branch set could be all of $(L_q\cap \overline{U_q})^{\circ}$, in which case the two local branches coincide, or just one point. 

Eventually observe that $(L_{p_i}\cap \overline{U_q})^{\circ}$ limits to $(L_q\cap \overline{U_q})^{\circ}$ since  the later is 
the union of the  graphs of the functions $\limsup y_i$ and $\liminf y_i$.

We obtained above a piece $(L_q\cap \overline{U_q})^{\circ}$ of the limit leaf at $q$. We will need now to show that 
the piece $(L_q\cap \overline{U_q})^{\circ}$ does not depend on the choice of the sequence $p_i\in M-C$ with $\lim p_i=q$. 
 It is enough to consider the case when the sequences $p_i$ and $p_i'$ are both increasing
 with respect to the order defined in the proof of Lemma \ref{convergence}.  As their limit is the same the sequences of functions $y_i$ and $y_i'$ 
 associated to the sequences $p_i$ and $p_i'$, respectively, are comparable, namely for every $i_1$ there exists some $i_2$, $j_1,j_2$ such that 
 $y_{i_1} \leq y'_{j_1} \leq y_{i_2}\leq y'_{j_2}$. Then $\lim y_i=\lim y'_i$, which implies our claim and hence item 4 of our Lemma. 
 
 Further,  if $q'\in M-C \cap (L_q\cap \overline{U_q})^{\circ}$, then $(L_q\cap \overline{U_q})^{\circ}=(L_{q'}\cap \overline{U_q})^{\circ}$, so that 
 $q\in L_{q'}$. But the curvature leaf $L_q'$ is complete with its intrinsic metric, according to Reckziegel's Theorem  \ref{Reckziegel} (see Lemma \ref{Reck}) and $L_{q'}\subset M-C$. This leads to a contradiction. Thus $(L_q\cap \overline{U_q})^{\circ}\subset C$.   On the other hand  every point of  $(L_q\cap \overline{U_q})^{\circ}$ is a limit point of  $(L_{p_i}\cap \overline{U_q})^{\circ}$, so that $(L_q\cap \overline{U_q})^{\circ}\subset \overline{M-C}$.

 We can define now the local limit leaf  $L_q\cap \overline{U_q}$, in the case when $L_{p_i}\cap \overline{U_q}$ is not connected, for large enough $i$. 
 It suffices to consider the set $p'_i\in M-C\cap U_q$ with the property that 
 $L_{p_i}\cap \overline{U_q}=(L_{p_i}\cap \overline{U_q})^{\circ} \sqcup (L_{p'_i}\cap \overline{U_q})^{\circ}$ and 
 $\lim p'_i=q'\in C$. Then we set  $L_q\cap \overline{U_q}= (L_q\cap \overline{U_q})^{\circ} \cup (L_{q'}\cap \overline{U_q})^{\circ}$. 
 It follows from above  that $L_{p_i}\cap \overline{U_q}$ limits to $L_q\cap \overline{U_q}$. 
 Such a local limit leaf $L_q\cap \overline{U_q}$ is obviously unique and independent on the choice of the sequence $p_i$. 
 Moreover, by uniqueness if $q'\in C$ is such that $q'$ is a limit point of $L_{p_i}\cap U_q$, then 
 $(L_{q'}\cap  U_{q'})\cap U_q=(L_{q}\cap  U_q)\cap U_{q'}$.  
\end{proof}

\begin{lemma}\label{normalinjectivity}
Assume that $M$ is locally $1$-QC. Let $E$ be a component of $M-C$ for which the restriction 
$\xi|_{E}$ is orientable. Then the normal injectivity radius of every curvature leaf $L\subset E$ is bounded below 
by some constant $\gamma$ depending on $A$ and the bounds on the sectional curvature. 
\end{lemma}
\begin{proof}  
Consider a normal orthogeodesic joining $p,q\in L_p$. We consider the metric neighborhood $N_{\gamma}(D_\gamma(p))$ of $p$. 
Then the argument used above to prove Lemma \ref{convergence}.(2)  shows that  $L_p\cap N_{\gamma}(D_\gamma(p))$ is connected. 
In particular $d_M(p,q) \geq \gamma$. This provides an uniform lower bounds for the length of an orthogeodesic. Using 
(\cite{Herm}, Thm.4.2) and the lower bound for the focal radius provided by Lemma \ref{injectivity} we obtain the claim. 
\end{proof}

We now define an equivalence relation, called {\em confluence}, 
for branches of local limit leaves as follows. 
First, if $U_q\cap U_{q'}\neq \emptyset$
we  say that $L^+_{q'}\cap  U_{q'}$ and $L^+_{q}\cap  U_{q}$  {\em have a common extension} if 
 $(L^+_{q}\cap  U_{q})\cap U_{q'}=(L^+_{q'}\cap  U_{q'})\cap U_{q}$. 
Then  $L^+_{q'}\cap  U_{q'}$ and $L^+_{q}\cap  U_{q}$ are 
called {\em confluent} if there exists a  finite sequence of branches of local limit leaves 
starting at $L^+_{q}\cap  U_{q}$ and ending with $L^+_{q'}\cap  U_{q'}$, where 
consecutive terms have a common extension. 

We now define {\em the branch  $L_q^+$ of the limit leaf passing through $q$} as the union of all branches of local limit leaves confluent with $L^+_{q}\cap  U_{q}$.    
By definition $L_q^+$ is an immersed smooth codimension one submanifold of $M$ with possible self-tangencies, 
which is contained within $\overline{M-C}\cap C$. 
Eventually, we define {\em the limit leaf} $L_q$ as the union of all branches of 
all local limit leaves confluent to some branch of $L_q\cap U_q$.

\begin{lemma}\label{compactlimitleaf}
Each branch $L^+_q$ of a limit leaf is compact. 
\end{lemma}
\begin{proof}
We want to show that limit leaves can be approached globally by leaves. 
The following property of the confluence relation will be the key ingredient of the proof. 
We suppose that: 
\begin{enumerate}
\item $L_{p_i}\cap U_q$ limits to $L_q^+\cap U_q$, where $p_i$ are below $q$,  
and $L_{p'_i}\cap U_{q'}$ limits to $L_{q'}^+\cap U_{q'}$; 
\item  $U_q\cap U_{q'}\neq \emptyset$; 
\item  $L_q^+\cap U_q$ and   
and $L_{q'}^+\cap U_{q'}$ have a common extension. 
\end{enumerate}
Then, we claim that there is a sub-sequence $L_{p_i''}$ of the sequence 
of leaves $L_{p_i}, L_{p'_i}$ with the property that    
$L_{p_i''}\cap (U_q\cup U_{q'})$ limits to $(L^+_q\cap U_q) \cup (L^+_{q'}\cap U_{q'})$.

Observe that $p_i$ are below $q$ and $p_i'$ are below $q'$. 
Our claim is a consequence of an analogous result on sequences of real functions $y_i$, 
having the property that $y_{i+1}-y_i$  have constant sign. 
Specifically, assume we have two increasing 
sequences of real functions $y_i,y'_i:U\cup V\to \R$ such that their restrictions to the 
open relatively compact sets $U$ and $V$, 
with $U\cap V\neq\emptyset$, converge to the restrictions of a function 
$z:U\cup V\to \R$ defined on $U\cup V$, namely 
${y_i}|_U\to z|_U$ and ${y'_i}|_V\to z|_V$. Moreover, suppose that  
$y_i-y'_j$ have constant sign, for all $i,j$. Then there exists a sub-sequence 
$y_i''$ of the sequence obtained by putting together $y_i, y'_i$ 
such that ${y_i''}|_{U\cup V}\to z|_{U\cup V}$. In fact it suffices to take 
$y_i''=\max(y_i,y'_i)$, which coincides with either $y_i$ or $y'_i$ on $U\cup V$. 
To allow a description of general hypersurfaces which might not be globally graphs we should   
extend this result to multi-valued functions. Consider now that 
${y_i}|_U$, ${y'_i}|_V$, $z|_U$ and $z|_V$ are now usual functions, but ${y_i}|_V$ and ${y'_i}|_U$  
are only required to be multi-valued functions whose graphs correspond to some smooth hypersurface. 
The argument above then shows that the graphs of the  possibly multi-valued functions 
$\max(y_i,y'_i)|_{U\cup V}$ converge to the graph of $z|_{U\cup V}$ in the ${\mathcal C}^0$-topology:  
if ${y_i}|_{U\cap V} > {y'_i}|_{U\cap V}$, then $\max(y_i,y'_i)=y_i$ and   
${y'_i}|_V < {y_i}|_V < z_V$. Thus the limit of any sub-sequence ${y_i}|_V$ as above is $z|_V$.  
However, if we assume that all derivatives of ${y'_i}|_V$ and  ${y_i}|_V$ are continuous  
then  ${y_i}|_V$ should be also usual functions, as otherwise their first derivative would be infinite 
in some direction. In our case of branches of limit leaves this assumption is satisfied since 
all leaves intersecting a open distinguished neighborhood $U_q$ have almost horizontal tangent space.

By induction on the number of local branches 
we obtain that, up to extracting a sub-sequence (and possibly changing the initial 
sequence $p_i$ by another sequence with the same properties) $L_{p_i}$ limits from below to $L_q^+$.

Each branch  $L_q^+$ is complete with respect to the induced metric. To see that, it suffices 
to note that the metric disk $D_{Q\gamma}(q)$ within the branch 
$L_q^+$ is contained within $L^+_{q}\cap  U_{q}$, for every $q$, 
according to Lemma \ref{round}. 
To prove that  $L_q^+$ is compact, it is therefore enough to show that it is a proper leaf. 

We assume that in our construction above the height $\gamma$ of each $U_q$ 
is smaller than $\frac{i_{p,\gamma}}{3}$. If $U_q=N_{\gamma}(D_{\gamma}(p_1))$  we denote by $U'_q$ the corresponding metric neighborhood of 
double height, namely $N_{2\gamma}(D_{\gamma}(p_1)$. 
Since $L^+_q$ is covered by some union $\cup U_{q'}$,  its closure  $\overline{L^+_q}$ in $M$ is 
covered by the corresponding union $\cup U'_{q'}$. By compactness we can extract a finite covering $U'_{q_i}$ of  $\overline{L^+_q}$. 
If $L^+_q$ were non-compact, then there would exist some neighborhood 
$U'_{q'}=N_{2\gamma}(D_{\gamma}(p'_1))$ which intersects infinitely many times $L^+_q$. 
But  $U'_{q'}$ has the shape of a cylinder, so that there are geodesics issued from $D_{\gamma}(p'_1)\subset L_{p'_1}$ orthogonal to $L_{p'_1}$  
which  intersect infinitely many times $L^+_q$. 
Recall from above that  every point  of $L^+_q$ can be approached 
arbitrarily closely by the leaves $L_{p_i}$. It follows that there exist leaves $L_{p_i}$ which intersect more than 100 times some 
orthogonal geodesic within $U'_{q'}$. 
This contradicts the fact that any transverse arc in $M-C$ should intersect at most twice every leaf of our foliation within a small neighborhood 
(see \cite{Hae}, Lemme 2, p.385 and Lemma \ref{convergence}.(2) above). 
\end{proof}

From above $L_q$ is the image of a codimension 1 immersion of some union $X$ of closed manifolds, 
with only tangency or self-tangency points. The norm of the second fundamental form of every branch is also bounded by $A$, 
as this bound is valid for any local branch $L_q^{\pm}\cap \overline{U_q}$.

When $\xi$ is orientable, then every branch is embedded, as the 
normal injectivity radius is bounded from below by Lemma \ref{normalinjectivity}  and hence there are not 
self-tangency points. Thus $L_q$ is a branched hypersurfaces. 
Further $L_q\subset \overline{M-C}\cap C$ and each branch is a totally umbilical hypersurface with constant sectional 
curvature $\lambda(q)$. 
When $\xi$ is non-orientable, then the lift $L_q$ to the double cover $\widehat{\Omega}$ is a branched  hypersurface, as claimed.   
 This proves Proposition \ref{integrability}. 
\end{proof}

An immediate corollary of the proof is the following:
\begin{lemma}\label{Hausdorff}
Let $q\in \overline{M-C}\cap C$ and  $p_i\in M-C$ such that $\lim_{i\to \infty}p_i=q$. Then, up to passing to a sub-sequence, 
the limit of $L_{p_i}$  in the Hausdorff topology is some branch of the limit leaf $L_q$. More specifically, if 
$p_i$ are below $q$ then the limit of $L_{p_i}$  is the branch $L_q^+$. 
\end{lemma}
Recall that the {\em Hausdorff distance} between the  
non-empty compact subsets $X,Y\subset M$ is 
given by $\max(\sup_{x\in X} \inf_{y\in Y} d(x,y), \sup_{y\in Y}\inf_{x\in X} d(x,y))$, where 
$d$ is the distance in $M$. It is well-known that the space of compact subsets of $M$ equipped with the Hausdorff metric 
is complete.

\subsection{Closure of non-compact globally $1$-QC components}  
We denote by ${\rm Fr}(E)={\rm Fr}(\overline{E})=\overline{E}\setminus E$ the {\em frontier} of a subset $E\subset M$. 
An {\em end} of a topological space $X=\cup_{n=1}^{\infty}K_n$ which is the ascending union of compact subspaces $K_n$ 
is an infinite sequence $V_1\supset V_2 \supset V_3 \supset \cdots$ with $\cap_n V_n=\emptyset$, where 
each $V_n$ is a connected component of $X-K_n$. When $X$ is a path-connected CW-complex ends correspond to 
equivalence classes of proper continuous maps $\R_+\to M$, 
up to proper homotopy of their restrictions to the subset of positive integers. 
For example $\R$ has two ends, $\R^2$ has one end while 
the infinite complete binary tree has uncountably many ends. 
 
Proposition \ref{integrability} shows that limit leaves are branched surfaces hence union of embedded (or immersed) 
compact hypersurfaces. Here we will obtain a more precise statement by showing that the number of branches 
in the frontier of a connected component $E$ of $M-C$ is either one or two. 
 
\begin{proposition}\label{bdy}
Let $M$ be a locally $1$-QC manifold and $E$ a non-compact connected component of $M-C$. Assume that $\xi|_E$ is orientable. 
Then  $\overline{E}$ is  either diffeomorphic to a space form cylinder or else homeomorphic to the result of identifying 
two subsets of the two boundary components of a space form cylinder.    
\end{proposition}
\begin{proof}
By Proposition \ref{dicho} $E$ is diffeomorphic to a cylinder $P\times \R$ over some closed 
space form $P$. 
It makes sense then to say that a sequence of curvature leaves which goes to infinity  
belongs to one or another of the ends of $\R$.

Let $q\in {\rm Fr}(E)$. By Proposition \ref{integrability} there exists a limit leaf $L_q$ passing through $q$, which is a compact branched hypersurface everywhere tangent to $\mathcal D$. Now $L_q\cap \overline{E}\subset {\rm Fr}(E)$ as $L_q\cap (M-C)=\emptyset$. 
Note that an arbitrarily small backward trajectory of $\widehat{\xi}$ issued from $q$ 
cannot be contained in $L_q$, as $\widehat{\xi}$ is orthogonal to $L_q$.  Then, according to Lemma \ref{Hausdorff}   any sequence of curvature leaves belonging to a 
given end of $E$ converges to a branch of $L_q\cap \overline{E}$. 
Thus  $L_q$ has at most two branches, corresponding to the two ends of $E$. 
Therefore ${\rm Fr}(\overline{E})$ is a compact branched hypersurface with at most two branches. 

If there is only one branch for each end, then $\overline{E}$ is a compact manifold endowed with a foliation with compact leaves and 
hence diffeomorphic to a cylinder. Otherwise, $\overline{E}$ is homeomorphic to the 
result of gluing together the two ends of the cylinder along 
the branch locus, namely the quotient of $P\times [0,1]$ by the equivalence relation 
$(x,0)\sim (x,1)$, for every branch point $x$. 
\end{proof}

  

\begin{proposition}\label{Ibundles}
If $E$ is a non-compact connected component of $M-C$ and $\xi|_E$ is non-orientable,  
then $\overline{E}$ is diffeomorphic to the image of 
an interval bundle with one boundary component by an immersion with self-tangencies along the 
boundary. 
\end{proposition}
\begin{proof} 
Let $\xi$ be defined on the open set  containing an open subset  $\Omega\supset \overline{M-C}$ with compact closure. 
Let then  $\overline{\widehat{E}}$ be the closure of  $\widehat{E}$ within $\widehat{\Omega}$.  
By Proposition \ref{bdy}  $\overline{\widehat{E}}$ is the image of a cylinder 
$\widehat{P}\times [0,1]$ by an immersion with possible self-tangencies between 
the two boundary components, where $\widehat{P}$ is a closed manifold. 
This identification is canonical, as trajectories of $\widehat{\xi}$ are horizontal, i.e. of the form $\{p\}\times [0,1]$. The action of the deck group $\mathbb G=\Z/2\Z$ of the cover 
 $\overline{\widehat{E}}\to \overline{E}$ is covered by 
a diffeomorphism $\varphi:\widehat{P}\times [0,1]\to \widehat{P}\times [0,1]$. 
Then $\varphi$ preserves globally the set of boundary points. 
To describe $\overline{E}$ it suffices to understand the action of $\varphi$. 

First, if $\varphi(\widehat{P}\times \{i\})=\widehat{P}\times \{i\}$, for $i\in \{0,1\}$, 
then ${\rm Fr}(E)$ consists of the union of the images of  
$(\widehat{P}\times \{i\})/\mathbb G$, for $i\in \{0,1\}$.   
Each oriented trajectory of $\widehat{\xi}$ descends to a trajectory of the line field $\xi$.
However, we can orient each trajectory  of $\xi$, by declaring that it is 
issued from  $(\widehat{P}\times \{0\})/\mathbb G$ and arrives to  
$(\widehat{P}\times \{1\})/\mathbb G$. 
This provides a lift of $\xi$ to a vector field on $\overline{E}$, 
contradicting its non-orientability. It follows that $\varphi$ must exchange the two boundary components of $\widehat{P}\times [0,1]$. 

\begin{lemma}\label{invol}
Assume that $\varphi$ exchanges the two components $\widehat{P}\times \{i\}$, for $i\in \{0,1\}$.  
Then $\varphi$ induces 
a free action of $\mathbb G$ on $\widehat{P}$. 
\end{lemma}
\begin{proof}
Let $p_E:\overline{\widehat{E}}\to \overline{E}$ denote the projection. 
We can write $\varphi(x,0)=(\phi(x),1)$, for any $x\in \widehat{P}$, where $\phi:\widehat{P}\to \widehat{P}$ is a diffeomorphism. 
As $\xi$ is locally orientable, we can consider the trajectory $\gamma$ of $\xi$ in $\overline{E}$ issued from the point 
$p_E((x,0))=p_E((\phi(x),1))$.  Then $p_E^{-1}(\gamma)$ is the union of two arcs, say $\widehat{\gamma}_{(x,0)}$ and 
 $\widehat{\gamma}_{(\phi(x),1)}$ which are respectively issued from $(x,0)$ and $(\phi(x),1)$. 
However, the pull-backs of $\xi$-trajectories are now trajectories of $\widehat{\xi}$, in particular they are 
horizontal, i.e. of the form $\{q\}\times [0,1]$. 
Now, the maximal extensions of these two $\widehat{\xi}$ trajectories are the 
arcs $\{x\}\times [0,1]$ and $\{\phi(x)\}\times [0,1]$. 

If $x=\phi(x)$, then the restriction $p_E: p_E^{-1}(\gamma)=\{x\}\times [0,1]\to \gamma$ 
cannot be a 2-fold cover,  which is a contradiction. Thus $\phi(x)\neq x$, for any $x$ and 
$p_E^{-1}(\gamma)=\{x\}\times [0,1]\sqcup \{\phi(x)\}\times [0,1]$. This implies that $\gamma$ intersects again 
${\rm Fr}(E)$, in $p_E((\phi(x),0))$. 
If we travel backward along $\gamma$ starting from $p_E((\phi(x),0))$ we have to end at  
$p_E((\phi,0))$; we derive that $\varphi(\phi(x),0))=(x,1)$, and thus $\phi(\phi(x))=x$. 
Moreover, as $\widehat{\xi}$ is a unit vector field we have more generally that $p_E((x,t))=p_E((\phi(x),1-t))$. 

This shows that $\phi:\widehat{P}\to\widehat{P}$ defines a fixed point free involution. 
\end{proof}

Set $P=\widehat{P}/\mathbb G$, which is a closed manifold by Lemma \ref{invol} and let 
$p_P:\widehat{P}\to P$ be the quotient map. 

Given a class $\alpha\in H^1(P,\Z/2\Z)$ we have associated an interval bundle  $P\times_{\alpha} [0,1]$
which is a fibration over $P$ by intervals,  whose monodromy homomorphism  $\pi_1(P)\to \Z/2\Z$ (obtained after  
the identification $\pi_0({\rm Homeo}([0,1]))=\{\pm 1\}$) is given by $\alpha$. 
 
\begin{lemma}
There is a homeomorphism  between $\overline{E}$ and the image of 
$P\times_{w_1(\xi|_E)}[0,1]$ by an immersion with possible self-tangencies on the boundary. 
\end{lemma}
\begin{proof}
We have an exact sequence 
\[ 1\to \pi_1(\widehat{P})\to \pi_1(P)\stackrel{\alpha}{\to} \Z/2\Z\to 1\]
and an element  $j\in \pi_1(P)$ with $\alpha(j)\neq 1$ is  represented by the projection of a loop in $\widehat{P}$ 
joining $x$ to $\phi(x)$. 

The projection map $\widehat{P}\times [0,1]\to P$  sending $(x,t)$ into $p_P(x)$ factors through a map 
 $\left(\widehat{P}\times [0,1]/_{(x,t)\sim (\varphi(x),1-t))}\right)\to P$. The latter induces 
a well-defined map $\overline{E}\to \overline{P}$, where $\overline{P}$ denotes the image of 
$P$ by the immersion in $M$. The preimage of a point $p_P(x)\in P$ by the 
induced map $\overline{E}\to \overline{P}$ is the maximal trajectory  arc 
$\gamma_x\subset \overline{E}$ joining  the points 
$x$ and $\phi(x)$ from ${\rm Fr}(E)$.  It is immediate that the map $\overline{E}\to \overline{P}$ 
is a fibration with fiber $[0,1]$, which is covered by an interval 
fiber bundle $\left(\widehat{P}\times [0,1]/_{(x,t)\sim (\varphi(x),1-t))}\right)\to P$. 
The  monodromy action of $j$ sends the fiber $\gamma_x$ into itself by exchanging its endpoints, while elements from $\pi_1(\widehat{P})$ act as identity. It follows that $\overline{E}\to \overline{P}$ is the twisted product $\overline{P}\times_{\alpha} [0,1]$, image by an immersion (and its equivariant lift) of the twisted product 
$P\times_{\alpha} [0,1]$, where $\alpha$ is a class in $H^1(P,\Z/2\Z)$.

On the other hand $\overline{\widehat{E}}\to\overline{E}$ is a 2-fold cover of class $w_1(\xi)\in H^1(\overline{E},\Z/2\Z)$, by hypothesis. 
This cover is the image of an immersion of 
the 2-fold cover $\widehat{P}\times [0,1]\to P\times_{\alpha} [0,1]$. Therefore its class in 
$H^1(P\times_{\alpha} [0,1])$ is determined by the exact sequence: 
\[ 1\to \pi_1(\widehat{P}\times [0,1])\to \pi_1(P\times_{\alpha} [0,1]){\to} \Z/2\Z\to 1\]
and thus it can be identified with $\alpha$, so that $\alpha$ is the pull-back  
of $w_1(\xi|_{\overline{E}})$ to $H^1(P,\Z/2\Z)$. Recall that $E$ is diffeomorphic an open interval bundle over $P$. Thus this class can be identified  geometrically with $w_1(\xi|_{E})\in H^1(P,\Z/2\Z)$. 
\end{proof} 
This settles Proposition \ref{Ibundles}. 
\end{proof}

\subsection{Limit leaves subspace}  
The {\em limit leaves subspace} of a $1$-QC manifold is 
the compact set $\mathcal L={\rm Fr}(M-C)=\overline{M-C}\cap C$.

Components $E$ of $M-C$ which are not cylinders will be called {\em topologically non-trivial}. 
\begin{lemma}
There are only finitely many topologically non-trivial components $E$ of $M-C$. 
\end{lemma}
\begin{proof}
Let $E$ be a topologically non-trivial component of $M-C$. 
If $\overline{E}$ is a cylinder with its ends glued along a non-empty subset of the boundary,  
then the class of a loop which follows the vertical line 
contributes a non-trivial factor $\Z$ to $H_1(M)$. If $E$ is a non-trivially twisted cylinder 
then it contributes a non-trivial factor $\Z/2\Z$ to $H_1(M)$. Since $M$ is compact $H_1(M)$ is 
of finite type and hence we cannot have but finitely many such components. 
\end{proof}

Consider a maximal set  $B$ of cylindrical components $E_i=P_i\times \R$ of $M-C$ such that 
$M\setminus \cup_i E_i$ is connected. Then each $E_i$ contributes a non-trivial factor 
$\Z$ in $H^1(M)$ which is of finite type and thus $B$ is finite. If $E$ is another 
cylindrical component $E=P\times \R$, then cutting $M\setminus \cup_i E_i$ along $P$ will disconnect it. 
In particular the class of $P$ is a sum of several classes among the $P_i$ in $H_{n-1}(M)$, or by 
duality in $H^1(M)$. Thus, the set of homology classes of cores $P$ of non-essential 
cylindrical components is finite. We pick up a set $\mathcal B$ of such representatives.   
It follows that for any cylindrical component 
$E=P\times \R$ in $M$, either $P$ bounds or there exists some $Q\in \mathcal B$ 
which is homologous to $P$ and hence $P\cup Q$ bounds a codimension zero submanifold $Z$. 

\begin{lemma}
Every set of disjoint codimension zero submanifolds $Z$ of $M$  with the property that 
$\partial Z$ is a  connected curvature leaf should be finite. 
\end{lemma} 
\begin{proof}
In the proof of Proposition \ref{integrability} we noticed that there exists some $\gamma >0$ 
such that the $\gamma$-neighborhood $N_{\gamma}(L)$ is diffeomorphic to a twisted product for any 
leaf $L$. 

Consider now a leaf $L$ bounding a submanifold $Z$ and a point $p\in Z$ at distance $\gamma/2$ from $L$.
Such a point exists as  $N_{\gamma}(L)$ is embedded in $M$. 
Then the metric ball of radius $\gamma/2$ with center $p$ is diffeomorphic to a ball and contained in 
$Z$. By the Rauch comparison theorem there exists some constant $v$ depending on $M$ such that 
this metric ball has volume at least $v$ and hence the volume of $Z$ is bounded below by $v$.  

Thus there are only finitely many disjoint submanifolds $Z$ with boundary a leaf.  
\end{proof}

We add to the set $\mathcal B$ a maximal set of leaves which bound disjoint submanifolds.

\begin{proposition}\label{cylinderneighb}
There exists $\varepsilon>0$  such that for any leaves $Q\in \mathcal B$ and $P$ cobounding with $Q$ a codimension zero 
submanifold of $M$, such that  $\min_{p\in P, q\in Q} d(p,q) < \varepsilon$, then $Z$ is diffeomorphic to a cylinder.   
\end{proposition}
\begin{proof}
We prove first a slightly weaker assertion, as follows. 

\begin{lemma}
Let $Q,P$ be the cores of two  leaves belonging to cylindrical components of $M-C$. 
There exists $\delta>0$  such that for any leaves $Q\in \mathcal B$ and $P$ cobounding with $Q$ a codimension zero 
submanifold of $M$, such that  $P\subset N_{\delta}(Q)$, then $Z$ is a cylinder.   
\end{lemma}
\begin{proof}
We use the finite cover $\{U_i\}$ by cylinders of the form ${\rm int}(\overline{U})$ from the proof of 
Proposition \ref{integrability} above. 
Then $P\cap \overline{U_i}$ and $Q\cap \overline{U_i}$ are both disjoint graphs of functions 
defined on $\overline{U_i}$, for every $i$. 
Therefore they are the upper and bottom cap of a cylinder on $\overline{U_i}$. The union of all these cylinders
is a cylinder over $Q$, since $\xi|_Q$ must be orientable. 
\end{proof}
 
It remains to prove that for every $\delta>0$ there exists  $\varepsilon>0$  such that for any leaves $Q\in \mathcal B$ and $P$ cobounding with $Q$ a codimension zero 
submanifold of $M$, such that  $\min_{p\in P, q\in Q} d(p,q) < \varepsilon$, then  $P\subset N_{\delta}(Q)$. 
Assume  the contrary. Then we have a sequence of leaves $Q_i$ cobounding with $P$ such that 
 $\lim_i\min_{p\in P, q\in Q_i} d(p,q) =0$, but $P\not\subset N_{\delta}(Q_i)$. This provides a sub-sequence
 $Q_i$ such that $Q_i$ converges to $P$ but then $P\not\subset N_{\delta}(Q_i)$ for all large enough $i$, 
 contradicting  Proposition \ref{integrability} and Lemma \ref{Hausdorff}.  
 \end{proof}

\begin{proposition}\label{boundary}
If $M$ is a compact locally $1$-QC manifold, 
then there exists a compact branched hypersurface $L\subset \mathcal L$ such that  
connected components of $M-L$ are either components of ${\rm int}(C)$ or twisted cylinders.  
\end{proposition}
\begin{proof}
By Proposition \ref{integrability} $\mathcal L$ is the union of all limit leaves $L_q$. 
Note that two limit leaves $L_q, L_p\in \mathcal L$ are either disjoint or coincide. 
In fact they cannot intersect transversely, as both are limits of sequences of 
foliations leaves and if they are tangent somewhere then they should coincide. 

In the proof of Proposition \ref{integrability} we noticed that there exists some $\gamma >0$ 
such that the $\gamma$-neighborhood $N_{\gamma}(L_p)$ is diffeomorphic to a twisted product for any 
leaf $L_p$. This continues to hold for any branch $L_q^*$ of some limit leaf $L_q$. 

We consider a maximal set $\mathcal F$ of leaves which are cores of cylindrical components $E$ such that 
no two of them belong to the same component and $\min_{p\in P, q\in Q} d(p,q) > \varepsilon/2$, for every two 
leaves $P$ and $Q$. By above there are only finitely many such leaves. 
Consider the collection of branched hypersurfaces associated to components $E$ which are either topologically non-trivial 
or else they have leaves from $\mathcal F$ and denote by $L$ their union.
We claim that $L$ verifies the claim. In fact the closures of the connected components of $M-L$ are either twisted cylinders or closures of components of ${\rm int}(C)$. In fact two  cylinder components are separated by a component of ${\rm int}(C)$ 
or by a limit leaf. In the former case we need that the distance between boundaries be at least $\varepsilon$ 
or otherwise the union of the three components will be a cylinder. In particular there are only finitely many 
such components of ${\rm int}(C)$. 
\end{proof}

\begin{remark}
Note that $\mathcal L$ might have infinitely many connected 
components, for instance an infinite union of compact leaves accumulating on 
a compact leaf. Nevertheless these leaves are contained in ${\rm int}(\overline{M-C})$. 
\end{remark}

\subsection{End of proof of Theorem \ref{arbitrary}}
According to Proposition \ref{boundary} $M$ is the union of some interval bundles and finitely many  pieces 
$\overline{X_i}$ where $X_i$ are connected components of ${\rm int}(C)$.  Each $X_i$ is an open manifold with a constant curvature metric. 
Moreover, the curvature of each point of the closure $\overline{X_i}\subset M$ is also constant. The frontier 
${\rm Fr}(X_i)$ consists in finitely many components, each one being either a totally umbilical constant curvature 
submanifold, the union of two such submanifolds of $M$ which are glued together along some subset in such a way 
that the tangent space is well-defined at any point of ${\rm Fr}(X_i)$ or else a 
totally umbilical constant curvature immersed hypersurface with self-tangencies.  

Suppose that there exists a branched hypersurface $L^+\cup L^-$  in the frontier  ${\rm Fr}(X)$ of some $X$ as above. 
Since $L^+\cup L^-$ is a limit leaf, there exist arbitrarily small one-sided neighborhoods
of every $L\in\{L^+,L^-\}$ which are diffeomorphic to cylinders, the other boundary components being curvature leaves. 
Thus, by adjoining finitely many cylinders to $X$ we obtain a manifold with boundary $X^{ext}\supset \overline{X}$.
We claim that $X^{ext}$ also admits a constant curvature metric with totally umbilical boundary components.

We first claim that there exists a smooth $f:[0,1]\to \R_+$ such that 
the warped product metric on  the cylinder $Q=L\times [0,1]$ with rescaling function $f$ has constant curvature $H$. 
Recall that $L$ has constant intrinsic curvature $H+\alpha^2$ and is totally umbilical.  The curvature $k_Q(\sigma)$ of a tangent plane 
$\sigma$ spanned by the vectors $x+v$ and $w$, where $v$ and $w$ are tangent to $L$, $x$ to the $[0,1]$ factor and 
$\parallel x\parallel^2+\parallel v\parallel^2=1$ (in the warped product metric)
is given by:
\[  k_{Q^{}}(\sigma)=-\frac{f''(t)}{f(t)}\parallel x\parallel^2 + \frac{k_{L^{}}(v,w)-f'^2(t)}{f^2(t)}\parallel v\parallel^2,\]
where $k_L$ denotes the sectional curvature of $L$. 
It follows that for every values of $H$ and $H+\alpha^2$ there exists some scaling function $f$ defined on a small interval 
containing $0$ such that $f(0)=1$ and $ k_{Q^{}}(\sigma)=H$, for any  $\sigma$ (see e.g. \cite{BO}, Corollary 7.10 when 
$H+\alpha^2\leq 0$).

Now we can glue together the Riemannian manifold obtained by cutting  $M$ open along  $L^{+}$ and gluing 
the cylinder $Q^{+}$ with its warped product metric of constant curvature $H$ along $L^{+}$. 
The result is a smooth Riemannian metric because the corresponding second fundamental forms of $L^+$ within 
$M$ and $Q^+$ agree. We do a similar construction on $Q^-$ and further for every branched leaf in ${\rm Fr}(X)$.

Eventually note that the warped metric formula also makes sense for the case where $L^{+}$ is an immersed hypersurface in $M$. As we only have self-tangencies along $L^+$ 
we can perform the gluing procedure as above. 

The result is a Riemannian manifold diffeomorphic to $X^{ext}$ with a metric of 
constant curvature and totally umbilical 
boundary components.  

Therefore $M$ is a graph of space forms. 
This proves Theorem \ref{arbitrary}.

\section{Proof of Theorem \ref{lambdapositive}}
\subsection{Outline}
The positivity of $\lambda$ implies that  connected components  of the set of 
non-isotropic points consist of  spherical cylinders. 
In the second step we describe the closures of these connected components, which are not necessarily manifolds.  
To simplify the topology we want to excise the cores of these tubes, although their number might be infinite.   
A key argument is Grushko's theorem in group theory which translates here into the finiteness of the number 
of factors in a free  amalgamated product presentation of 
a finitely presented group. This shows that one could carry out 
surgery for only finitely many tubes. 

It remains to understand the topology of the remaining pieces, which are 
neighborhoods of the set of isotropic points. 
We show there is a way to cap off  the boundary components of each piece, 
by means of hypersphere 
caps in the hyperbolic space whose geometry is controlled. 
 A deep result of Nikolaev  (see \cite{Nik} and Proposition \ref{Nikolaev} for a precise statement) states that 
a manifold with bounded volume and diameter whose integral 
anisotropy is small enough should be diffeomorphic to a space form. 
This procedure constructs closed manifolds out of regular neighborhoods 
of isotropic points, whose geometry is controlled and whose 
integral anisotropy is arbitrarily small. Then Nikolaev's result 
shows that these are diffeomorphic to space forms.

\subsection{The topology of $M-C$ when $\lambda >0$}
Let $M^n$ be a closed  $n$-manifold, $n\geq 3$, admitting a $1$-QC metric with orientable distinguished line field, which is 
assumed conformally flat when $n=3$.   
Suppose that $\lambda$ is positive and $\pi_1(M)$ is infinite torsion-free.

  \begin{proposition}\label{leaves}
  Any non-compact connected component $E$ of $M-C$ 
  is diffeomorphic to $S^{n-1}\times \R$. 
  \end{proposition}
  \begin{proof}
 Since $\lambda|_E >0$ curvature leaves are 
diffeomorphic to spheres quotients.

By Proposition \ref{immerse} there exists an isometric immersion 
$\Phi: \widetilde{M}\to \mathbb H_{\kappa}^{n+1}$. Note that 
each curvature leaf is  mapped isometrically onto some round 
$(n-1)$-sphere by the immersion in $\mathbb H_{\kappa}^{n+1}$ 
(see e.g. \cite{Re2}). Recall that {\em round $n$-spheres} 
(sometimes also called geodesic or extrinsic $n$-spheres) are 
intersections of a hypersphere with $\mathbb H_{\kappa}^{n}\subset \mathbb H_{\kappa}^{n+1}$  (see 
\cite{Re2}, section 2).

Now any curvature leaf $L$ in $M$ is covered by 
a disjoint union of spheres in $\widetilde{M}$ and thus $L$ is 
diffeomorphic to a spherical space forms. 
The stabilizer of one sphere in $\widetilde{M}$ 
is finite and obviously contained in $\pi_1(M)$. 
Thus it must be trivial as $\pi_1(M)$ was assumed to be torsion-free 
and hence the curvature leaf $L$ is a sphere.  Therefore $E$ is diffeomorphic to $S^{n-1}\times \R$. 
\end{proof}
When $C=\emptyset$ then, $M$ is 
diffeomorphic to  a spherical bundle over the circle.

\subsection{Closures  of globally 1-QC components within $1$-QC manifolds}
Suppose now that there exist isotropic points, so 
that $C\neq\emptyset$.

\begin{lemma}\label{compactification}
Let $E$ be a connected component of $M-C$ such that $\xi|_E$ is orientable. We assume that 
$\lambda \geq c >0$ on $E$. 
Then  $\overline{E}$ is  obtained from $E$ by adjoining 
two boundary spheres, which might possibly be degenerate or tangent at one point or along a codimension one sphere. 
\end{lemma} 
\begin{proof}
The proof is similar to (\cite{DDM}, Lemma 2.7). 
Let $q\in \overline{E}\cap C$. 
Choose a component $E_0$ of $\pi^{-1}(E)\subset \widetilde{M}$, which is diffeomorphic to a spherical cylinder
since $E$ is simply connected.  
Let then $\widetilde{q}$ be a lift of $q$ which belongs to $\overline{E}$.

Any connected component $L_{\widetilde{p}}$ of the preimage $\pi^{-1}(L_{p})\subset \widetilde{M}$  
of a curvature leaf through $p$ 
is sent by the isometric immersion $\Phi$ onto an embedded round sphere  $S_{\lambda(p)}$ of curvature $\lambda(p)$ 
in $\mathbb H_{\kappa}^{n+1}$. Now $L_{p}$ is a quotient of  the  sphere $L_{\widetilde{p}}$ by a finite subgroup of $\pi_1(M)$. Since $\pi_1(M)$ is torsion-free  $L_{p}$ is also a sphere, which is isometric to  $S_{\lambda(p)}$. 

Let $p_i\in E$ be a sequence converging in $M$ to $q$ and $\widetilde{p_i}\in E_0$ lifts converging to $\widetilde{q}$
in $\widetilde{M}$. 
The curvature leaves $L_{p_i}$ lift to curvature leaves $L_{\widetilde{p_i}}$ in $\widetilde{M}$ 
which are round spheres of  radius $R_i$.  Since $\lambda_i$ is bounded from below, 
$R_i$ are bounded. The sphere $L_{\widetilde{p_i}}$ lives in a  geodesic 
$n$-plane $\mathbb H_{\kappa}^n$ orthogonal at  the lift  of the unit vector field $\xi|_{\widetilde{p_i}}$. 
By a compactness argument and passing to a sub-sequence we can assume that $R_i$ and $\xi|_{\widetilde{p_i}}$
converge to $R_{\infty}$ (which might be $0$) and $\xi_{\infty}$, respectively. Then the spheres $L_{\widetilde{p_i}}$ converge to the 
round sphere $L_{\infty}$ of radius $R_{\infty}$ living in the $n$-plane $\mathbb H_{\kappa}^n$ orthogonal 
to $\xi_{\infty}$ and passing through $\widetilde{q}$. 
 
The argument from (\cite{DDM}, Lemma 2.7) applies to show that 
for any sequence $\widetilde{p'_i}\in E_0$  converging to $\widetilde{q}$
in $\widetilde{M}$ the spheres $L_{\widetilde{p'_i}}$ converge to the 
round sphere $L_{\infty}$. Indeed, if we had another limit sphere $L'_{\infty}$ passing through $\widetilde{q}$, then a 
small enough connected neighborhood $V$ of some point $p\in L_{\infty}-L'_{\infty}$ would intersect 
all leaves   $L_{\widetilde{p_i}}$ and would miss all  $L_{\widetilde{p'_i}}$. Since any curvature leaf in 
$E_0$ disconnects $E_0$, this would imply that $V$ is not connected, contradicting our choice. 

Therefore $\overline{E_0}$ is obtained from $E_0$ by adjoining two boundary (possibly degenerate) 
round spheres corresponding to the ends. 
Since $\overline{E_0}$ is compact  and $\pi$ is a covering  $\pi(\overline{E_0})=\overline{E}$.  
Then $\overline{E}$ is obtained from $E$ by adjoining two round spheres. 
They cannot intersect transversely, as otherwise close enough leaves of $E$ should intersect. 
Thus the boundary spheres should be either disjoint or tangent. Lifting  both of them to $\widetilde{M}$ we obtain two round 
$(n-1)$-spheres in $\mathbb H_{\kappa}^{n+1}$.  They bound two convex $n$-balls whose intersection is convex. 
Therefore the spheres intersection is either one point or a $(n-2)$-sphere and in particular it is connected. 
\end{proof}
 
If one only assumes that the metric is 1-QC the curvature leaves function $\lambda$ might have poles at $C$, contrasting with 
the situation described in  Lemma \ref{extension-lambda} for locally 1-QC metrics.  
The behavior of $\lambda$ is directly related with the behavior, possibly wild,  of the distinguished line field $\xi$ in a neighborhood of an isotropic point. Consider a spherical 
cylinder with a ball adjoined to it, having the form of a 
rotationally invariant cap whose north pole is the only isotropic point. 
Then curvature leaves are spheres whose intrinsic curvature    
$\lambda$ explodes when approaching the isotropic point.
In fact,  this is the only phenomenon which can arise:

\begin{lemma}\label{poles}
Assume that  $\lambda >0$ on $M-C$ and set 
\[C_{\infty}= \{q\in \overline{M-C}\cap C; \lim\sup_{p\to q}\lambda (p) =\infty\}\]  
Then, 
\begin{enumerate}
\item either there exists a connected $E$ component of $M-C$ with the property that 
$q\in\overline{E}$ and $\overline{E}$ is  diffeomorphic to a ball containing $q$. 
\item or else there exists a sequence $E_i$ of components of $M-C$ 
with the property that $q\in \overline{\cup_{i=1}^{\infty} E_i}$, but 
$q\not\in \overline{E_i}$, for any $i$. Moreover, $\overline{E_i}$ are embedded annuli with disjoint interiors. 
\end{enumerate}   
\end{lemma}
\begin{proof}
Given a compact Riemannian 
$M$ there exists some  constant $c_1$ such that 
any embedded sphere $S$ of diameter $d\leq c_1$ bounds an embedded ball $D$ of 
diameter $f_1(d)$ and volume $f_2(d)$, such that $\lim_{d\to 0}f_i(d)=0$, $i=1,2$. 
It suffices to cover $M$ by finitely many open charts 
diffeomorphic to balls and to choose then $c_1$ so that any metric ball on $M$ 
of radius $c_1$ is contained within some chart.
 
Now, the  diameter of a round sphere  $S_{\lambda}$  goes to $0$, 
when $\lambda\to \infty$. The same holds for the diameter of $L_{\lambda}$ 
in the Riemannian metric of $M$. 
Therefore, if $\lambda$ is larger than some $c_2$ then $L_{\lambda}$
bounds an embedded ball $D_{\lambda}$ in $M$. 

Consider $\lambda_{1}$ such that the intrinsic diameter of $S_{\lambda_{1}}$ is smaller than $c_1$. 
There exists a curvature leaf $L_1$ of curvature $2\lambda_{1}$ in $E$, which therefore bounds a 
ball $D_1\subset M$. 
Let $L_2\subset E\cap (M-D_1)$ be a curvature leaf  whose  
curvature is $\lambda >\lambda_{1}$ and thus it bounds 
also a ball $D_2$ in $M$. The leaves $L_1$ and $L_2$ bound a 
submanifold $E_0$. From  Proposition \ref{leaves} and the strong form of the global Reeb stability theorem 
(see \cite{Mil,EMS}) $E_0$ is compact and diffeomorphic to a spherical cylinder. 

Observe that there is an unique embedded ball $D$ bounded by an 
embedded codimension one sphere in $M$, as otherwise  
$M$ would be homeomorphic to a sphere, contradicting the fact that 
$\pi_1(M)$ is infinite. Hence $D_2=E_0\cup D_1$. 

If $q\not\in D_1$, as $q\in C_{\infty}$ we could choose $\lambda$ 
large enough so that $f_2(\lambda) < {\rm vol}(D_1)$. 
On the other hand $D_2\subset D_1$ and hence the volume of the unique 
embedded ball bounding $L_2$ would be larger than $ {\rm vol}(D_1)$, 
contradicting our choice of $\lambda$.  
Thus $q\in D_1$.

Note that, when $q\in C_{\infty}$ is not as above, then there exists an infinite sequence $E_i$ of connected components of $M-C$ such that 
$q\in \overline{\cup_i E_i}$, but $q\not\in \overline{E_i}$, for all $i$. 
These components should have disjoint interiors and they are diffeomorphic to spherical cylinders. 
The size of $\overline{E_i}$ (namely the radius of a boundary circle and the width) 
should converge to $0$. Moreover,  for every $\varepsilon >0$ all but finitely many $\overline{E_i}$ 
 are  contained in  a metric ball of radius  $\varepsilon$ around $q$.  They are therefore null-homotopic. 
\end{proof}

\subsection{Construction of neighborhoods of isotropic points}
\begin{lemma}\label{saturated}
Suppose that $\lambda> 0$ on $M-C$. We set $V_{\epsilon}(C)=\{p\in M; |H(p)-N(p)| < \epsilon\}$. 
Then there exist arbitrarily small saturated neighborhoods of $C$, namely  for every $\varepsilon >0$ small enough 
there exists an open neighborhood $W_{\varepsilon}\subset V_{\varepsilon}(C)$ such that $W_{\varepsilon}-C$ is a union of leaves and 
$W_{\varepsilon}\supset V_{\delta}(C)$. 
\end{lemma}
\begin{proof}
This is a classical result in  codimension one foliations with compact leaves.  
From Lemmas \ref{compactification} and \ref{poles} for each 
$p\in C\cap \overline{M-C}$  there exists a (possibly not unique) limit curvature leaf $L_p\subset C\cap \overline{M-C}$ which passes through $p$. 
This limit curvature leaf might degenerate to a single point.

Assume now that $V_{\varepsilon}(C)$ does not contain a saturated neighborhood. Then 
we can find a sequence of pairs  of points $p_i,q_i$, each pair belonging to the 
same curvature leaf in $M-C$ such that $p_i\to p\in C$ but 
$q_i\not\in V_{\varepsilon}(C)$. If $q$ is an accumulation point of $q_i$ then it cannot belong to $C$ since $|H(q)-N(q)| \geq \varepsilon$. The curvature leaf $L_q$ passing 
through $q$ is contained therefore in $M-C$.  This implies that  
the compact curvature leaves $L_{q_i}$ passing through $q_i$ converge 
in the Hausdorff topology towards $L_q$. 
On the other hand we have 
$L_{p_i}=L_{q_i}$ so that it also converges towards  some $L_p$. 
This  contradicts the fact that 
$L_q$ is complete and thus it cannot contain points from $C$. 

Eventually, if  there does not exist some $\delta>0$ such that $W_{\varepsilon}\supset V_{\delta}(C)$, 
we would find a sequence $p_n\in V_{\frac{1}{n}}(C)-W_{\varepsilon}$. If $p$ is a limit point of the sequence $p_n$ then 
$p\not\in W_{\varepsilon}$, as $W_{\varepsilon}$ is open. On the other hand $p\in C$, which is a contradiction. 
\end{proof}

Now $W_{\varepsilon}$ is a manifold containing $C$ and whose frontier 
$\overline{W_{\varepsilon}}-W_{\varepsilon}\subset M-C$. 
Since $W_{\varepsilon}$ is a saturated subset, its frontier consists of union of curvature leaves. 
From Lemma \ref{compactification} these are round spheres, and in particular $\overline{W_{\varepsilon}}$ is 
a codimension zero submanifold of $M$. 

Denote by $N_i(\varepsilon)$ the connected components of $\overline{W_{\varepsilon}}$. There are manifolds with boundary, which we call {\em standard neighborhoods of isotropic points}.  The complementary $M-W_{\varepsilon}$ is a union of spherical 
cylinders, which will be called {\em necks}.

The first main result is the following chopping result: 

\begin{proposition}\label{essential}
Assume that $\lambda(p) \geq c >0$ for $p\in M-C$. 
Then there exist only  finitely many  components  $N_i(\varepsilon)$ and also only finitely many 
necks. 
\end{proposition}

\subsection{Preliminaries concerning the positive $\lambda$}

\begin{lemma}\label{lowerlambda}
If $\lambda(p) >0$, for  all $p\in M-C$  and $H|_C >0$, then there exists $\lambda_{\min}$ such that 
$\lambda(p) \geq \lambda_{\min}>0$, for all $p\in M-C$. 
\end{lemma}
\begin{proof}
As $H$ is smooth on $M$ and $C$ is compact, there exists $\delta >0$ and 
an open neighborhood $V$ of $C$ in $M$ 
such that $H|_V > \delta$. On the other hand $\lambda$ is smooth on $M-C$ and 
$M-V$ is compact hence $\inf_{p\in M-V} \lambda(p)\geq \delta' >0$. We also have 
$\lambda(p)\geq H(p)\geq \delta$, if $p\in V$, so that we can take $\lambda_{\min}=\min(\delta,\delta')$. 
\end{proof}

\begin{lemma}\label{upperlambda}
There exists some constant $\lambda_{\max}(\varepsilon)$ such that  whenever $p\not\in W_{\varepsilon}$, then 
\[ \lambda(p)\leq \lambda_{\max}(\varepsilon).\]
\end{lemma}
\begin{proof}
We have 
\[ \lambda(p)= H(p)+\frac{\parallel {\rm grad} (H)_p\parallel}{4(H(p)-N(p))^2}
\leq \sup_{p\in M} H(p)+ \frac{1}{4\delta^2} \sup _{p\in M} \parallel {\rm grad} H|_p\parallel\]
where $\delta=\delta(\varepsilon)$ is given in Lemma \ref{saturated}. 
\end{proof}

\begin{lemma}\label{caps}
Assume that the compact $n$-manifold $X$  can be isometrically immersed into  
$\mathbb H_{\kappa}^{n+1}$ so that its boundary $\partial X$ goes onto 
the round sphere $S_{\lambda}$. Then the volume of $X$ 
is greater  than or equal to  the volume of the  
standard $n$-ball $B_{\lambda}$  with boundary $S_{\lambda}$ 
lying in an  $n$-plane 
$\mathbb H_{\kappa}^{n}$ embedded in  $\mathbb H_{\kappa}^{n+1}$.  
\end{lemma}
\begin{proof}
As the orthogonal projection  $\mathbb H_{\kappa}^{n+1}\to \mathbb H_{\kappa}^{n}$  
is distance-nonincreasing, the measure of $X$ is greater 
than the measure of its projection. Now, the image of $X$ is a homological 
$n$-cycle with  boundary $S_{\lambda}$. Then,  this $n$-cycle cannot miss a point  of 
$B_{\lambda}$ because the inclusion of $S_{\lambda}$ into $\mathbb H_{\kappa}^{n}-\{{\rm pt}\}$ 
is a homology isomorphism. 
\end{proof}

In particular, Lemma \ref{caps} shows that the 
$n$-dimensional Lebesgue measure of $X$ is uniformly bounded 
from below by  the volume of $B_{\lambda}$.   Recall that a round sphere $S_{\lambda}$ is a metric hypersphere 
in a totally geodesic $\mathbb H^n_{\kappa}\subset \mathbb H^{n+1}_{\kappa}$. Now a hypersphere 
of radius $d$ in  $\mathbb H^n_{\kappa}$, with $\kappa \neq 0$, has principal curvatures 
$\frac{\sqrt{\kappa}}{\tanh(\sqrt{\kappa}d)}$ and hence sectional curvature: 
\[ \lambda =\kappa \left (\frac{1- \tanh^2(\sqrt{\kappa}d)}{\tanh^2(\sqrt{\kappa}d)}\right).\]
Thus its volume of $B_{\lambda}$ is a function 
$C_1(\lambda,\kappa,n) >0$  decreasing as a function of 
the intrinsic curvature $\lambda$ for all $\kappa\geq 0$, namely such that:  
\[\lim_{\lambda\to \infty}C_1(\lambda,\kappa,n)=0, \; 
 \lim_{\lambda\to 0}C_1(\lambda,\kappa,n)=\infty. \]

\subsection{Compactness of $N_i(\varepsilon)$}
We first prove:
\begin{proposition}\label{compact}
Assume that $\lambda(p) \geq c >0$ for $p\in M-C$. 
Then each  $N_i(\varepsilon)$ is compact. 
\end{proposition} 
\begin{proof}
Such a  component $N_i(\varepsilon)$ is compact, unless its boundary $\partial N_i(\varepsilon)$ consists of   
infinitely many $(n-1)$-spheres. 
All but finitely many such $(n-1)$-spheres should be separating, by 
Grushko's theorem since each non-separating $(n-1)$-sphere contributes a  free factor $\mathbb Z$ 
to $\pi_1(M)$.

Now, let $M_j$ be the compact connected  submanifolds of $M$ bounded 
by the separating $(n-1)$-spheres in $\partial N_i(\varepsilon)$. By van 
Kampen's  theorem and Grushko's theorem all but finitely many $M_j$ 
should be simply connected. Since $M_j$ is simply connected, the inclusion into $M$ lifts to 
an embedding of $M_j$ into $\widetilde{M}$, and 
thus $M_j$ is isometrically immersed into $\mathbb H_{\kappa}^{n+1}$ such that 
its boundary goes onto some $S_{\lambda}$. 

Now, each  $\partial M_j$   is a 
curvature leaf  isometrically embedded 
into $\mathbb H_{\kappa}^{n+1}$ as a round $(n-1)$-sphere $S_{\lambda}$ 
of curvature $\lambda$. 
From lemma \ref{caps} the 
$n$-dimensional measure of $M_j$ is uniformly bounded from below by 
$C_1(\lambda,\kappa,n) >0$. 

Since $\sup_{p\in \partial W}\lambda \leq \lambda_{\max}(\varepsilon) < \infty$  by Lemma \ref{upperlambda} 
the $n$-dimensional measure of $M_j$ is uniformly bounded from below by 
$C_1(\lambda_{\max}(\varepsilon),\kappa,n) >0$. Since $M$ has finite volume there are only 
finitely many pairwise disjoint domains $M_j$ and hence only finitely many  components of $\partial N_i(\varepsilon)$.
\end{proof}
\begin{remark}\label{separating}
The proof above shows that the 
number of separating  boundary components which separate 
simply connected  domains is uniformly bounded by 
$\frac{{\rm vol}(M)}{C_1(\sup_{p\in M-W_{\varepsilon}}\lambda,\kappa,n)}$. 
\end{remark}

\subsection{Proof of Proposition \ref{essential}}
As above, all but finitely many $N_i(\varepsilon)$ are simply connected, by van Kampen's and Grushko's theorems.

\begin{lemma}\label{sc}
If $N_i(\varepsilon)$ is simply connected then it is diffeomorphic to a 
sphere with  finitely many disjoint open balls removed.  
\end{lemma}
\begin{proof}
Since $N_i(\varepsilon)\subset M$ is 
compact and simply connected it can be isometrically immersed into 
$\mathbb H_{\kappa}^{n+1}$. Moreover, the boundary 
$\partial N_i(\varepsilon)$ consists of finitely many $(n-1)$-spheres.
Since $N_i(\varepsilon)$ is compact and any ball in  $\mathbb H_{\kappa}^{n+1}$ 
is conformally equivalent to a ball in  the Euclidean space 
$\mathbb R^{n+1}$ it follows that $N_i(\varepsilon)$ can be conformally immersed into $\mathbb R^{n+1}$.
There is then a standard procedure of gluing balls 
along the  boundary within the realm of conformally flat 
manifolds (see \cite{Pi}, section 2). We obtain then a conformally flat 
closed manifold $\widehat{N_i(\varepsilon)}$ 
which is simply connected. Kuiper's theorem  (\cite{Ku1})
implies that $\widehat{N_i(\varepsilon)}$ 
is diffeomorphic to a sphere, and hence the lemma.   
\end{proof}

\begin{lemma}
There are only finitely many components 
$N_i(\varepsilon)$ other than cylinders. 
\end{lemma}
\begin{proof}
We already saw above that there are finitely 
many components $N_i(\varepsilon)$ with one boundary component 
or not simply connected. It remains to show that we cannot have 
infinitely many holed spheres $N_i(\varepsilon)$ 
with at least three holes. Assume the contrary.  
Then all but finitely many of them should be connected  among themselves 
using cylinders. But a trivalent graph has at least 
that many free generators of its fundamental group as vertices. 
This implies that $\pi_1(M)$ has a free factor of 
arbitrarily large number of generators, contradicting the 
Grushko theorem. 
\end{proof}

The only possibility left is to 
have infinitely many manifolds $N_i(\varepsilon)$ which are 
diffeomorphic to cylinders $S^{n-1}\times [0,1]$, to be called {\em tubes}. 
Only finitely many of these tubes can be non-separating. 
We say that two separating tubes $N_i(\varepsilon)$ and $N_j(\varepsilon)$ 
are equivalent if  the component of 
$M-(N_i(\varepsilon)\cup N_j(\varepsilon))$ joining 
the boundary components of the tubes is a cylinder, i.e. the tubes are isotopic in $M$.

A separating tube is {\em inessential} if one of the two 
components of its complement is a ball and {\em essential} otherwise. 
There are also only finitely many equivalence classes  of separating 
essential tubes  $N_i(\varepsilon)$. 

In an infinite family of essential pairwise equivalent $N_i(\varepsilon)$ we have necks $C_i$
connecting them in a chain. Their union is then an open spherical cylinder.
By Lemma \ref{compactification} its closure is a closed 
spherical cylinder or a ball. Moreover, every $N_i(\varepsilon)$ contains 
a round sphere $X_i\subset N_i(\varepsilon)$ obtained as a limit 
curvature leaf. Then, the limit $X_{\infty}$ of $X_i$ exists and is a 
round sphere possibly degenerate, by the arguments from the proof of Lemma \ref{poles} and \cite{DDM}. On the   
other hand $X_{\infty}\subset C$ as $H(p)=N(p)$ for all $p\in X_i$. 
Thus $X_{\infty}\subset W_{\varepsilon}$ and according to Lemma \ref{saturated}
sufficiently closed $X_i$ are also contained in $W_{\varepsilon}$. This 
contradicts the fact that there is a neck $C_i$ separating $X_i$ 
from $X_{i+1}$ which is disjoint from $W_{\varepsilon}$. 

If we have an infinite family of inessential  tubes $N_i(\varepsilon)$ accumulating to a point $q$ of $C$, then 
all but finitely many will be contained within the neighborhood $N_j(\varepsilon)$ containing $q$, leading again to a contradiction.  
This proves Proposition \ref{essential}.

\subsection{The geometry of $N_i(\varepsilon)$ in the case $H|_{M-C}>0$} 
We give here a simple argument which permits to conclude the proof when $H>0$. 
\begin{lemma}
Assume that $H >0$ and $n\geq 3$. Then  $N_i(\varepsilon)$ is conformally equivalent 
to a spherical space form with boundary. 
\end{lemma}
\begin{proof}
We have $|N(p)-H(p)| \leq \varepsilon$, for 
any $p\in N_i(\varepsilon)$. 
As $H$ was supposed to be positive on $M$ there exists some $K$ 
such that $\min(H(p), N(p)) \geq K > 0$ for any $p\in N_i(\varepsilon)$. 

Consider a connected component $V_i$ of $\pi^{-1}(N_i(\varepsilon))$ , 
where $\pi:\widetilde{M}\to M$ is 
the universal covering projection. 
Since the sectional curvature of $V_i$ is bounded from 
below by $K >0$, the Bonnet-Myers theorem  implies that 
$V_i$ has bounded diameter and hence it is 
compact. Each boundary component of $V_i$ has constant positive curvature and hence 
is diffeomorphic to a sphere. 
This implies that $\pi_1(V_i)$ is a free factor of $\pi_1(\widetilde{M})=1$ 
and so $V_i$ must be simply connected. 
Then the conformal gluing of balls 
described in (\cite{Pi}, section 2) along the boundary components 
of $V_i$ produces a closed simply connected conformally flat manifold 
$\widehat{V_i}$ which can be immersed into $\mathbb R^{n+1}$.  
Therefore by Kuiper's theorem  (\cite{Ku1})  $\widehat{V_i}$ is conformally equivalent to a sphere.
We derive that  $V_i$ is conformally equivalent to a holed sphere, since it is obtained from a sphere 
by deleting out the interiors of several disjoint round balls.  

Now, notice that any closed loop in $N_i(\varepsilon)$ can be lifted 
to a path in $V_i$, so that the action of the subgroup 
$\pi_1(N_i(\varepsilon))$ of $\pi_1(M)$ as a deck transformation subgroup 
acting on $\widetilde{M}$ keeps $V_i$ globally invariant. 
Thus $V_i$ is endowed with a free action by $\pi_1(N_i(\varepsilon))$. 
Since $V_i$ is compact $\pi_1(N_i(\varepsilon))$ must be finite. 

The action of $\pi_1(N_i(\varepsilon))$ on the holed sphere 
$V_i$ is by conformal diffeomorphisms. Then Liouville's theorem (see \cite{Kul3}, Thm.3.1) implies 
that each one of these diffeomorphisms is the restriction of some global conformal diffeomorphism 
of the round sphere $S^{n}$, as $n\geq 3$. The group generated by the  global conformal diffeomorphisms of $S^n$ is still isomorphic to $\pi_1(N_i(\varepsilon))$, as any conformal diffeomorphism which is identity on an open subset of $S^n$ 
is the identity.  The spherical metric induces then a spherical space form structure on $N_i(\varepsilon)$. 
Moreover,  $N_i(\varepsilon)$ is obtained from the quotient of $S^n$ by the finite group $\pi_1(N_i(\varepsilon))$
by removing several open balls. 
\end{proof}

\begin{remark}
Observe that the immersability property and 
the strict positive curvature enables 
us to find an elementary proof disposing 
of the differentiable sphere theorem under 
a pointwise pinched curvature assumption. 
\end{remark}

\begin{remark}
Note that the  $\pi_1(N_i(\varepsilon))$-action on $S^n$ might possibly not be free.
Its point stabilizers might be finite, in which case the quotient $S^n/\pi_1(N_i(\varepsilon))$ is a so-called orbifold. 
The proof actually shows that $M$ is conformally equivalent 
to a connected sum of spherical space forms and classical Schottky 
manifolds (see \cite{Pi}). 
\end{remark}

\begin{remark}
The proof of Theorem \ref{arbitrary} shows  that closed manifolds having 
positive horizontal curvature for some $1$-QC 
{\em analytic} metric should be conformally equivalent to either 
the $n$-sphere, $S^1\times S^{n-1}$ or $S^1\times_{-1}S^{n-1}$. 
Here $S^1\times_{-1}S^{n-1}$ denotes the non-orientable $S^{n-1}$-bundle 
over $S^1$ with structure group $O(n)$, i.e. obtained when the monodromy map 
is a reflection in a hyperplane. 
\end{remark}

\subsection{Proof of Theorem \ref{lambdapositive}} 
The more general situation when $H|_C>0$ is treated along the same lines as before.
A small neighborhood of isotropic points has boundary consisting of spheres. 
We wish to glue along each boundary component a ball by identifying the corresponding boundary spheres 
by means of some (gluing) diffeomorphism. This operation will be called {\em capping off} boundary spheres. 
The homeomorphism type of the closed manifold obtained by capping off its boundary components is well-defined,  
 namely independent on the gluing diffeomorphisms. We should stress that, in contrast,  the diffeomorphism type of 
the resulting manifold is determined by the classes of the gluing diffeomorphisms  
in the group of sphere diffeomorphisms up to isotopy, which can be non-trivial -  for instance when $n=7$.
However, if the boundary spheres are endowed with Riemannian metrics of constant curvature metrics 
and the gluing diffeomorphism is an isometry, then the capped off manifold is well-defined up to diffeomorphism, since 
orientation preserving isometries are isotopic.  
 
Our aim is to realize a  {\em geometric capping off}  using  spherical caps of constant curvature as balls and 
isometries as gluing diffeomorphisms. Both the manifold and the spherical caps will be isometrically immersed in 
$\mathbb H_{\kappa}^{n+1}$. The resulting manifold inherits a  $\mathcal C^1$ Riemannian metric structure, which is 
$\mathcal C^{\infty}$ outside the joint spheres. A slight perturbation of the metric will produce a $\mathcal C^{\infty}$-metric.   

Our key result below  shows that small neighborhoods of isotropic points 
are standard:

\begin{proposition}\label{isotrop}
If   $\varepsilon$ is small enough 
then  those $N_i(\varepsilon)$ which have positive sectional curvature are 
diffeomorphic to spherical space forms with several disjoint open balls 
removed. 
\end{proposition}
We postpone the rather involved proof until the next section \ref{isotrop-proof}. 

\vspace{0.5cm}
{\em End of proof of Theorem \ref{lambdapositive}}. 
By hypothesis $\lambda >0$ on $M-C$ and $H|_C >0$.  From 
Lemma \ref{lowerlambda} $\lambda$ has a positive lower bound on $M-C$ and 
by Proposition \ref{essential} there are only finitely many $N_i(\varepsilon)$ and finitely many necks, which are spherical cylinders. 
By choosing $\varepsilon$ small enough we can assume that all $N_i(\varepsilon)$   have positive sectional 
curvature, because $N|_C=H|_C>0$ and $N_i(\varepsilon)$ are saturated by  Lemma \ref{saturated}. Proposition \ref{isotrop} above shows that 
$N_i(\varepsilon)$ are diffeomorphic to spherical space forms with several disjoint open balls removed. 
Thus $M$ can be obtained by gluing together the  holed spherical space forms and spherical cylinders along 
their boundary spheres, which correspond to attaching 1-handles.

\subsection{Proof of Proposition \ref{isotrop}}\label{isotrop-proof}
The saturation Lemma \ref{saturated} shows that 
given $\nu>0$ there exists  $\epsilon(\nu) >0$ small enough such that  
$|N(p)-H(p)| \leq  \nu $, for any $p\in N_i(\varepsilon)$, meaning that  
$N_i(\varepsilon)$ has pointwise pinched curvature. 

Consider the manifold $\widehat{N_i(\varepsilon)}$ obtained 
topologically from $N_i(\varepsilon)$ by capping off boundary spheres by balls. We will show  
that $\widehat{N_i(\varepsilon)}$   also  admits a pointwise pinched  Riemannian metric.

Let $S$ be a boundary component of $N_i(\varepsilon)$ and 
$F$ be a collar of $S$ in $N_i(\varepsilon)$ which is saturated by 
curvature leaves. In particular, $F$ is diffeomorphic to 
$S^{n-1}\times [0,1]$. As $F$ is simply connected, the inclusion $F\subset M$  could be lifted 
to $\widetilde{M}$ and thus $F$ is embedded into $\mathbb H_{\kappa}^{n+1}$.  
We say that $F$ has {\em width} $w$ if the distance between the 
two boundary components is $w$.

Before to proceed, recall that the {\em envelope} of a one-parameter family of hypersurfaces in 
$\mathbb H_{\kappa}^{n+1}$ is a hypersurface tangent to each member of the family (see \cite{Ei} for the case of surfaces). 
Alternatively, it is the boundary of the region filled by the family of hypersurfaces. If the family is given by 
equations $F_t(x)=0$, depending smoothly on the parameter $t$, then its envelope is the set of points $x$ 
for which there exists $a$ such that $F_a(x)=\left.\frac{\partial F_t}{\partial t}\right|_{t=a}(x)=0$. 
The envelope of a one-parameter family of hyperspheres is usually called a {\em canal} hypersurface in the literature. 
 
 We will first need the following lemma, whose proof is postponed to section \ref{proofcap} below: 

\begin{lemma}\label{cap}
\begin{enumerate}
\item Assume $H(S) >0$. If the width $w$ of the collar $F$ is 
small enough then we can cap off smoothly $F$ along the boundary component 
$S$ by attaching  a hyperspherical cap $\mathcal H^+(S)$ using an isometry between  its boundary 
$\partial \mathcal H^+(S)$ and $S$ and 
of radius 
\[d=\frac{1}{\sqrt{\kappa}}{\rm arctanh} \sqrt{\frac{\kappa}{H(S)+\kappa}}=
\frac{1}{2\sqrt{\kappa}}\log\left(\frac{1+\sqrt{\frac{\kappa}{H(S)+\kappa}} }{1-\sqrt{\frac{\kappa}{H(S)+\kappa}}}\right)\] 
so that $F\cup_S \mathcal H^+(S)$ embeds isometrically 
in $\mathbb H_{\kappa}^{n+1}$. 
\item If $H(S) \leq 0$ then there is no hyperspherical cap tangent to $F$  along $S$. 
\end{enumerate}
\end{lemma}

We suppose from now on that $H(S)>0$. 
Let now $\widehat{N_i(\varepsilon)}$ denote the manifold obtained 
from $N_i(\varepsilon)$ by adding to each boundary component $S$ 
the corresponding spherical cap $\mathcal H^+(S)$. Its diffeomorphism type is uniquely determined, independently on the 
choice of the orientation preserving isometries used for gluing. 

Note that both $F$ and $\mathcal H^+(S)$ are embedded 
in $\mathbb H_{\kappa}^{n+1}$, because they are simply connected. By choosing a collar $F$ of small enough width 
we can assume that their union $F\cup \mathcal H^+(S)$ remains embedded in  $\mathbb H_{\kappa}^{n+1}$. We actually only need 
this union to be immersed in $\mathbb H_{\kappa}^{n+1}$.  Then, the Riemannian metrics on $F$ and $\mathcal H^+(S)$
are restrictions of the metric induced on their union by the embedding  
$F\cup_S \mathcal H^+(S)\hookrightarrow \mathbb H_{\kappa}^{n+1}$. 

We define a tensor on $\widehat{N_i(\varepsilon)}$ to be the 1-QC Riemannian metric on points of  
$N_i(\varepsilon)$ and the Riemannian metric on $F\cup_S \mathcal H^+(S)$ which is induced by its embedding in $\mathbb H_{\kappa}^{n+1}$. 
Since these two metrics coincide on $F$ this tensor is a Riemannian metric on $\widehat{N_i(\varepsilon)}$.

The regularity of the metric depends on the regularity of the embedding $F\cup_S \mathcal H^+(S)\hookrightarrow \mathbb H_{\kappa}^{n+1}$. Thus the metric is $\mathcal C^{\infty}$ everywhere except possibly at the joint spheres 
$\partial N_i(\varepsilon)$ where the metric might be only  of class $\mathcal C^1$, because  
$F\cup_S  \mathcal H^+(S)$ is only a $\mathcal C^1$-submanifold of $\mathbb H_{\kappa}^{n+1}$, in general. 
However, an arbitrarily small perturbation of the embedding (or the metric) 
by an infinitely flat function along a small collar of the boundary sphere $S$ will make it smooth. Recall that an infinitely flat  
function is a smooth real function which vanishes on $(-\infty, 0]$ and equals $1$ on $[\delta, \infty)$.

Further, we need to introduce another metric invariant of a 
Riemannian manifold. Recall from \cite{Nik} that the 
{\em pointwise anisotropy}  of the metric is the quantity: 
\[I(p)=\sup_{\sigma} \left|K_p(\sigma)- \frac{R(p)}{n(n-1)}\right|,\] 
where $K_p(\sigma)$ is the curvature of the plane $\sigma$ at $p$ 
and $R(p)$ is the scalar curvature at $p$. Furthermore, the 
{\em integral anisotropy} of the manifold $N$ is  the number 
\[\int_NI(p)\;  d\;{\rm vol}\]

\begin{lemma}\label{niko}
Assume $H(S) >0$. 
For $n\geq 3$  there exist two constants $C_1,C_2$ such that for any $\delta >0$ 
and for any for small enough $\varepsilon$ we have:
\begin{enumerate}
\item ${\rm vol}(\widehat{N_i(\varepsilon)}) \geq C_1 >0$;    
\item ${\rm diam}^2(\widehat{N_i(\varepsilon)})\sup_{p\in \widehat{N_i(\varepsilon)}, \sigma}|K_p(\sigma)| < C_2$; 
\item and the integral isotropy of $\widehat{N_i(\varepsilon)}$ is smaller than $\delta$.   
\end{enumerate}
\end{lemma}

The proof of Lemma \ref{niko} is deferred to section \ref{proofniko}. 

The last ingredient needed is the following result of Nikolaev 
(see \cite{Nik}):
\begin{proposition}\label{Nikolaev}
 Given 
 any constants $C_1,C_2$ there exists 
some $\delta(C_1,C_2,n) >0$ such that 
any closed $n$-manifold $N$ satisfying
${\rm vol}(N) \geq C_1 >0$,  
${\rm diam}^2(N)\sup_{p\in N, \sigma}|K_p(\sigma)| < C_2$,  
whose integral isotropy is smaller than $\delta$  
is diffeomorphic to a space form. 
\end{proposition}

\vspace{0.2cm}
{\em End of proof of Proposition \ref{isotrop}.}
Recall that  for every $\nu>0$ there exists  $\epsilon(\nu) >0$  such that  
$|N(p)-H(p)| \leq  \nu $, for any $p\in N_i(\varepsilon)$. 
Let  $\nu>0$ be such that  $\nu \cdot {\rm vol}(M) < \delta(C_1,C_2)$. We choose $\varepsilon < \varepsilon(\nu)$.  
Then, the Riemannian  manifold  $\widehat{N_i(\varepsilon)}$ satisfies all conditions of Proposition \ref{Nikolaev}. 
Thus $\widehat{N_i(\varepsilon)}$ is diffeomorphic to a 
space form and our claim follows.

\subsection{Proof of Lemma \ref{cap}}\label{proofcap}
The image of the collar $F$ in $\mathbb H_{\kappa}^{n+1}$ by the immersion is a canal hypersurface. 
This was already proved in \cite{Pi,GM} for the case where $\kappa=0$. 

Assume now that $\kappa >0$. 
We want to show that each curvature leaf $S$ the hypersurface $F$ is tangent to a fixed 
round hypersphere $\mathcal H(S)$. 
Let $\psi$ be the unitary normal vector field on the image of 
$F$ in $\mathbb H_{\kappa}^{n+1}$ and $\overline{\nabla}$ be the 
Levi-Civita connection on $\mathbb H_{\kappa}^{n+1}$. 
The Weingarten (or shape) operator $L_{\psi}$ is defined as 
$L_{\psi}X=(\overline{\nabla}_X\psi)^{T}$, where the superscript $^T$ means 
the tangent part. Moreover, $g(L_{\psi}X,Y)=h(X,Y)$, where $X,Y$ are tangent to $M$ and 
$h$ is the second fundamental form of $M$. It follows that 
$L_{\psi}X=-\sqrt{H+\kappa}\cdot X$ 
for any vector field $X$ tangent to the curvature leaf $S$.

\begin{lemma}
The image of a curvature leaf $S$ with $H(S) >0$  by the immersion into $\mathbb H_{\kappa}^{n+1}$ 
is a round sphere. Furthermore, either the focal locus of the image of $F$ is empty or else there is a point 
$q\in \mathbb H^{n+1}_{\kappa}$ such that a hypersphere
$\mathcal H(S)$ with center $q$ contains $S$. 
\end{lemma} 
\begin{proof}
This seems to be widely known (see \cite{GM} for $\kappa=0$). 
The first part follows from the case $\kappa=0$, as balls in 
$\mathbb H^{n+1}_{\kappa}$ are conformally equivalent 
to balls in the Euclidean space and curvatures leaves are sent into rounds spheres. 

An alternate proof goes as follows. According to \cite{CR} the focal locus of $F$ is diffeomorphic 
to an interval, because  the focal points have multiplicity $n-1$.  
An endpoint of the focal locus of $F$ 
is equidistant to the points of $S$, since the derivative of the focal map (\cite{CR}, section 1.e, 
Thm. 2.1 and proof of Thm. 3.1)) is trivial. This shows that $S$ is contained in a 
hypersphere $\mathcal H(S)$.

Still another proof of this statement 
could be found in \cite{IPS} in a slightly different wording. 
\end{proof}

In order to compute the radius of the hypersphere $\mathcal H(S)$,  we
recall that  for a unit speed geodesic $\gamma$ parameterized by 
$[0,l]$ issued from  a point $p$ and 
normal to the submanifold $W$ of $\mathbb H_{\kappa}^{n+1}$ 
the Jacobi field $Y(t)$ along $\gamma$ is a $W$-Jacobi field 
if $Y(0)$ is tangent to $W$ at $p$, $Y(t)$ is orthogonal to $\dot{\gamma}(t)$ 
for all $t$, and 
$\nabla_{\dot{\gamma}(0)} Y(0) + L_{\dot{\gamma}(0)}Y(0)$ is orthogonal 
to $W$, where $L_{\dot{\gamma}(0)}$ is the Weingarten (or shape) 
operator of $W$. The point $q$ on $\gamma$ is called a focal 
point of $W$ along $\gamma$ if there exists a non-trivial 
$W$-Jacobi field along $\gamma$ vanishing at $q$. 

If $X_1,X_2,\ldots,X_n$ is a set of parallel vector fields along $\gamma$ 
which form along with $\dot{\gamma}$ a basis of the ambient  
tangent space to $\mathbb H_{\kappa}^{n+1}$, then   
a Jacobi field on $\mathbb H_{\kappa}^{n+1}$ along $\gamma$ 
has the form (see \cite{CE}): 
\[ Y(t)=\sum_{i=1}^{n} (a_i \sinh \sqrt{\kappa} t+ 
b_i \cosh \sqrt{\kappa} t) X_i(t),\]
where $a_i,b_i\in\mathbb R$.  

An immediate computation 
yields the fact that focal points of $S$ along 
geodesics pointing in the direction of 
$\psi$  exist  if and only if $H(S)>0$, in which case
they are all at the same distance: 
\[d=\frac{1}{\sqrt{\kappa}}{\rm arctanh} \sqrt{\frac{\kappa}{H+\kappa}}\]

Moreover $\mathcal H(S)$ and $F$ are tangent along $S$ since  they have the same 
normal vector field $\psi$ at the points of $S$. 

Eventually, $S$ bounds two balls in $\mathcal H(S)$ and one of them is a spherical cap 
$\mathcal H^+(S)$ with the property $F\cup_S\mathcal H^+(S)$  
is $\mathcal C^1$.  This proves Lemma \ref{cap}.

\begin{remark}
We have results of similar nature when $H(S) \leq 0$. The focal locus of  $F$ is empty as 
the focal map is only defined when $\left|\frac{\kappa}{H(S)+\kappa}\right| <1$.
Further, the normal geodesics in the direction 
$\pm \psi$ diverge. Nevertheless there exists a totally umbilical embedded  
$n$-plane of  constant negative  intrinsic curvature $H(S)$ 
which plays this time the role of the round sphere $\mathcal H(S)$. 
However, now the cap $\mathcal H^+(S)$ is unbounded.
These are called {\em equidistant planes} (or hypersurfaces) and were 
first considered in \cite{CR1}. In fact they are the locus of points at 
a given distance  $r$ from a hyperbolic hyperplane 
of $\mathbb H^{n+1}_{\kappa}$, or a horosphere respectively.  
For instance, when $H(S)=0$, $\mathcal H(S)$ is a 
horosphere centered at the point at 
infinity where all normal geodesics to $S$ abut. 
Further, when $H(S) <0$  we have an equidistant plane  for the distance 
$r=  {\rm arcosh}\left(\frac{1}{\sqrt{|H(S)|}}\right)$.
\end{remark}

\subsection{Proof of Lemma \ref{niko} }\label{proofniko}
The argument used in the proof of Lemma \ref{compact} shows that 
the measure of the spherical cap $\mathcal H^+(S)$ is at least that of the hyperbolic ball 
$B_{\lambda}$ lying in a $n$-plane and having the same boundary sphere $S$, and thus uniformly 
bounded from below by  the positive constant $C_1(\lambda_{\max}(\varepsilon),\kappa, n)$.
Therefore, the volume of $\widehat{N_i(\varepsilon)}$ 
is uniformly bounded from below by $C_1(\lambda_{\max}(\varepsilon), \kappa, n) >0$.

Each boundary component of $N_i(\varepsilon)$ is either {\em separating} $M$ in two disjoint components or 
{\em non-separating}. The separating spheres are of two types: either one separated component is  
simply connected, or else none of the components are simply connected. 
The boundary component is accordingly called trivial and respectively non-trivial.  

Each non-separating or separating non-trivial boundary component induces a non-trivial splitting of the 
fundamental group $\pi_1(M)$ as a free amalgamated product. It follows that their total number is bounded from above by the 
maximal number $r(\pi_1(M))$ of factors in a free splitting of the 
fundamental group $\pi_1(M)$. Further, according to Remark \ref{separating}, the  number of trivial separating  boundary components  is uniformly bounded by $\frac{{\rm vol}(M)}{C_1(\lambda_{\max}(\varepsilon),\kappa,n)}$. 
Moreover, each spherical cap $\mathcal H^+(S)$ lies on a hypersphere of radius $d$, where $d$ is given in Lemma \ref{cap}.  
It follows that the diameter of $\widehat{N_i(\varepsilon)}$ is  uniformly bounded  from above by 
\[ C_2=C_2({\rm vol}(M),\lambda,H, \kappa,n)= 
r(\pi_1M)\cdot {\rm diam}(M) + 2\pi d\cdot \frac{{\rm vol}(M)}{C_1(\lambda_{\max}(\varepsilon),\kappa,n)},\]  
while the volume $\widehat{N_i(\varepsilon)}$ is at least  
$C_1=C_1(\lambda_{\max}(\varepsilon), \kappa, n)$. Observe that $d$ is uniformly bounded 
in terms of $|\inf_{p\in M}H|$. 

Now hyperspheres $\mathcal H(S)$  of radius $d$ in $\mathbb H^{n+1}_{\kappa}$ have 
constant sectional curvature. As they are totally umbilical, their principal curvatures 
are all equal to $\frac{\sqrt{\kappa}}{{\rm tanh}(\sqrt{\kappa}d)}$. It follows that the sectional curvature 
of the hypersphere $\mathcal H(S)$ is equal to $H(S)$. By choosing the 
smoothing collar $F$ of $S$  of small enough width we obtain:   
\[\sup_{p\in \widehat{N_i(\varepsilon)}, \sigma}|K_p(\sigma)|
\leq
\max 
(\sup_{p\in M}|H|,\sup_{p\in M}|N|)\]

On the other hand the pointwise anisotropy $I(p)$ is non-zero 
only on the subset $N_i(\varepsilon)$, as spherical caps have 
constant curvature. Thus the integral anisotropy 
of $\widehat{N_i(\varepsilon)}$ is 
bounded by above by $\nu\int_M |H| d{\rm vol}$. 
The claim of Lemma \ref{niko} follows.

\begin{remark}
We don't know whether  pieces $N_i(\varepsilon)$ 
for small enough $\varepsilon$ can be assigned a well-defined 
sign according to the curvature at their  
isotropic points. This would permit to use 
Gromov's pointwise pinching result in negative curvature (see \cite{Gro})
under bounded volume assumptions.
This would be so if locally the dimension of $C$ were at least 3, 
by F. Schur (\cite{Sch}). 
\end{remark}

\begin{remark}
Isotropic points where  the sectional curvature is negative can only be approached by leaves $S$ with $H(S) <0$ 
and the method considered above cannot provide an isometric 
capping off for the boundary components. 
For instance, if some negatively curved Gromov-Thurston 
manifold (see \cite{GT})  admits a spine which can be realized by 
points of constant sectional curvature, then we could provide 
examples of $1$-QC manifolds which are not obtained by gluing together space forms. 
Recall that a spine for a closed manifold $M$ is a polyhedron $P\subset M$ with the property that 
$M-D^n$ collapses onto $P$, where $D^n\subset M$ denotes an $n$-ball. A typical spine could be geometrically 
constructed using the cut locus of nice metrics. 
\end{remark}

\section{More about the equivariant immersions of $1$-QC manifolds}
Throughout this section $M$ is a closed orientable $n$-manifold endowed with a $1$-QC Riemannian metric, which is assumed 
conformally flat when $n=3$.  
\begin{proposition}\label{immerse3}
Suppose that $n=3$ and $\kappa$ is such that $H+\kappa >0$ on $M$.
\begin{enumerate}
\item If $N+\kappa\neq 0$ then there exists an unique isometric immersion 
$f:\widetilde{M}\to \mathbb H_{\kappa}^{n+1}$, up  
to an isometry of $\mathbb H_{\kappa}^{n+1}$.
\item Assume that $N+\kappa=0$. Then the set of such 
equivariant isometric immersions of a component $E$ of $M-C$ 
is in bijection with the set of paths in the space of complex structures 
over the surface $L$ corresponding to a curvature leaf.
\end{enumerate} 
\end{proposition}
\begin{proof}
If $\{X_1,X_2,X_3=\xi\}$  is a local orthonormal basis and $h_{ij}$ are components of the second fundamental form of the immersion then Gauss equations read:
\[ h_{11}h_{22}-h_{12}^2=\mu, 
 h_{11}h_{33}-h_{13}^2=\nu,  h_{22}h_{33}-h_{23}^2=\nu\]
\[ h_{23}h_{11}=h_{12}h_{13}, h_{13}h_{11}=h_{12}h_{23},
h_{12}h_{33}=h_{13}h_{23}\]
where $\mu=H+\kappa$, $\nu=N+\kappa$.
The homogeneous equations above imply that 
either $h_{11}=h_{22}$ or else $\nu(h_{11}+h_{22})=\mu h_{33}$.
The first alternative leads to the immersion arising from the 
form $h$ in Proposition \ref{immerse}, as happened for $n=4$.
The second one coupled with the non-homogeneous equations 
leads to $\nu=0$. In this case we obtain the solution consisting of 
$h_{ij}$ with  
\[ h_{33}=h_{13}=h_{23}=0\]
while the other components are satisfying:  
\[h_{11}h_{22}-h_{12}^2=\mu\]
It follows that  the restriction of $h$ to each lift $\pi^{-1}(L)$ 
of a curvature leaf $L\subset E$ (of constant curvature $H_0$) 
in $\widetilde{M}$  is the second fundamental form of 
an immersion $f_{\lambda_0}:\pi^{-1}(L)\to  \mathbb H_{\kappa}^n$. 
This immersion $f_{\lambda_0}$ is a slice of the 
immersion $f:\widetilde{M}\to \mathbb H_{\kappa}^{n+1}$,  
with respect to the foliation of $\mathbb H_{\kappa}^{n+1}$ by 
hyperbolic hyperplanes $\mathbb H_{\kappa}^n$. 

In (\cite{Lab}, Prop. 2.5) one identified 
the 2-tensors verifying the Gauss and Codazzi equations for 
a constant curvature metric on a surface with the complex structures on the surface. Therefore, the complex structures of the slices of $E$  give 
us a path of complex structures on the slice surface. 
Recall that the sectional curvature of the slice surface 
is the not necessarily constant function $\lambda$ on $E$, 
so that we need to rescale  it in every slice in order to 
fit exactly into the framework of (\cite{Lab}, Prop 2.5). 
\end{proof}

Observe that the immersion $f:\widetilde{M}\to \mathbb H_{\kappa}^{n+1}$ 
is covered by an equivariant immersion into the (unit) tangent bundle  
$\hat{f}:\widetilde{M}\to T\mathbb H_{\kappa}^{n+1}$, defined by 
\[ \hat{f}(p)=(f(p), n_{f(p)})\in  T\mathbb H_{\kappa}^{n+1}\]
where $n_{f(p)}$ denotes the unit (positive) normal vector at 
$f(\widetilde{M})$ at the point $f(p)$ associated to the local sheet 
defined around $p$. 

Let us further consider the hyperbolic Gauss map $G:T\mathbb H_{\kappa}^{n+1}\to S^n$ 
which sends the pair $(p,v)$, where  $p\in \mathbb H_{\kappa}^{n+1}, 
v\in T_p\mathbb H_{\kappa}^{n+1}$ to the point at infinity to which the 
geodesic issued from $p$ in direction $v$ will reach the boundary.  
Here $S^n$ is identified with the boundary at infinity  
$\partial\mathbb H_{\kappa}^{n+1}$ of the hyperbolic space.

\begin{proposition}\label{conf}
Suppose that  $H+\kappa >0$, $N+\kappa>0$ and $n\geq 3$.
The composition $D_{\infty}=G\circ \hat{f}:\widetilde{M}\to S^n$ is an immersion.
Moreover the flat conformal structure defined by the $1$-QC metric on $M$ 
is the one given by the developing map $F$ and the holonomy homomorphism $\rho$.
\end{proposition}
\begin{proof}
If  $H+\kappa >0$ and $N+\kappa >0$, then $f(\widetilde{M})$ is infinitesimally convex, namely the eigenvalues 
of its second fundamental form  are positive. According to a well-known result of Bishop (see \cite{Bi}) 
$f(\widetilde{M})$ is locally convex.  This implies that $D_{\infty}$ is an immersion. 
Moreover, by the arguments above  the immersion $F$ is equivariant 
with respect to the homomorphism $\rho:\pi_1(M)\to SO(n+1,1)$, this time  
interpreting $SO(n+1,1)$ as the group of M\"obius (conformal) transformations of the sphere. 
It follows that the conformally flat structure on $M$ is 
actually given by the developing map $D_{\infty}$ and the holonomy homomorphism 
$\rho$.
\end{proof}

\begin{proposition}\label{hyperbolicmetric}
To any closed $n$-manifold $M$, $n\geq 3$, endowed with a 1-QC Riemannian metric we can associate  
a hyperbolic metric on $M\times [0,\infty)$ which restricts to (a homothetic of) 
the initial $1$-QC metric on the boundary. 
\end{proposition}
\begin{proof}
Propositions \ref{immerse} and \ref{unique-immerse} for $n\geq 4$ and Proposition \ref{immerse3} for $n=3$ show that for large enough $\kappa$ there exists a unique equivariant immersion $f:\widetilde{M}\to \mathbb H_{\kappa}^{n+1}$. 
We can construct under the conditions of Proposition \ref{conf} an equivariant immersion 
$\phi:\widetilde{M}\times [0,\infty) \to \mathbb H_{\kappa}^{n+1}$
by setting: 
\[ \phi(p, t)=\exp(t n_{f(p)}).\]
This is well-defined for all $t\in [0, \infty)$, 
and it is an immersion, since $f(\widetilde{M})$ is locally convex. 
In particular, we have a constant 
curvature metric $\phi^*g_{\mathbb H_{\kappa}^{n+1}}$ induced by this immersion 
as pull-back of the usual Riemannian metric  $g_{\mathbb H_{\kappa}^{n+1}}$ on 
$\mathbb H_{\kappa}^{n+1}$. Since $\phi$ is equivariant the metric descends to 
$M\times [0,\infty)$ and it becomes of constant curvature $-1$ after rescaling. 
In order to apply Proposition \ref{conf} we have to choose 
$\kappa$ large enough. However, a different $\kappa$ will lead to the same 
hyperbolic metric. 
\end{proof}

\begin{proposition}\label{completemetric}
The hyperbolic metric $\phi^*g_{\mathbb H_{\kappa}^{n+1}}$ on $M\times [0,\infty)$ is complete 
if the map $\phi: \widetilde{M}\times [0,\infty) \to \mathbb H_{\kappa}^{n+1}$ is proper. 
In particular,  this is true if $2N < H \leq 0$.  
\end{proposition}
\begin{proof} 
Observe that $M\times [0,\infty)$ has a complete product metric constructed from the 1-QC Riemannian metric on the factor $M$ and the standard metric on $[0,\infty)$.  The  pull-back metric $\phi^*g_{\mathbb H_{\kappa}^{n+1}}$  on $M\times [0,\infty)$ is complete if and only if every countable bounded subset in this metric is also bounded in the product metric. 

Suppose that  the pull-back metric is not complete. Thus, there exists a countable set $A\subset M\times [0,\infty)$ 
which is unbounded in the product metric - i.e. its projection on $[0,\infty)$ is unbounded - while $A$ is bounded in the hyperbolic metric. Then, a lift  $\widetilde{A}$ of $A$ to $\widetilde{M}\times [0,\infty)$ is also bounded with respect 
to the lift of the hyperbolic metric. Since $\phi$ is an isometric immersion, it  contracts distances and hence 
$\phi(\widetilde{A})$ is contained in a compact metric ball $B$ in 
$\mathbb H_{\kappa}^{n+1}$. Therefore, $\phi^{-1}(B)$ is not compact 
in $\widetilde{M}\times [0,\infty)$ endowed with its usual topology, because it contains the unbounded subset 
$\widetilde{A}$ and hence $\phi$ is not proper. This settles the first claim. 

Further, Alexander  proved in (\cite{Al}, Prop.2) that an isometric hypersurface  immersion of a complete, connected 
Riemannian manifold into a simply connected complete manifold of sectional curvature at most $-\kappa$ having 
the absolute value of the eigenvalues of its second fundamental form bounded by $\sqrt{\kappa}$ must be 
an embedding. Recall that the eigenvalues of $II_{f(\widetilde{M})}$ are $\sqrt{H+\kappa}$ and $\frac{N+\kappa}{\sqrt{H+\kappa}}$. Then the immersion $f$ satisfies the hypothesis of Alexander's theorem above if 
$2N < H \leq 0$. In this case $f$ is an embedding and $f(\widetilde{M})$ is convex. It follows that 
$\phi$ is a proper embedding. 
\end{proof}

\begin{remark}
If the holonomy group $\rho(\pi_1(M))\subset SO(n+1,1)$ is not discrete, then the immersion $f$ is not proper and 
hence $\phi$ is not proper either.  
\end{remark}

\begin{remark}
A construction due to Thurston in the case $n=2$ (see \cite{KT}), further extended to all dimensions by 
Apanasov (\cite{A}), Kulkarni and Pinkall (\cite{KP}) provides a canonical $\mathcal C^{1,1}$ Riemannian 
metric on $M$ and a canonical totally geodesic stratification, using the convex hull construction on the limit set of $\rho(\pi_1(M))$. They defined  a continuous 
$\rho$-equivariant map $D_0: \widetilde{M}\to \mathbb H^{n+1}_{\kappa}$ with locally convex image (under the assumption that $S^n-D_{\infty}(\widetilde{M})$ contains at least two points).  
Specifically,  we set  $D_0(p)=q$, for $p\in \widetilde{M}$, where $q$ is the unique closest point 
to $D_{\infty}(p)$ which belongs to the intersection of all half-spaces in $\mathbb H^{n+1}_{\kappa}$ 
determined by  the round  spherical balls $S^n-B$, where $D_{\infty}(p)\subset B \subset D_{\infty}(\widetilde{M})$. 
There is then an equivariant stratification of $\widetilde{M}$, such that $D_0$ maps each stratum of $\widetilde{M}$ 
isometrically onto some pleat in $\mathbb H^{n+1}_{\kappa}$, namely some $k$-dimensional 
convex hull of a subset in $S^n$, for $1\leq k \leq n+1$. 

We believe that 1-QC manifolds have only strata of dimensions $n$ and $n-1$, namely they are channel 
M\"obius manifolds in the terminology of \cite{KP}.  This will provide an action of $\pi_1(M)$ on a $\R$-tree. 

We can extend $D_0$ and $D_{\infty}$ to an equivariant $\mathcal C^1$-immersion $D:\widetilde{M}\times (0,\infty)\to \mathbb H_{\kappa}^{n+1}$ by sending $(p,t)$ into the point sitting  at hyperbolic distance $t$ from $D_0(p)$ on the geodesic ray joining $D_0(p)$ to 
$D_{\infty}(p)$. Each slice $D(*, t)$, $t>0$, is a locally convex $\mathcal C^1$-immersion. 
Then the pull-back $D^*g_{\mathbb H^{n+1}_{\kappa}}$  gives also a hyperbolic (after rescaling) metric on 
$M\times (0,\infty)$. The relationship between the two hyperbolic structures induced by $D$ and $\phi$ 
respectively has still to be clarified. 
\end{remark}

\begin{remark}
When $n=2$ the situation is different, and it was described by Labourie
in \cite{Lab}.  The  space $\mathcal I(k_0)$ 
of codimension one isometric equivariant isometries of a 
surface of a surface $S$  of genus at least $2$  
into the hyperbolic space $\mathbb H_1^{3}$, such that 
the pull-back metric is of constant curvature $k_0\in(-1,0)$  -- 
up to left composition by isometries of 
$\mathbb H_1^{3}$ and right composition by lifts of diffeomorphisms 
isotopic to identity -- is not one point as for $n\geq 4$.  
In fact, there is a natural map of $\mathcal I(k_0)$ into 
the Teichm\"uller space $\mathcal T(S)$ which associates 
to an immersion the conformal class $\iota_0$ of the induced 
constant curvature metric. The fiber $\mathcal I(k_0,\iota_0)$ of this map 
can also be mapped homeomorphically onto the Teichm\"uller space 
$\mathcal T(S)$ by sending the class of an immersion 
to the  conformal class of its second (or third) fundamental form  
or, equivalently, of its associated complex structure. Then 
the map from  $\mathcal I(k_0)$  to the space of $\mathbb CP^1$ (i.e. 
conformally flat) structures on $S$ which associates to an immersion 
$f$ the $\mathbb CP^1$-structure given by $F$ is also a homeomorphism. 
This gives a one parameter family of parameterizations of the space of  
 $\mathbb CP^1$ structures by $\mathcal T(S)\times \mathcal T(S)$. 
  
Thus, for each $\mathbb CP^1$-structure on $S$ and $k\in (-1,0)$ 
there exists some metric of constant curvature $k$ on $S$ and an 
equivariant isometric embedding $f_k:S\to \mathbb H_1^{3}$, whose associated 
$\mathbb CP^1$ structure is the one with which we started. 
Further, the immersion $\phi_k$ associated to $f_k$ provides 
a hyperbolic structure $\mathcal M(k)$ on 
$S\times [0,\infty)$, called a geometrically finite end (see \cite{Lab}).
In \cite{Lab,Lab2} one proved that for each constant $k\in [k_0,0)$ there exists a unique incompressible embedded surface $S_{k}\subset S\times [0,\infty)$ 
of constant curvature $k$  homeomorphic to $S$ and the family   $S_{k}$ foliates 
the geometrically finite end $\mathcal M(k_0)$.  
This proves that actually $\mathcal M(k_1) \subset \mathcal M(k_2)$ 
if $k_1 > k_2$. When $k$ converges to $-1$ then 
$\mathcal M(k)$ converges to a geometrically finite end. 
The $\mathbb CP^1$-structure associated to this geometrically finite end is 
also the one from the beginning. Moreover, the finite 
boundary of $\mathcal M(0)=\cup_{k}\mathcal M(k)$ is isometric 
to a surface of constant curvature $-1$ but its embedding 
into $\mathbb H_1^3$ is a pleated surface 
along a geodesic lamination, according to Thurston (see \cite{Lab} for 
more details).  Higher dimensional versions  of Labourie's results 
concerning moduli of hyperbolic flat conformal were obtained 
by Smith in \cite{Smith}.  
\end{remark}

\bibliographystyle{plain}

\end{document}